\documentclass[reqno]{amsart}

\usepackage{mathabx}
\usepackage{amsfonts}
 \usepackage{mathrsfs}               

\usepackage{amssymb,amsmath,amsfonts,bm}
\usepackage{dsfont}
\usepackage{epsfig}
\usepackage{graphicx}
\usepackage{enumerate}
\usepackage{enumitem}
\usepackage{verbatim}
\usepackage{comment}
\usepackage{color}
\usepackage{hyperref}
\usepackage{caption}
\usepackage{ulem}
\usepackage[margin=2.7 cm]{geometry}
\usepackage{geometry}
\usepackage{tikz}  
\usetikzlibrary{arrows}
\usepackage{rotating}
\usetikzlibrary{shapes,quotes,calc,arrows,through,intersections,shadings}
\usepackage{tkz-euclide}
\tikzstyle{int}=[draw, fill=blue!20, minimum size=2em]
\tikzstyle{init} = [pin edge={to-,thin,black}]
\DeclareMathSizes{10}{9}{6}{5}


\date{} 

 \DeclareMathOperator\jac{Jac}

\DeclareMathOperator\card{card}

\theoremstyle{plain}
\newtheorem{theorem}{Theorem}[section]
\newtheorem{definition}[theorem]{Definition}
\newtheorem{proposition}[theorem]{Proposition}

\newtheorem{lemma}[theorem]{Lemma}
\newtheorem{corollary}[theorem]{Corollary}
\newtheorem{remark}[theorem]{Remark}

\newcommand{\D}{ {\mathcal{D}_{m,\bm{\epsilon}_M} } }
\newcommand{\Dsig}{ {\mathcal{D}_{m,\bm{\sigma}_M} } }
\newcommand{\DN}{ {\mathcal{D}_{N,\bm{\epsilon}_M} } }

\newcommand{\R}{\mathbb{R}}
\newcommand{\N}{\mathbb{N}}
\newcommand{\E}{\mathbb{E}^{\ell d -1}_1}

\newcommand{\ind}{\mathds{1}}

\numberwithin{theorem}{section}
\numberwithin{equation}{section}
\numberwithin{figure}{section}

\let\oldtocsection=\tocsection
 
\let\oldtocsubsection=\tocsubsection
 
\let\oldtocsubsubsection=\tocsubsubsection
 
\renewcommand{\tocsection}[2]{\hspace{0em}\oldtocsection{#1}{#2}}
\renewcommand{\tocsubsection}[2]{\hspace{1em}\oldtocsubsection{#1}{#2}}
\renewcommand{\tocsubsubsection}[2]{\hspace{2em}\oldtocsubsubsection{#1}{#2}}

\begin{document}

\parskip=1pt

\vspace*{1.5cm}
\title[Derivation of the Higher Order Boltzmann Equation for Hard Spheres ]
{Derivation of the Higher Order Boltzmann Equation for Hard Spheres}
\author[Ioakeim Ampatzoglou]{Ioakeim Ampatzoglou}
\address{Ioakeim Ampatzoglou,  
Courant Institute of Mathematical Sciences, New York University.}
\email{ioakampa@cims.nyu.edu}

\author[Nata\v{s}a Pavlovi\'{c}]{Nata\v{s}a Pavlovi\'{c}}
\address{Nata\v{s}a Pavlovi\'{c},  
Department of Mathematics, The University of Texas at Austin.}
\email{natasa@math.utexas.edu}

\author[William Warner]{William Warner}
\address{William Warner,
Department of Mathematics, The University of Texas at Austin.}
\email{billy.warner@utexas.edu}

\vspace*{-3cm}
\begin{abstract}
In this paper we complete the program initiated in \cite{ternary} and rigorously derive a Boltzmann-type equation that incorporates  higher order collisions among gas particles. More precisely,  starting from a finite $N$-particle system where the particles can perform symmetric hard sphere type interactions up to arbitrarily high order $M$, we derive   a kinetic equation which consists of a linear combination of  higher order collisional terms. We identify the new scaling regime that facilitates such collisions and systematically generalize geometric techniques allowing us to analyze the correlation of potentially distinct order recollisions as time evolves. 

\end{abstract}
\maketitle

\tableofcontents
\section{Introduction}

The Boltzmann equation, introduced by L. Boltzmann \cite{Boltzmann 1872} and J. Maxwell \cite{Maxwell 1867}, is a kinetic integro-differential equation describing the dynamics of a rarefied monatomic gas. For dimension $d\geq 2$, and a probability density function $f:[0,\infty) \times \R^d \times \R^d \mapsto \R$, representing the probability of finding a particle with position $x\in \R^d$ and velocity $v\in \R^d$ at time $t\geq 0$, the Boltzmann equation is given by
\begin{equation}\label{binary boltz}
    \partial_t f + v \cdot \nabla_x f = Q_2(f,f),
\end{equation}
where $Q_2$ is the integral collision operator written as
\begin{equation}
    Q_2(f,f) = \int_{\mathbb{S}_1^{d-1}\times \R^d} (f^* f^*_1 - f f_1) B(v- v_1, \omega) d\omega dv_1,
\end{equation}
where
\begin{equation*}
    B(v- v_1, \omega) = |v-v_1|^\gamma \tilde{b}(\langle \omega , \frac{v - v_1}{|v- v_1|} \rangle), \quad \gamma \in \R,
\end{equation*}
with $\tilde{b}$ an an even function
and
\begin{equation*}
    f^* = f(t,x,v^*), \quad f_1^* = f(t,x,v_1^*), \quad f = f(t,x,v), \quad f_1 = f(t,x,v_1),
\end{equation*}
where
\begin{align}
    v^* = v + \langle \omega, v_1 - v \rangle \omega,\\
    v_1^* = v_1 - \langle \omega, v_1 - v \rangle \omega.
\end{align}
We note that $Q_2(f,f)$ is quadratic in $f$, and thus only captures binary interactions of gas molecules in its underlying model, neglecting any possible higher order interactions that may occur at a given instant. For this reason, the Boltzmann equation given in \eqref{binary boltz} provides an accurate model for very dilute gases. However, higher order collisional terms are often needed for analyzing the dynamics of denser gases, see e.g. \cite{Dobnikar.1,Dobnikar.2,Hynninen,Russ} for details on the importance of considering ternary or higher order interactions in colloidal gases. The goal of this paper is to provide a rigorous derivation of a Boltzmann-like model that incorporates higher order collisions among gas particles. 
\par
More precisely, in this paper we answer the question on a rigorous derivation of a higher order Boltzmann equation (posed in the first paper \cite{ternary} of a sequence of works \cite{ternary,AmpatzoglouThesis,AmpatzoglouPavlovic2020, amgapata22,amgapata22b}) by deriving the following higher order Boltzmann-type equation from a finite $N$-particle system of hard spheres: 
\begin{equation}\label{general boltz}
    \partial_t f + v \cdot \nabla_x f = \sum_{\ell=1}^{M} \frac{1}{\ell !} Q_{\ell+1}(\underbrace{f,\cdots, f}_{\ell+1}), \quad (t,x,v) \in (0,\infty) \times \R^d \times \R^d,
\end{equation}
with $M+1$ denoting the highest order interaction considered for $M \in \N$.
\par
Recently, it has been shown that versions of \eqref{general boltz} have some interesting analytical properties. The space inhomogeneous binary-ternary Boltzmann equation (\eqref{general boltz} for $M=2)$ was studied in \cite{amgapata22} where global in time well-posedness near vacuum was shown in Maxwellian weighted $L^\infty$-spaces, see also \cite{WeYu24} for global well-posedness in an integrable space.
In \cite{amgapata22b}, generation and propagation of polynomial and exponential moments as well as global well-posedness was shown for the space homogeneous binary-ternary Boltzmann equation. Interestingly, the addition of the ternary correction term improves the time decay of moments in certain situations. In particular, the equation exhibits the generation of moments that corresponds to the part of the kernel with the highest potential rate $\gamma$. A consequence of this is that generation in time of moments is established even if one of the potentials corresponds to the Maxwell molecules case -- a result which was not available prior to the introduction of the ternary correction. This is an indication that addition of ternary collisions could potentially serve as a higher order correction term of the Boltzmann equation, providing important motivation for obtaining a rigorous derivation of a Boltzmann equation \eqref{general boltz}. Evidence of ``better" 
behavior of the higher-order Boltzmann equation \eqref{general boltz} compared to the binary Boltzmann equation \eqref{binary boltz} serves as our main motivation for deriving equation \eqref{general boltz}.
 \par
The problem of deriving the Boltzmann equation itself (i.e. \eqref{general boltz} with $M=1$) from an underlying finite particle model is a difficult problem. The short time derivation of the space inhomogeneous  Boltzmann equation from hard spheres undergoing Newton's laws of motion goes back to the seminal work of Lanford \cite{Lanford} and was later revisited by the  collaboration of Gallagher, Saint-Raymond, and Texier in \cite{Gallagher}, who were able to refine Lanford's original convergence argument. A derivation was also obtained for short range potentials first by King \cite{Ki75}, and then by Gallagher et. al. \cite{Gallagher} as well as Pulvirenti, Saffirio and Simonella in \cite{PuSaSi14}.  See also \cite{bodineau1,bodineau2, bodineau3} for recent exciting works on the derivation of the Boltzmann equation using more probabilistic techniques such as cluster expansions. Very recently, Deng, Hani and Ma \cite{DeHaMa24} provided a derivation of the Boltzmann equation from a finite system of hard spheres evolving from random data. This derivation is valid as long as the Boltzmann equation itself is well-posed in a certain space.
\par 
Regarding derivation of \eqref{general boltz}, recently in \cite{ternary}, the first and second authors of this paper  extended Lanford's program to a gas of particles undergoing exclusively ternary collisions, and subsequently in \cite{AmpatzoglouPavlovic2020} rigorously derived a binary-ternary Boltzmann equation (\eqref{general boltz} for $M=2)$. This was achieved   by considering a finite system that incorporates standard binary elastic collisions between identical particles of diameter $\epsilon_2$ and ternary collisions defined by an asymmetric ternary distance function 
\begin{equation}
    d_3(x_i,x_j,x_k): \R^{d}\times \R^{d} \times \R^d \mapsto \R_{\geq 0},
\end{equation}
which signals an interaction between the $i,j,$ and $k$ particles when,
\begin{equation}
    d_3(x_i,x_j,x_k) = \epsilon_3,
\end{equation}
where the ternary interaction zone $\epsilon_3$ satisfies $0< \epsilon_2 << \epsilon_3 << 1$. The crucial realization for detecting both interaction types simultaneously in the limit $N \rightarrow \infty$ is the implementation of two different scalings for the interaction zones $\epsilon_2$ and $\epsilon_3$ with respect to the total number of particles in the system $N$. Going beyond ternary interactions was left as an open problem. The paper at hand addresses exactly that problem by allowing a linear combination of higher order interaction terms, as on the right hand side of \eqref{general boltz}.

We also mention that a homogeneous version of the higher order Boltzmann equation \eqref{general boltz} was recently derived
starting from a stochastic model in \cite{BIN} by E. C\'ardenas together with the second and third authors of this paper, who  expanded on the fundamental work of Kac \cite{Kac}.

Now we are ready to precisely define what we mean by higher order interactions (referred to as collisions too). Following that, we will describe the main challenges of the current particle interaction framework, outline the derivation process and informally state our main result.

\subsection{Higher order interactions}\label{subsec::higher order coll}
Let $N,M \in \N$ and $\epsilon_2 >0$. Consider $N$ hard spheres of diameter $\epsilon_2$ and $M$ interaction zones $\epsilon_2, \cdots, \epsilon_{M+1}$ satisfying, 
\begin{equation}
    0< \epsilon_2 << \cdots << \epsilon_{M+1} << 1.
\end{equation}
Let $\ell \in \{ 1,\cdots, M \}$. We define the \textbf{symmetric} $\ell+1$-distance between particles $x_1,\cdots, x_{\ell+1}$ via,
\begin{equation}\label{dist}
    d^2_{\ell+1}(x_1,\cdots, x_{\ell+1}) = \sum_{1\leq i < j \leq \ell+1} |x_i - x_j|^2.
\end{equation}
An $\ell+1$-nary interaction occurs among $x_1,\cdots, x_{\ell+1}$ particles when their $\ell+1$ symmetric distance $d_{\ell+1}$ coincides with the $\ell+1$ interaction zone $\epsilon_{\ell+1}$,
\begin{equation}\label{dist int zone}
    d^2_{\ell+1}(x_1,\cdots, x_{\ell+1}) = \epsilon_{\ell+1}^2.
\end{equation}

\subsubsection*{Scaling of $\ell+1$-nary interactions}
One can interpret this type of symmetric interaction as a symmetrization  of $\ell+1$ asymmetric interactions 
in the spirit of \cite{AmpatzoglouPavlovic2020}. Indeed, one can write

\begin{equation}\label{symmetrization}
2d^2_{\ell+1}(x_1,\cdots,x_{\ell+1})= \sum_{j=1}^{\ell+1}d^2_{j,\ell+1}=2\epsilon_{\ell+1}^2,    
\end{equation}
where $d_{j,\ell+1}^2:=\sum_{k=1,\,k\neq j}^{\ell+1}|x_j-x_k|^2$. Hence, for some $j\in\{1,\dots,\ell+1\}$, the quantity $d^2_{j,\ell+1}$ needs to be of order $\epsilon_{
\ell+1}^2$. Assume without loss of generality that $j=1$. As in \cite{AmpatzoglouPavlovic2020}, this can be seen as a hard sphere interaction in $\mathbb{R}^{\ell d}$ between the $\ell d$-particles with positions 
$$\bm{x}=(x_1,\dots,x_1),\quad \bm{y}=(x_2,\dots, x_{\ell+1}).$$ Then the particle $\bm{x}$ would span a volume of order $\epsilon^{\ell d-1}$ in a unit of time, and there are $\binom{N}{\ell}$ options for the particle $\bm{y}$. As $N$ is large, we obtain the scaling $N^\ell \epsilon^{\ell d-1}=O(1)$, or equivalently
\begin{equation}\label{scaling intro}
   N\epsilon^{d-\frac{1}{\ell}}=O(1). 
\end{equation}
As we will see this is the scaling that each $\ell+1$-nary interaction will obey.

\subsubsection*{Collisional law}     
When a $\ell+1$-nary interaction occurs, we assume that the velocities of the particles transform according to the the collisional law $$(v_1,\cdots, v_{\ell+1}) \mapsto (v_1^*,\cdots, v_{\ell+1}^*),$$  given by
\begin{align}\label{coll law}
    v_i^* = v_i + {C}(X_{\ell+1}, V_{\ell+1}) \sum_{j = 1}^{\ell+1} \frac{x_j - x_i}{\epsilon_{\ell+1}}, \quad \text{for all } i = 1,\cdots, \ell+1,
\end{align}
where we denote $X_{\ell+1} = (x_1,\cdots, x_{\ell+1})$, $V_{\ell+1} = (v_1,\cdots, v_{\ell+1})$, and 
\begin{equation}\label{C collisional law}
    {C}(X_{\ell+1}, V_{\ell+1})  = \frac{2}{\ell+1} \sum_{1\leq i < j \leq \ell+1} \langle v_i - v_j, \frac{x_i - x_j}{\epsilon_{\ell+1}} \rangle.
\end{equation}
\par
The conservation of momentum
\begin{equation}\label{cons momentum}
v_1^{*}+\dots+ v_{\ell+1}^*=v_1+\dots+ v_{\ell+1},
\end{equation}
follows from \eqref{coll law}, and the conservation of energy
\begin{equation}
   |v_1^{*}|^2+\dots+ |v_{\ell+1}^*|^2=|v_1|^2+\dots+ |v_{\ell+1}|^2, 
\end{equation}
is guaranteed by \eqref{C collisional law}.    

It is immediate from \eqref{dist} and \eqref{coll law} that the $\ell + 1$-nary collisional law is symmetric with respect to a relabeling of the particles. Hence, without loss of generality, for this paper, we use the following convention: for an $\ell+1$-order interaction we define the impact parameters $\{\omega_i \}_{i = 1}^{\ell}$ by,
\begin{equation}\label{omega convention}
    \omega_i = \frac{x_{i+1} - x_1}{\epsilon_{\ell+1}}, \quad i \in \{1,\cdots, \ell \}.
\end{equation}
Under this parameterization, the $\ell+1$-collisional law \eqref{coll law}-\eqref{C collisional law} becomes,
\begin{align}
    &v_1^* = v_1 + c(\bm{\omega},V_{\ell+1}) \sum_{i=1}^{\ell} \omega_i, \\
    &v_{j+1}^* = v_{j+1} + c(\bm{\omega}, V_{\ell+1}) \bigg( -\ell \omega_{j} + \sum_{\substack{i=1 \\ i\neq j}}^{\ell} \omega_i \bigg) \quad j \in \{1,\cdots, \ell \}, \\
    c(\bm{\omega},&V_{\ell+1})= \frac{2}{\ell+1} \bigg( \sum_{i=1}^{\ell} \langle v_{i+1} - v_1, \omega_i \rangle + \sum_{1\leq i < j \leq \ell} \langle v_{i+1} - v_{j+1}, \omega_i - \omega_j \rangle \bigg),
\end{align}
where $\bm{\omega} = (\omega_1,\cdots, \omega_{\ell})$.
\begin{remark}\label{symm remark}
    We note that, unlike in \cite{AmpatzoglouPavlovic2020}, the first particle $x_1$ is actually not special or ``central" thanks to the symmetry of the collisional law given in \eqref{coll law}.
\end{remark}
\begin{remark}\label{ellipsoid remark}
Another key departure from the asymmetric setting in \cite{AmpatzoglouPavlovic2020}, where the impact parameters exist on the sphere, is the fact that now 
the impact parameters \eqref{omega convention} belong to an ellipsoid. More precisely, making the substitution for $\bm{\omega} = (\omega_1,\cdots \omega_\ell)$ with $\omega_i$ given by \eqref{omega convention}, into \eqref{dist int zone}, we find that the impact parameters belong to the following ellipsoid, which in this paper we denote as $\E$. That is, $\bm{\omega} \in \E \subseteq \R^{\ell d}$ where,
\begin{equation}
    \E = \{ (\nu_1,\cdots, \nu_{\ell}) \in \R^{\ell d} \, : \,\sum_{i=1}^\ell|\nu_i|^2  + \sum_{1\leq i < j \leq \ell} |\nu_i - \nu_j|^2 = 1 \}.
\end{equation}
\end{remark}

The differences described in the above two remarks present new challenges compared to the previous works. We will address them in more details in the next two subsections. We start by presenting two big picture challenges.

\subsection{Challenges in the derivation of the higher order Boltzmann equation }
\begin{enumerate} 
\item The first challenge that we needed to address is the ability to detect each of the finitely many higher order interactions on its own, i.e. binary interactions, ternary interactions, all the way to the interactions involving $M+1$ particles. Our previous work \cite{AmpatzoglouPavlovic2020} solved a version of this problem with two different types of interactions, binary and ternary, by considering two different, but related scalings. In this paper we generalize that concept as follows. Consider $\epsilon_2, \dots,\epsilon_{M+1}<<1.$ We assume the following:

\begin{itemize}
\item Particles are hard spheres of diameter $\epsilon_2$, 

\item that can also interact as triplets via the interaction zone $\epsilon_3$, as quartets via the interaction zone $\epsilon_4$, quintets via the interaction zone $\epsilon_5$, all the way to groups of $M+1$ interacting particles for which the interaction zone 
$\epsilon_{M+1}$ applies. As we heuristically pointed out in \eqref{scaling intro}, we will rigorously see that the natural scaling to simultaneously detect all these interactions is:
\footnote{Note that in the case of binary interactions when $\ell=1$ the scaling
\eqref{intro scaling} recovers 
the Boltzmann-Grad scaling, while in the case of ternary interaction when $\ell=2$, the scaling \eqref{intro scaling} recovers the scaling law used in \cite{ternary} on derivation of a ternary Boltzmann equation.}
\begin{equation}\label{intro scaling}
    N  \epsilon_{\ell+1}^{d-\frac{1}{\ell}} = \frac{1}{\ell!}, \quad \forall \ell \in \{1,\cdots, M \}.
\end{equation}
We note that we opted to use $1/\ell!$ in the right hand side of \eqref{intro scaling} to guarantee that the well-posedness time does not shrink as the collision order increases. For more details see Section \ref{sec:local}.
\item Note that \eqref{intro scaling} implies the nesting of the interaction zones
\begin{equation}\label{epsilons}
0 <\epsilon_2 << \cdots << \epsilon_{M+1}<<1. 
\end{equation}

\end{itemize}

\item The next and our main challenge is to build the mechanism to decouple interactions. Our framework a-priori allows that the same particles are involved in different interactions (that could be of the same or distinct order from  $2, \cdots, M+1$). For example, 
particles $\{x_1,...x_k\}$ interacting via  
$$d_k(x_1,...,x_k)=\epsilon_k,$$
could interact with another particle $x_{k+1}$ via 
$$d_{k+1}(x_1,...,x_k,x_{k+1})=\epsilon_{k+1}.$$ 
Pathological configurations, including the one above, are going to be shown to be negligible. This is far from a trivial task, and we will describe it in more details in the next subsection where we summarize the derivation process.

\end{enumerate}

\subsection{Roadmap of derivation and main results in a nutshell} \label{road}
In this subsection we present a roadmap of the derivation of \eqref{general boltz} and give an informal statement of our main results.

\subsubsection{Description of the global flow}
The phase space is given by
\begin{align}
    \mathcal{D}_{N,\bm{\epsilon}_M} =& \{ Z_N = (X_N,V_N) \in \mathbb{R}^{2dN} \, : \, \forall \ell \in \{1,\cdots,\min \{N-1,M\} \} \\
    &\text{ and } \forall (i_1, \cdots, i_{\ell+1}) \in I^{\ell+1}_N \text{ we have } d_{\ell+1}(x_{i_1},\cdots, x_{i_{\ell+1}}) \geq \epsilon_{\ell+1}   \},
\end{align}
where we denote,
\begin{equation}
    \bm{\epsilon}_M = (\epsilon_2,\cdots, \epsilon_{M+1}),
\end{equation}
and 
\begin{equation}
    I^{\ell+1}_N = \{ (i_1,\cdots, i_{\ell+1}) \in \{1, \cdots, N \}^{\ell+1} \,: \, i_1 < \cdots < i_{\ell+1} \}, \quad \ell \in \{ 1, \cdots, M\}.
\end{equation}
Furthermore, let $\mathring{\mathcal{D}}_{N,\bm{\epsilon}_M}$ denote the interior of the phase space and $\partial \mathcal{D}_{N,\bm{\epsilon}_M}$ denote its boundary. We define the time evolution of this system as follows:
\begin{enumerate}
    \item For $Z_N \in \mathring{\mathcal{D}}_{N,\bm{\epsilon}_M}$ we define $Z_N(t)$ by rectilinear motion up until some collision time $t_c>0$ where $Z_N(t_c) \in \partial \mathcal{D}_{N,\bm{\epsilon}_M}$.
    \item For $Z_N(t) \in \partial \mathcal{D}_{N,\bm{\epsilon}_M}$ the velocities of the particles in the interaction are instantaneously transformed $(v_{i_1},\cdots, v_{i_{\ell+1}}) \rightarrow (v^*_{i_1},\cdots, v^*_{i_{\ell+1}})$ where $\ell+1$ is the order of the interaction and $(i_1,\cdots, i_{\ell+1}) \in I^{\ell+1}_N$ are the interacting particles. We denote the resulting configuration as $Z_N^{*,\ell+1}$.
\end{enumerate}

In general, time evolution by this method can run into a slew of different pathological phase space configurations including: more than one interaction occurring at the same time, grazing collisions, and infinitely many interactions occurring in finite time. By extending the work done in \cite{AmpatzoglouPavlovic2020}, we show that these pathological sets can be covered by a Lebesgue zero set and we derive the global flow of the system on the complement. The global flow yields the Liouville equation for our setting, which describes the evolution of the $N$-particle probability density function $f_N$ from an initial distribution given by $f_{N,0}$,
 \begin{equation}\label{intro Liouville}
	\begin{cases}
	\partial_t f_N + \sum_{i=1}^N v_i \cdot \nabla_{x_i} f_N = 0 \quad (t,Z_N) \in (0,\infty) \times \mathring{\mathcal{D}}_{N,\bm{\epsilon_{M}}}  \\
	f_N(t,Z^{*,\ell+1}_N) = f_N(t,Z_N), \quad (t,Z_N) \in [0,\infty) \times \partial \mathcal{D}_{N,\bm{\epsilon_M}} \quad \forall \ell \in \{ 1, \cdots, M\} \\
	f_{N,0}(0,Z_N) = f_{N,0}(Z_N), \quad Z_N \in \mathring{\mathcal{D}}_{N,\bm{\epsilon_{M}}}.
	\end{cases} 
\end{equation}
Further details of the derivation of the global flow and the Liouville equation can be found in Section \ref{sec::dynamics}.
\subsubsection{Finite and infinite hierarchies}
We now proceed by describing the BBGKY hierarchy of our system, which is a hierarchy of differential equations acting  on the marginals $f_N^{(s)}$, which are defined as,
\begin{equation}	
	f^{(s)}_N(Z_s)
	=
	\begin{cases}
	\int_{ \mathbb{R}^{d ( N -s )} }  f_N (Z_N) dx_{s+1} \cdots dx_{N} d v_{s + 1 }\cdots d v_{N} \,  &  s < N  \\
	f_N & s =  N  \\ 
	0 & s > N.
	\end{cases} 
\end{equation}
By multiplying \eqref{intro Liouville} by a suitable test function and applying a few simple manipulations, we obtain the BBGKY hierarchy,
\begin{equation}\label{intro bbgky}
 	\begin{cases}
 	\partial_t f_N^{(s)} + \sum_{i=1}^s v_i \cdot \nabla_{x_i} f_N^{(s)} &= \sum_{\ell = 1}^M \mathcal{C}^N_{s,s+\ell}f_N^{(s+\ell)}, \quad (t,Z_s) \in (0,\infty)\times \mathring{\mathcal{D}}_{s,\bm{\epsilon_M}} \\
 	f^{(s)}_N(t,Z^{*,\ell+1}_s) &= f_N^{(s)}(t,Z_s), \quad (t,Z_s) \in [0,\infty) \times \partial \mathcal{D}_{s,\bm{\epsilon_M}} \quad \forall \ell \in \{1,\cdots, M \}\\ 
 	f_N^{(s)}(0,Z_s) &= f_{N,0}^{(s)}(Z_s), \quad Z_s \in {\mathcal{D}}_{s,\bm{\epsilon_M}},
\end{cases} 
\end{equation}
where $(\mathcal{C}^N_{s,s+\ell})_{s\in\N} $ are called the BBGKY hierarchy operators and are given by \eqref{BBGKY heir}. Our goal is to prove convergence, in the proper sense, of the BBGKY hierarchy to the Boltzmann hierarchy, which is defined as,
\begin{equation}\label{intro boltz}
 \begin{cases}
 	\partial_t f^{(s)} + \sum_{i=1}^s v_i \cdot \nabla_{x_i} f^{(s)} = \sum_{\ell = 1}^M \mathcal{C}^\infty_{s,s+\ell}f^{(s+\ell)} \quad (t,Z_s) \in (0,\infty)\times \mathbb{R}^{\ell ds} \\
 	f^{(s)}(0,Z_s) = f^{(s)}_0 (Z_s) \quad \forall Z_s \in \mathbb{R}^{\ell ds},
\end{cases} 
\end{equation}
where we call $(\mathcal{C}^\infty_{s,s+\ell})_{s\in\N} $, defined via \eqref{Boltzmann heir}, the Boltzmann hierarchy operators. This convergence occurs in the limit $N \rightarrow \infty$ under the common scaling given by \eqref{intro scaling}.
\par

\subsubsection{Main results of the paper}
Performing the limit $N \rightarrow \infty$ requires us to define an appropriate adjunction of particles that avoids the problem of recollisions. At this stage of the derivation a symmetric collision law is necessary, see Remark \ref{symm remark}. Attempting to utilize a higher order generalization of the asymmetric collision law used in \cite{AmpatzoglouPavlovic2020,ternary} fails due to the inability to sufficiently control recollisions between the same particles with different central particles. However, the adoption of the symmetric collision law does introduce additional challenges associated with the fact that the impact parameters for a given $\ell+1$ interaction exists on the ellipsoid $\E$ as opposed to the sphere $\mathbb{S}^{\ell d -1}_1$, see Remark \ref{ellipsoid remark}. The majority of Sections \ref{sec::geometric estimates} and \ref{sec:stability} are dedicated to the estimation of the more complicated geometric sets that arise. Once we have properly estimated these pathological configurations we are able to take the limit $N\rightarrow \infty$ and prove our main result, Theorem \ref{convergence}, which we state in less rigor now:
\par
\textbf{Informal statement of the convergence result.} Consider $F_0$ initial data  for the Boltzmann hierarchy \eqref{intro boltz} and $F_{N,0}$ initial data for the BBGKY hierarchy \eqref{intro bbgky} which approximates $F_0$ as $N \rightarrow \infty$ and $\epsilon_{\ell+1} \rightarrow 0^+$ for all $\ell \in \{1,\cdots, M \}$ in the scaling \eqref{intro scaling}. Let $\bm{F_N}$ be the mild solution to the BBGKY hierarchy with initial data $F_{N,0}$ and let $\bm{F}$ be the mild solution of the Boltzmann hierarchy with initial data $F_0$ up to a short time $T>0$. Then, $\bm{F_N}$ converges in observables to $\bm{F}$ in $[0,T]$ as $N \rightarrow \infty$ and $\epsilon_{\ell+1} \rightarrow 0^+$ for all $\ell \in \{1,\cdots, M \}$ in the scaling \eqref{intro scaling}.
\par
We also prove the propagation of chaos as a corollary to Theorem \ref{convergence} in Corollary \ref{cor prop of chaos}. That is, for tensorized initial data $(f_0^{\otimes s})_{s \in \N}$ we prove convergence to the mild solution of the Boltzmann equation, denoted as $f$ up to a small time $T > 0$,
\begin{equation}\label{generalized boltz}
 \begin{cases}
 	\partial_t f + v \cdot \nabla_x f  = \sum_{\ell = 1}^M \frac{1}{\ell !}Q_{\ell+1} (f,\cdots ,f), \quad (t,x,v) \in (0,T)\times \mathbb{R}^{2d}, \\
 	f(0,x,v) = f_0 (x,v) \quad (x,v) \in \mathbb{R}^{2d},
\end{cases} 
\end{equation}
where, for $\ell = 1,\cdots, M $, the $\ell+1$-nary collisional operator $Q_{\ell+1}$ is given by
\begin{equation} 
    Q_{\ell+1}(f,\cdots,f)(t,x,v) =  \int_{\E \times \mathbb{R}^{\ell d}} (f^{*}\cdot f^{*}_1 \cdots f^{*}_{\ell} - f \cdot  f_1 \cdots f_{\ell}) b^+_{\ell+1}(\bm{\omega}, v_1 - v, \cdots, v_\ell - v)  d\bm{\omega} dv_1 \cdots dv_{\ell},
\end{equation}
where,
\begin{equation}
\begin{split}
    &f^{*} = f(t,x,v^{*}), \quad f = f(t,x,v), \\
    &f^{*}_i = f(t,x,v^{*}_i), \quad f_i = f(t,x,v_i) \quad i\in \{1,\cdots, \ell \},
\end{split}
\end{equation}
and $b^+_{\ell+1}$ represents the positive part of the $\ell+1$-order hard sphere cross-section which is given in \eqref{cross section}.
\par
We conclude the introduction by emphasizing that 
the main difficulty that we faced in this paper was the possibility of recollisions of arbitrary, potentially distinct, order. Consequently, such situations could lead to a variety of pathological scenarios under backwards time evolution. As opposed to previous works where geometric tools were built for specific situations, in this paper we took more a systematic approach to quantify novel intersections of  e.g. a generic ellipsoid and a cylinder (see Lemma \ref{ellipsoidal est}), or an ellipsoid with the annulus corresponding to a {\it general quadratic form} (see Lemma \ref{annuli est}).

We anticipate that the higher order Boltzmann equation \eqref{general boltz} derived in this work can be studied on its own right, e.g. global well-posedness and convergence to the equilibrium. Also, as mentioned earlier in the introduction, the solutions to the binary-ternary Boltzmann equation (\eqref{general boltz} with $M=2$) exhibit better moments behavior, in certain situations, than solutions of the binary-only equation. Is that the case for the higher order Boltzmann equation \eqref{general boltz} too?

\subsection*{Organization of the paper}
In Section \ref{sec::dynamics}, we provide definitions for our higher order collision law and the phase space of the finite system. We go on to prove the global flow of this system and rigorously derive the corresponding Liouville equation. In Section \ref{sec::hier}, we define the BBGKY and Boltzmann hierarchies and prove their local well-posedness in Section \ref{sec:local}. Section \ref{sec_conv statement} provides the precise statement of our main result, Theorem \ref{convergence}. In Section \ref{sec: series expansion} we reduce the proof of  Theorem \ref{convergence} to proving the term by term convergence of observables. Section \ref{sec::geometric estimates} contains the necessary geometric estimates used for proving the stability of adjunctions, Proposition \ref{bad set triple}, which is contained in Section \ref{sec:stability}. In Section \ref{sec::reco} we introduce pseudo trajectories for the elimination of recollisions and complete the proof of Theorem \ref{convergence} in  Section \ref{sec::convergence}.

\subsection*{Acknowledgements}
I.A. gratefully acknowledges support from the NSF grants No. DMS-2418020, DMS-2206618. N.P. gratefully acknowledges support from the NSF under grants No. DMS-1840314, DMS-2009549 and DMS-2052789. 
W.W. gratefully acknowledges support from  the NSF grant No. DMS-1840314.

\section{Finite Dynamics}\label{sec::dynamics}
In this section we derive the global in time flow of $m\in \N$ hard spheres undergoing $\ell+1$-order collisions for $\ell \in \{ 1,\cdots, M\}$ and $M \in \N$. The existence of the global flow is not immediate, requiring the removal of  certain pathological configurations which the system may run into through time evolution. These pathological configurations correspond to the rare instances of grazing collisions, multiple interactions occurring at the same time, and infinitely many interactions occurring in finite time. We follow the program set forth by Alexander in \cite{alexander} for the binary case and extended to the ternary case in \cite{AmpatzoglouPavlovic2020, ternary}. We further expand this program to accommodate for collisions of arbitrary order to prove that these pathological configurations can be covered by a measure zero set and to establish the existence of a measuring preserving global in time flow on the complement. 
\subsection{Higher order collisional transformation}
 In our setting, an $\ell + 1$ interaction for $\ell \in \{1,\cdots, M\}$ is defined via a symmetric distance $d_{\ell  +1}$ and interaction zone $\sigma_{\ell+1}$. Given the positions of $\ell + 1$ particles, $(x_1, \cdots, x_{\ell+1})$ the distance $d_{\ell+1}$ is given by,
\begin{equation}
        d_{\ell+1}^2(x_1, \cdots,x_{\ell+1}) =\sum_{1\leq i < j \leq \ell+1} |x_{i} - x_{j} |^2.
\end{equation}
An $\ell + 1$-order interaction occurs when the symmetric distance between $\ell + 1$ particles is equal to the corresponding interaction zone, $\sigma_{\ell+1}$. That is, for the positions $(x_1, \cdots, x_{\ell+1})$, we require,
\begin{equation}\label{dist sig}
    d_{\ell+1}^2(x_{1}, \cdots,x_{\ell + 1}) = \sigma_{\ell+1}^2.
\end{equation}
In order to ensure that higher order interactions can occur, we impose the restriction,
\begin{equation}
    0 < \sigma_2 < \cdots < \sigma_{M+1} << 1.
\end{equation}
The scaled relative positions $\omega_i$ for an $\ell + 1$ interaction of $(x_1, \cdots, x_{\ell+1})$ are defined by 
\begin{equation}\label{scaled rel imp}
    \omega_i = \frac{x_{i+1} - x_1}{\sigma_{\ell+1}} \quad i \in \{ 1, \cdots, \ell \}
\end{equation}
and due to \eqref{dist sig} they satisfy the condition
\begin{equation}
    \sum_{i=1}^{\ell}|\omega_i|^2 + \sum_{1\leq i < j \leq \ell} |\omega_i - \omega_j|^2 = 1
\end{equation}
which defines a $(\ell d -1)$-dimensional  ellipsoid which we denote as $\E$ given by,
\begin{equation}\label{ellipsoid}
    \E = \{ (\nu_1, \cdots, \nu_\ell) \in \R^{\ell d} \, : \, \sum_{i=1}^\ell|\nu_i|^2  + \sum_{1\leq i < j \leq \ell} |\nu_i - \nu_j|^2 = 1   \}.
\end{equation}
Under the dynamics of the system, the particles undergo rectilinear motion until the symmetric distance of some $\ell + 1$ particles equals the corresponding interaction zone $\sigma_{\ell + 1}$ thus triggering an interaction and updating their velocities. 
\par
While \eqref{scaled rel imp} serves as a motivation for impact parameters, in order to use them for systems of infinitely many particles as well as for the nonlinear equation, we define the collisional transformation for a general set of impact parameters independent of positions. 

\begin{definition}\label{defn T omega}
    Consider the impact parameters $(\omega_1, \cdots, \omega_\ell ) \in \E$. The $(\ell+1)$-order collisional transformation induced by $\bm{\omega}= (\omega_1, \cdots, \omega_\ell)$ is defined as the map $T^{\ell + 1}_{\bm{\omega}}: \R^{(\ell + 1)d} \to \R^{(\ell + 1) d}$,
    \begin{align}\label{v star defn}
        T^{\ell + 1}_{\bm{\omega}}(v_1, \cdots v_{\ell + 1 }) \mapsto (v_1^*, \cdots, v_{\ell+1}^*),
    \end{align}
    where
\begin{align}
    &v_1^* = v_1 + c(\bm{\omega},v_1,\cdots, v_{\ell+1}) \sum_{i=1}^{\ell} \omega_i, \\
    &v_{j+1}^* = v_{j+1} + c(\bm{\omega}, v_1,\cdots, v_{\ell+1}) \bigg( -\ell \omega_{j} + \sum_{\substack{i=1 \\ i\neq j}}^{\ell} \omega_i \bigg) \quad j \in \{1,\cdots, \ell \}, \\
    c(\bm{\omega},&v_1,\cdots, v_{\ell+1})= \frac{2}{\ell+1} \bigg( \sum_{i=1}^{\ell} \langle v_{i+1} - v_1, \omega_i \rangle + \sum_{1\leq i < j \leq \ell} \langle v_{i+1} - v_{j+1}, \omega_i - \omega_j \rangle \bigg). \label{impact param}
\end{align}
\end{definition}
We also define the $\ell+1$-order cross-section
\begin{equation}
    b_{\ell+1}(\bm{\omega}, \nu_1,\cdots,\nu_\ell) := \sum_{i=1}^\ell \langle \omega_i, \nu_i \rangle + \sum_{1\leq i < j \leq \ell} \langle \omega_i-\omega_j, \nu_i - \nu_j \rangle, \quad (\omega_1,\cdots, \omega_\ell) \in \E, \quad (\nu_1,\cdots, \nu_\ell) \in \R^{\ell d},
\end{equation}
so that we have, 
\begin{equation} \label{cross section}
    b_{\ell+1}(\bm{\omega}, v_2- v_1, \cdots, v_{\ell+1} - v_1) = \frac{\ell+1}{2} c(\bm{\omega}, v_1, \cdots, v_{\ell+1}).
\end{equation}
\begin{proposition}\label{coll prop}
    For $\ell \in \{1,\cdots, M \}$ and fixed $\bm{\omega}= (\omega_1, \cdots, \omega_\ell)\in \E$, the collisional transformation $T^{\ell+1}_{\bm{\omega}}$ has the following properties:
    \begin{enumerate}
        \item Conservation of momentum
        \begin{equation}
            v_1^* + \cdots + v_{\ell + 1}^* = v_1 + \cdots + v_{\ell + 1}.
        \end{equation}
        \item Conservation of energy
        \begin{equation}
            |v_1^*|^2 + \cdots + |v_{\ell + 1}^*|^2 = |v_1|^2 + \cdots + |v_{\ell + 1}|^2.
        \end{equation}
        \item Conservation of relative velocities magnitude
        \begin{equation}
            \sum_{1\leq i < j \leq \ell+1} |v_i^* - v_j^*|^2 = \sum_{1\leq i < j \leq \ell+1} |v_i - v_j|^2.
        \end{equation}
        \item Micro-reversibility of the $\ell+1$ cross-section
        \begin{equation}\label{micro reversibility}
              b_{\ell+1}(\bm{\omega}, v^*_2- v^*_1, \cdots, v^*_{\ell+1} - v^*_1)=-b_{\ell+1}(\bm{\omega}, v_2- v_1, \cdots, v_{\ell + 1} - v_1).
        \end{equation}
        \item $T^{\ell + 1}_{\bm{\omega}}$ is a linear involution.
        
    \end{enumerate}
\end{proposition}
\begin{proof}
    The proof of this proposition can be found in \cite{WarnerThesis}.
\end{proof}
\subsection{Phase space definitions.}
We now move our attention to defining the relevant phase spaces necessary for deriving our local and global dynamics. Recall that for this section we fix $m\in \N$ to denote the number of particles in the system. We first introduce the index sets:
\begin{equation}
    I^{\ell+1}_m = \{ (i_1,\cdots, i_{\ell+1}) \in \{1, \cdots, m \}^{\ell+1} \,: \, i_1 < \cdots < i_{\ell+1} \}, \quad \ell \in \{ 1, \cdots, M\}.
\end{equation}
Thus, we define the phase space $\mathcal{D}_{m,\bm{\sigma}_M}$,
\begin{align}\label{phase space defn}
    \mathcal{D}_{m,\bm{\sigma}_M} =& \{ Z_m = (X_m,V_m) \in \mathbb{R}^{2dm} \, : \, \forall \ell \in \{1,\cdots,\min \{m-1,M\} \} \\
    &\text{ and } \forall (i_1, \cdots, i_{\ell+1}) \in I^{\ell+1}_m \text{ we have } d_{\ell+1}(x_{i_1},\cdots, x_{i_{\ell+1}}) \geq \sigma_{\ell+1}   \},
\end{align}
where we denote,
\begin{equation}
    \bm{\sigma}_M = (\sigma_2,\cdots, \sigma_{M+1}).
\end{equation}
\begin{remark}
    Note, that in the definition of the phase space given in \eqref{phase space defn} we allow $\ell$ to iterate up to $\min \{m-1,M\}$. That is because if $m < M+1$, the highest order of interaction that we can observe is $m$.
\end{remark}
Let us denote by $\partial D_{m,\bm{\sigma}_M}$ the boundary of the phase space. We decompose $\partial \mathcal{D}_{m,\bm{\sigma}_M}$ as follows:
\begin{equation}
 \partial \mathcal{D}_{m,\bm{\sigma}_M}=\bigcup_{\ell=1}^M  \partial_{\ell+1} \mathcal{D}_{m,\bm{\sigma}_M}, 
\end{equation}
where the $\ell+1$-collisional boundary is defined by 
\begin{equation}
    \partial_{\ell+1} \mathcal{D}_{m,\bm{\sigma}_M} := \bigcup_{(i_1,\cdots, i_{\ell+1})\in I^{\ell+1}_m} \Sigma^{\ell+1}_{i_1,\cdots,i_{\ell+1}}
\end{equation}
with
\begin{equation}
    \Sigma^{\ell+1}_{i_1,\cdots,i_{\ell+1}} = \{ Z_m \in \mathcal{D}_{m,\bm{\sigma}_M} \, : \, d_{\ell+1}(i_1,\cdots, i_{\ell+1}) =  \sigma_{\ell+1} \}.
\end{equation}
Now we define the space of ${\ell+1}$-\textit{simple collisions} to be 
\begin{equation}\label{simple collisions def}
\begin{split}
    \partial^{sc}_{\ell+1} \mathcal{D}_{m,\bm{\sigma}_M} = \{ Z_m \in \mathcal{D}_{m,\bm{\sigma}_M} | \, &\exists (i_1,\cdots, i_{\ell+1}) \in I^{\ell+1}_m \quad \text{s.t.} \quad Z_m \in \Sigma^{\ell+1}_{i_1,\cdots,i_{\ell+1}} \\
    &\text{and } Z_m \notin \Sigma^{\ell+1}_{i'_1,\cdots,i'_{\ell+1},} \forall (i'_1,\cdots, i'_{\ell+1}, ) \in I_N^{\ell+1}\setminus\{(i_1,...,i_{\ell+1})\}  \\
    &\text{and } \forall k \neq \ell \in \{1, \cdots, M \}, \, Z_m \notin  \Sigma^{k+1}_{i'_1,\cdots,i'_{k+1}} \forall (i_1 ',\cdots, i_{k+1} ') \in I^{k+1}_ m\}.
\end{split}
\end{equation}
Then the  space of all simple collisions is introduced as the union over all $\ell +1$-simple collisions: 
\begin{equation}
    \partial ^{sc} \mathcal{D}_{m,\bm{\sigma}_M} := \bigcup _{\ell = 1 }^{M} \partial_{\ell+1}^{sc} \mathcal{D}_{m,\bm{\sigma}_M}.
\end{equation}
We define multiple collisions as those which are not simple:
\begin{equation}
    \partial ^{mu} \mathcal{D}_{m,\bm{\sigma}_M} :=\partial  \mathcal{D}_{m,\bm{\sigma}_M} \setminus \partial ^{sc} \mathcal{D}_{m,\bm{\sigma}_M}.
\end{equation}
We provide the following definition to introduce the notation $Z_m^{*,\ell+1}$ for $Z_m \in \partial^{sc}_{\ell+1} \mathcal{D}_{m,\bm{\sigma}_M}$ and $\ell \in \{1, \cdots, M \}$ which will be used in the remainder of this section.
\begin{definition}\label{z* defn}
    Let $m\in\N $, $\ell\in\{1,...,M\}$, $(i_1,\cdots,i_{\ell+1})\in I_{m}^{\ell+1}$ and $Z_m = (X_m, V_m) \in \partial^{sc}_{\ell+1} \mathcal{D}_{m,\bm{\sigma}_M} \cap \Sigma^{\ell+1}_{i_1,\cdots,i_{\ell+1}}$. We denote $Z_m^{*,\ell+1} = (X_m, V_m^{*, \ell+1})$ where 
    \begin{equation}
        V_m^{*, \ell+1} = (v'_1, \cdots, v'_m),
    \end{equation}
    where for all $i \in \{1,\cdots, m \}$
    \begin{align}
        v'_{i} &= v_{i}^*, \quad \text{ for } i\in (i_1,\cdots, i_{\ell+1}), \\
        v'_{i} &= v_i, \quad \text{ for } i\notin (i_1,\cdots, i_{\ell+1}).
    \end{align}
\end{definition}
We now provide the following classification for simple collisions.
\begin{definition}
Let $m\in\N $, $\ell\in\{1,...,M\}$, $(i_1,\cdots,i_{\ell+1})\in I_{m}^{\ell+1}$ and $Z_m \in \partial^{sc}_{\ell+1} \mathcal{D}_{m,\bm{\sigma}_M} \cap \Sigma^{\ell+1}_{i_1,\cdots,i_{\ell+1}}$. We shall say that $Z_m$ is 
\begin{itemize}
    \item ${\ell+1}$-precollisional if $b_{\ell+1}(\omega_{i_1},\cdots,\omega_{i_\ell},v_{i_2}-v_{i_1},\cdots, v_{i_{\ell+1}}-v_{i_1}) < 0$,
    \item ${\ell+1}$-postcollisional if $b_{\ell+1}(\omega_{i_1},\cdots,\omega_{i_\ell},v_{i_2}-v_{i_1},\cdots, v_{i_{\ell+1}}-v_{i_1}) > 0$,
    \item ${\ell+1}$-grazing if $b_{\ell+1}(\omega_{i_1},\cdots,\omega_{i_\ell},v_{i_2}-v_{i_1},\cdots, v_{i_{\ell+1}}-v_{i_1}) = 0$.
\end{itemize}
\end{definition}
\par
We now define the refined phase space as follows
\begin{equation}
    \mathcal{D}_{m,\bm{\sigma}_M}^* = \mathring{\mathcal{D}}_{m,\bm{\sigma}_M} \cup \partial^{sc,ng} \mathcal{D}_{m,\bm{\sigma}_M},
\end{equation} 
where the interior $\mathring{\mathcal{D}}_{m,\bm{\sigma}_M}$ is given by
\begin{align}
    \mathring{\mathcal{D}}_{m,\bm{\sigma}_M} = \{ Z_m = (X_m,V_m) \in \mathbb{R}^{2dm} \, :& \, \forall \ell \in \{1,\cdots,M \} \text{ and } \forall (i_1, \cdots, i_{\ell+1}) \in I^{\ell+1}_m \\
    &\text{ we have } d_{\ell+1}(x_{i_1},\cdots, x_{i_{\ell+1}}) > \sigma_{\ell+1}   \},
\end{align}
and $\partial^{sc,ng} \mathcal{D}_{m,\bm{\sigma}_M}$ denotes the part of the boundary consisting of simple, non-grazing collisions:
\begin{equation}
\partial^{sc,ng}\mathcal{D}_{m,\bm{\sigma}_M}:=\left\{Z_m\in\partial^{sc}\mathcal{D}_{m,\bm{\sigma}_M}: Z_m\text{ is non-grazing}\right\}.    
\end{equation}
Notice that $\mathcal{D}_{m,\bm{\sigma}_M}^*$ is a full measure subset of the full phase space $\mathcal{D}_{m,\bm{\sigma}_M}$.

\subsection{Local Flow}
We begin the construction of the global in time flow by first defining the local flow of a system with configuration $Z_m \in \mathcal{D}_{m,\bm{\sigma}_M}^*$ undergoing rectilinear motion up until the first collision.
\begin{lemma}
Let $Z_m \in \mathcal{D}_{m,\bm{\sigma}_M}^*$. There exists a time $\tau_{Z_m}^1 \in (0,\infty]$ such that if $Z_m(\cdot) : [0,\tau_{Z_m}^1] \rightarrow \mathbb{R}^{2dm}$ is defined by
\begin{equation}
    Z_m(t) = 
    \begin{cases}
        (X_m + tV_m, V_m) \quad \text{if $Z_m$ is noncollisional or postcollisional,}\\
        (X_m + tV_m^{*{\ell+1}}, V_m^{*{\ell+1}}) \quad \text{if $Z_m$ is ${\ell+1}$-precollisional,}\\
    \end{cases}
\end{equation}
then the following holds:
\begin{enumerate}[label=(\roman*)]
\item $Z_m(t) \in \mathring{\mathcal{D}}_{m,\bm{\sigma}_M} \quad \forall t\in (0,\tau_{Z_m}^1). $
\item if $\tau_{Z_m}^1 < \infty$, then $Z_m(\tau_{Z_m}^1)\in \partial \mathcal{D}_{m,\bm{\sigma}_M}$.
\item If $Z_m \in \Sigma^{\ell+1}_{i_1,\cdots,i_{\ell+1}}$, then $Z_m(\tau_{Z_m}^1) \notin  \Sigma^{\ell+1}_{i_1,\cdots,i_{\ell+1}}$.
\end{enumerate}
\end{lemma}
\begin{proof}
For $Z_m \in \mathcal{D}_{m,\bm{\sigma}_M}^*$ we define $\tau_{Z_m}^1$ by,
\begin{equation}
    \tau_{Z_m}^1 = 
    \begin{cases}
        \text{inf}\{ t> 0 \, | \, X_m + t V_m \in \partial \mathcal{D}_{m,\bm{\sigma}_M} \}, \quad \text{if } Z_m \text{ is noncollisional or postcollisional}\\
        \text{inf}\{ t> 0 \,| \, X_N + t V_m^{*{\ell+1}} \in \partial \mathcal{D}_{m,\bm{\sigma}_M} \}, \quad \text{if } Z_m \text{ is ${\ell+1}$-precollisional}.
    \end{cases}
\end{equation}
By the construction of $\tau_{Z_m}^1$ we immediately obtain $(ii)$ and since $\mathring{\mathcal{D}}_{m,\bm{\sigma}_M}$ is open we obtain $(i)$. 
\par
To prove $(iii)$ we assume $Z_m \in \partial^{sc,ng} \mathcal{D}_{m,\bm{\sigma}_M}$ is in a $(i_1,\cdots,i_{\ell+1})$ ${\ell+1}$-postcollisional configuration. For $t>0$ we have
\begin{align}
    \sum_{1 \leq j < k \leq \ell+1 } | x_{i_k} - x_{i_j} + &(v_{i_k} - v_{i_j})t |^2 =  \sum_{1 \leq j < k \leq \ell+1 } \bigg(|x_{i_k} - x_{i_j}|^2 + t^2 |v_{i_k}-v_{i_j}|^2 + 2t \langle x_{i_k} - x_{i_j}, v_{i_k}-v_{i_j} \rangle \bigg) \\
    &\geq \sigma_{\ell+1}^2 + 2t b_{{\ell+1}}(x_{i_2}-x_{i_1},\cdots, x_{i_{\ell+1}} - x_{i_1}, v_{i_2} - v_{i_1}, \cdots, v_{i_{\ell+1}} - v_{i_1}) \\
    &> \sigma_{\ell+1}^2.
\end{align}
If $Z_m$ is precollisional, then the  same calculation can be repeated by replacing $Z_m$ with $Z_m^{*{\ell+1}}$ and applying micro-reversibility of the cross section $b_{\ell+1}$, see \eqref{micro reversibility}.

\end{proof}

\subsection{Global Flow}\label{subsec::Global Flow}
In order to derive a global in time flow, we must truncate velocities and positions with the parameters $1 << R < \rho$ respectively and truncate time with the parameter $\delta>0$ such that, 
\begin{equation}\label{time step}
    0 < \delta_2 < \cdots < \delta_{M+1} << \sigma_2 < \cdots < \sigma_{M+1} << 1 << R < \rho,
\end{equation}
where
\begin{equation}
    \delta_{\ell+1} = 4(\ell+1)^2 R \delta \quad \text{for all }\ell \in \{1,\cdots, M \}.
\end{equation}
\par
We assume initial positions are in the ball centered at the origin $B^{md}_\rho$ and initial velocities in the ball $B^{md}_R$. For $m \in \N$ we decompose $\mathcal{D}^*_{m,\bm{\sigma}_M} \cap (B^{md}_{\rho} \times B^{md}_{R} )$ into the subsets,
\begin{equation}\label{I sets}
\begin{split}
    &I_{free} = \{ Z_m = (X_m,V_m) \in \mathcal{D}^*_{m,\bm{\sigma}_M} \cap (B^{md}_{\rho} \times B^{md}_{R} ) \, | \, \tau_{Z_m}^1 > \delta \}, \\
    &I^1_{sc,ng} = \{ Z_m = (X_m,V_m) \in \mathcal{D}^*_{m,\bm{\sigma}_M} \cap (B^{md}_{\rho} \times B^{md}_{R} ) \, | \, \tau_{Z_m}^1 \leq \delta, \; Z_m(\tau^1_{Z_m}) \in \partial^{sc,ng}\mathcal{D}_{m,\bm{\sigma}_M} \text{ and } \tau^2_{Z_m} > \delta \}, \\
    &I^1_{sc,g} = \{ Z_m = (X_m,V_m) \in \mathcal{D}^*_{m,\bm{\sigma}_M} \cap (B^{md}_{\rho} \times B^{md}_{R} ) \, | \, \tau_{Z_m}^1 \leq \delta, \; Z_m(\tau^1_{Z_m}) \in \partial^{sc}\mathcal{D}_{m,\bm{\sigma}_M} \text{ and } Z_m(\tau^1_{Z_m}) \text{ is grazing} \}, \\
    &I^1_{mu} = \{ Z_m = (X_m,V_m) \in \mathcal{D}^*_{m,\bm{\sigma}_M} \cap (B^{md}_{\rho} \times B^{md}_{R} ) \, | \, \tau_{Z_m}^1 \leq \delta, \; Z_m(\tau^1_{Z_m}) \in \partial^{mu}\mathcal{D}_{m,\bm{\sigma}_M}  \}, \\
    &I^2_{sc,ng} = \{ Z_m = (X_m,V_m) \in \mathcal{D}^*_{m,\bm{\sigma}_M} \cap (B^{md}_{\rho} \times B^{md}_{R} ) \, | \, \tau_{Z_m}^1 \leq \delta, \; Z_m(\tau^1_{Z_m}) \in \partial^{sc,ng}\mathcal{D}_{m,\bm{\sigma}_M} \text{ and } \tau^2_{Z_m} \leq \delta \}, \\
\end{split}
\end{equation}
where $\tau^2_{Z_m}$ appearing in $I^1_{sc,ng}$ and $I^2_{sc,ng}$ is defined via $ \tau^2_{Z_m} := \tau^1_{Z_m(\tau^1_{Z_m})}$.
\par
The aim of this section is to estimate the measure of the \textit{pathological} subsets of our phase space, which are $I^1_{sc,g}$, $I^1_{mu}$, and $I^2_{sc,ng}$. Specifically, we will prove the theorem, 
\begin{theorem}
For $\delta > 0$, $m \in M$, and $I^1_{sc,g}$, $I^1_{mu}$, $ I^2_{sc,ng}$ given in \eqref{I sets}, we have the following estimate,
\begin{equation}
    |I^1_{sc,g} \cup I^1_{mu} \cup I^2_{sc,ng}| \leq C_{d,R,\rho} \delta^{2}.
\end{equation}
\end{theorem}
\par 
We first note that grazing collisions occur when $b_{\ell+1}(\omega_{i_1},\cdots,\omega_{i_\ell},v_{i_2}-v_{i_1},\cdots, v_{i_{\ell+1}}-v_{i_1}) = 0$, which is a measure zero event and thus immediately gives $|I^1_{sc,g}| = 0$. We will control the measure of $ I^1_{mu} \cup I^2_{sc,ng}$ by the approximation of covering sets. To construct these covering sets, we first define $U^{\ell+1}_{i_1,\cdots,i_{\ell+1}}$, which denotes the subset of the phase space which is in a $\delta_{\ell+1}$ neighborhood of a $(i_1,\cdots,i_{\ell+1})$-collision. Concretely, we define
\begin{equation}
    U^{\ell+1}_{i_1,\cdots,i_{\ell+1}} = \{ Z_m \in B_\rho^{md} \times B_R^{md} \, | \, \sigma_{\ell+1}^2 \leq d^2_{\ell+1}(x_{i_1},\cdots, x_{i_{\ell+1}}) \leq \big(\sigma_{\ell+1} + \delta_{\ell+1} \big)^2\}, \quad \ell \in \{1,\cdots, M \}.
\end{equation}
Now, we define our covering sets $U_{\ell+1,k+1}$ as all subsets of the phase space which occur in the neighborhood of some $\ell + 1$-collision and some different $k+1$-collision, for $\ell, k \in \{1, \cdots, M \}$,
\begin{equation}
U_{\ell+1,k+1} = \bigcup_{\substack{(i_1,\cdots,i_{\ell+1}) \in I^{\ell+1}_m,\; (i_1',\cdots,i_{k+1}') \in {I}^{k+1}_m\\ (i_1,\cdots,i_{\ell+1})\neq (i_1',\cdots, i_{k+1}') }} U^{\ell+1}_{i_1,\cdots,i_{\ell+1}}
 \cap U^{k+1}_{i_1',\cdots,i_{k+1}'}.
 \end{equation}
In the next two lemmas we will establish that the sets $U_{\ell+1,k+1}$ are, in fact, covering sets for our pathological set $I^1_{mu} \cup I^2_{sc,ng}$. 
\begin{lemma}\label{inclusion lemma}
Let $ Z_m \in \mathcal{D}_{m,\bm{\sigma}_M}^* $. Assume $\tau^1_{Z_m} \leq \delta$ and  $Z_m(\tau^1_{Z_m}) \in \Sigma^{\ell+1}_{i_1,\cdots, i_{\ell+1}} \cap (B_\rho^{md}\times B_R^{md})$ for some $\ell \in \{ 1, \cdots, M \}$ and $(i_1,\cdots, i_{\ell+1}) \in I_m^{\ell+1}$. Then  $Z_m \in U^{\ell+1}_{i_1,\cdots,i_{\ell+1}}$.
\end{lemma}
\begin{proof}
Free motion up to $\tau_{Z_m}^1$ implies that for all $k,k' \in \{1,\cdots,\ell\}$, with $k<k'$, we have 
\begin{equation}
\begin{split}
    |x_{i_k} - x_{i_{k'}}|^2 &= |x_{i_k}(\tau_{Z_m}^1) - x_{i_{k'}}(\tau_{Z_m}^1) + \tau_{Z_m}^1(v_{i_k}-v_{i_{k'}})|^2 \leq \big(|x_{i_k}(\tau_{Z_m}^1) - x_{i_{k'}}(\tau_{Z_m}^1) | + 2\delta R\big)^2\\ 
    &= |x_{i_k}(\tau_{Z_m}^1) - x_{i_{k'}}(\tau_{Z_m}^1) |^2 +4\delta R |x_{i_k}(\tau_{Z_m}^1) - x_{i_{k'}}(\tau_{Z_m}^1) | + 4\delta^2 R^2.
\end{split}
\end{equation}
The fact that $Z_m(\tau^1_{Z_m}) \in \Sigma^{\ell+1}_{i_1,\cdots,i_{\ell+1}}$ implies that 
\begin{equation}
    d^2_{\ell+1}\big(x_{i_1}(\tau^1_{Z_m}),\cdots, x_{i_{\ell+1}}(\tau^1_{Z_m})\big) = \sigma_{\ell+1}^2 \implies |x_{i_k}(\tau_{Z_m}^1) - x_{i_{k'}}(\tau_{Z_m}^1) | \leq \sigma_{\ell+1}.
\end{equation}
Combining the above inequalities we have, \begin{equation}
    |x_{i_k} - x_{i_{k'}}|^2 \leq |x_{i_k}(\tau_{Z_m}^1) - x_{i_{k'}}(\tau_{Z_m}^1) |^2 +4\delta R \sigma_{\ell+1} + 4\delta^2 R^2 \quad \text{for all } j \in \{1, \cdots, \ell \}
\end{equation}
Therefore, upon summing over $k,k'\in\{1,...,\ell\}$, with $k<k'$
\begin{equation}
    \sigma_{\ell+1}^2 \leq d^2(x_{i_1},\cdots, x_{i_{\ell+1}}) \leq \sigma_{\ell+1}^2 +4(\ell+1)^2\delta R \sigma_{\ell+1} + 4(\ell+1)^2 \delta^2 R^2 \leq \big(\sigma_{\ell+1} + 4(\ell+1)^2\delta R\big)^2 \leq (\sigma_{\ell+1} + \delta_{\ell+1})^2.
\end{equation}
which completes the proof.
\end{proof}
\begin{lemma}
The sets $\{U_{\ell+1, k+1} \}_{\ell, k =1}^M$, are covering sets for $I^1_{mu}$ and $ I^2_{sc,ng}$, that is
\begin{equation}
    I^1_{mu} \cup I^2_{sc,ng} \subset \bigcup_{\ell,k \in \{1,\cdots,M \}} U_{\ell+1,k+1}
\end{equation}
\end{lemma}
\begin{proof}
$I^1_{mu} \subset \bigcup_{\ell,k \in \{1,\cdots,M \}} U_{\ell+1,k+1}$ follows directly from the definition of $I^1_{mu}$ in \eqref{I sets} and Lemma \ref{inclusion lemma}. $I^2_{sc,ng} \subset \bigcup_{\ell,k \in \{1,\cdots,M \}} U_{\ell+1,k+1}$ follows from two applications of Lemma \ref{inclusion lemma} for $\tau^1_{Z_m} \leq \delta$ and $\tau^2_{Z_m} \leq \delta$.
\end{proof}
Now, we move our attention to estimating the covering sets  $U_{\ell+1,k+1}$. For $(i_1,\cdots, i_{\ell+1})\in {I}_m^{\ell+1}$ we denote  $U^{\ell+1,X}_{i_1,\cdots ,i_{\ell
+1}}$ to be the spatial projection of $U^{\ell+1}_{i_1,\cdots,i_{\ell+1}}$. We first will estimate the spatial projection onto the $i_k$-particle, defined as:
\begin{equation}
    S_{i_k}^{\ell+1}(x_{i_1}, \cdots,\hat{x}_{i_k},\cdots, x_{i_{\ell+1}}) = \{ x_{i_k} \in \mathbb{R}^d \, | \, (x_{i_1},\cdots,x_{i_k},\cdots,x_{i_{\ell+1}}) \in U^{\ell+1,X}_{i_1,\cdots,i_{\ell+1}} \}, \quad k\in \{1,\cdots, \ell+1 \},
\end{equation}
where $\hat{x}_{i_k}$ denotes that $x_{i_k}$ is excluded from the sequence. 
\begin{lemma}\label{projection}
Let $\ell \in \{1,\cdots, M\}$ and $(i_1,\cdots, i_{\ell+1}) \in I^{\ell+1}_m$. Then for any  $k\in\{1,...,\ell+1 \}$, we have the estimate:
\begin{equation}\label{projection estimate}
    |S^{\ell+1}_{i_k}(x_{i_1}, \cdots,\hat{x}_{i_k},\cdots, x_{i_{\ell+1}})|_d \leq C_{d,R} \delta.
\end{equation}
\end{lemma}
\begin{proof} 
Given $k\in\{1,...,\ell+1\}$, we have
\begin{equation}
S_{i_k}^{\ell+1}(x_{i_1}, \cdots,\hat{x}_{i_k},\cdots, x_{i_{\ell+1}}) = \{ x_{i_k} \in \mathbb{R}^d: \sigma_{\ell+1}^2\leq d^2_{\ell+1}(x_{i_1},\cdots,x_{i_{\ell+1}})\leq (\sigma_{\ell+1}+\delta_{\ell+1})^2\}.    
\end{equation}
By rearranging the inequalities above in order to isolate $x_{i_k}$, and  completing the square, we may write, 
\begin{equation}
S^{\ell+1}_{i_k}(x_{i_1}, \cdots,\hat{x}_{i_k},\cdots, x_{i_{\ell+1}}) = \bigg\{ x_{i_k} \in \mathbb{R}^d \, : \, \frac{\sigma_{\ell+1}^2 - \alpha^2_{\ell+1}}{\ell} \leq \bigg| x_{i_k} - \frac{\sum_{j=1, j\neq k}^\ell x_{i_j}}{\ell } \bigg|^2 \leq \frac{(\sigma_{\ell+1}+\delta_{\ell+1})^2 - \alpha^2_{\ell+1}}{\ell} \bigg\},
\end{equation}
where
\begin{equation}
    \alpha_{\ell+1}^2 = \sum_{j = 1, j\neq k}^{\ell+1} |x_{i_j}|^2 - \frac{1}{\ell}\left| \sum_{j = 1, j\neq k}^{\ell+1} x_{i_j} \right|^2 + \sum_{\substack{j,j'\in\{1,...,\ell+1\}\\  j,j'\neq k,j<j'}}|x_{i_j}-x_{i_{j'}}|^2.
\end{equation}
\textit{Case 1:} $\alpha_{\ell+1} > \sigma_{\ell+1}$. Recalling the scaling $0<\delta_{\ell+1} <<\sigma_{\ell+1} << 1$ we have,
\begin{equation}
    (\sigma_{\ell+1} + \delta_{\ell+1})^2 - \alpha_{\ell+1}^2 < (\sigma_{\ell+1} + \delta_{\ell+1} )^2 - \sigma_{\ell+1}^2 = \delta_{\ell+1} (2\sigma_{\ell+1} + \delta_{\ell+1} ) < \delta_{\ell+1}.
\end{equation}
This implies
\begin{equation}\label{projection estimate case 1}
    |S^{\ell+1}_{i_k}(x_{i_1}, \cdots,\hat{x}_{i_k},\cdots, x_{i_{\ell+1}})|_d \lesssim \big(4(\ell+1)^2R\big)^{d/2} \delta^{d/2} \leq C_{d,R} \delta.
\end{equation}
\textit{Case 2:} $\alpha_{\ell+1} \leq \sigma_{\ell+1}$.

\begin{align}
    |S^{\ell+1}_{i_k}&(x_{i_1}, \cdots,\hat{x}_{i_k},\cdots, x_{i_{\ell+1}})|_d \simeq \bigg( \sqrt{(\sigma^2_{\ell+1} + \delta_{\ell+1} )^2 - \alpha_{\ell+1}^2 } \bigg)^d - \bigg( \sqrt{\sigma_{\ell+1}^2 - \alpha_{\ell+1}^2}  \bigg)^d \nonumber\\
    &= \frac{\delta_{\ell+1} (2\sigma_{\ell+1}+ \delta_{\ell+1})}{ \sqrt{(\sigma^2_{\ell+1} + \delta_{\ell+1})^2 - \alpha_{\ell+1}^2 }  + \sqrt{\sigma_{\ell+1}^2 - \alpha_{\ell+1}^2}} \sum_{j=0}^{d-1}\bigg(  \sqrt{(\sigma^2_{\ell+1} + \delta_{\ell+1} )^2 - \alpha_{\ell+1}^2 } \bigg)^{d-1-j} \bigg( \sqrt{\sigma_{\ell+1}^2 - \alpha_{\ell+1}^2}  \bigg)^j\nonumber \\
    &\leq \frac{\delta_{\ell+1}}{ \sqrt{(\sigma^2_{\ell+1} + \delta_{\ell+1} )^2 - \alpha_{\ell+1}^2 }  + \sqrt{\sigma_{\ell+1}^2 - \alpha_{\ell+1}^2}}  \bigg( \sqrt{(\sigma^2_{\ell+1} + \delta_{\ell+1})^2 - \alpha_{\ell+1}^2 }  + (d-1) \sqrt{\sigma_{\ell+1}^2 - \alpha_{\ell+1}^2} \bigg) \nonumber\\
    &\leq (d-1)\delta_{\ell+1} \leq C_{d,R} \delta.\label{projection estimate case 2}
\end{align}
Combining \eqref{projection estimate case 1}-\eqref{projection estimate case 2}, estimate \eqref{projection estimate} follows.
\end{proof}
\par
We are now ready to estimate the covering sets $U_{\ell+1,k+1}$ for all $\ell, k \in \{1,\cdots, M \}$.
\begin{lemma}
For $\ell, k \in \{1,\cdots, M \}$ we have $|U_{\ell+1,k+1}|_{2dN}\leq C_{d,R,\rho} \delta^{2}$.
\end{lemma}
\begin{proof}
 It is sufficient to show,
 \begin{equation}
     |U^{\ell+1}_{i_1,\cdots, i_{\ell+1}} \cap U^{k+1}_{i'_1,\cdots, i'_{k+1}}| \leq C_{d,R,\rho} \delta^{2} \quad \text{for all } (i_1,\cdots,i_{\ell+1}) \neq (i'_1,\cdots, i_{k+1}'), \quad \ell,k\in\{1,...,M\}. 
 \end{equation}
 Without loss of generality, assume that $\ell \geq k$. Let $q = |(i_1,\cdots,i_{\ell+1}) \cap (i_1',\cdots,i_{k+1}')|$ denote the number of shared particles between the $\ell+1$ and $k+1$ collisions. Note that there must be at least one un-shared particle, so we have,
 \begin{equation}
 q \in
     \begin{cases}
         \{0,\cdots, k+ 1 \} \quad \text{ if } \ell \neq k\\
         \{0,\cdots, \ell \}\hspace{10mm}\text{ if } \ell = k.
     \end{cases}
 \end{equation} 
 For ease of notation, assume without loss of generality that the $i_1,\cdots,i_{\ell+1}$ and $i_1', \cdots, i_{k+1}'$ particles are ordered such that,
 \begin{equation}
     i_1 < i_2 < \cdots < i_{\ell+ 1-q} < i_1' < \cdots < i_q' < \cdots i_{k+1}'
 \end{equation}
 where the particles $(i_1', \cdots i_q') = (i_{\ell-q+2},\cdots,i_{\ell+1})$ are the $q$ shared particles between the $\ell+1$ and $k+1$ collisions. By isolating and integrating over one of the un-shared particles, using Lemma \ref{projection} we can make the following estimates, 
 \begin{equation*}
 \begin{split}
     &|U^{\ell+1}_{i_1,\cdots,i_{\ell+1}} \cap U_{i_1',\cdots,i_{k+1}'}^{k+1}|_{2dm} \lesssim \\
     &\lesssim R^{dm} \rho^{d \tilde{m}} \int_{B_\rho^{(\ell+k+2-q)d}}
     \mathds{1}_{S^{\ell+1}_{i_1}(i_2,\cdots,i_{\ell +1 -q},i_1',\cdots,i_q') \cap S^{k+1}_{i_1'}(i_2',\cdots, i_{k+1}')}dx_{i_1}\cdots dx_{i_{\ell+1-q}}\cdot dx_{i_1'}\cdots dx_{i_{k+1}'}\\
     &\leq R^{dm} \rho^{d\tilde{m}} \int_{B_\rho^{(\ell+k-q+1)d}} \bigg( \int_{\mathbb{R}^d}
     \mathds{1}_{S^{\ell+1}_{i_1}(i_2,\cdots,i_{\ell+1-q},i_1', \cdots, i_q')} \mathds{1}_{S^{k+1}_{i_1'}(i_2',\cdots, i_{k+1}')}  dx_{i_1}\bigg) dx_{i_2}\cdots dx_{i_{\ell+1-q}}\cdot dx_{i_1'}\cdots dx_{i_{k+1}'}\\
     &= R^{dm} \rho^{d\tilde{m}} \int_{B_\rho^{(\ell+k+1-q)d}}  \mathds{1}_{S^{k+1}_{i_1'}(i_2',\cdots, i_{k+1}')}\bigg( \int_{\mathbb{R}^d} \mathds{1}_{S^{\ell+1}_{i_1}(i_2,\cdots,i_{\ell 1-q},i_1',\cdots,i_q)}   dx_{i_1}\bigg) dx_{i_2}\cdots dx_{i_{\ell+1-q}}\cdot dx_{i_1'}\cdots dx_{i_{k+1}'} \\
     &\leq C_{d,R} \rho^{d\tilde{m}} \delta \int_{B_\rho^{(\ell+k-q)d}} \bigg(\int_{\mathbb{R}^d}\mathds{1}_{S^{k+1}_{i_1'}(i_2',\cdots, i_{k+1}')} dx_{i_1'} \bigg) dx_{i_2} \cdots dx_{i_{\ell+1 - q}} \cdot dx_{i_2'} \cdots dx_{i_{k+1}'}\\
     &\leq C_{d,R} \rho^{d(m-2)}\delta^2 \leq C_{d,R,\rho} \delta^2.
 \end{split}
 \end{equation*}
 where $\tilde{m} = m - (\ell + k + 2 - m)$ represents the remainder of unique particles involved in the $\ell+1 $ and $k+1$ interactions.
\end{proof}
Now we are in the position to state the Existence Theorem for the $m$-particle $(\sigma_{\ell+1})_{\ell =1}^M$-flow. For a given $Z_m \in \R^{2dm}$, let us define its kinetic energy by
\begin{equation}\label{kinetic energy}
    E_m(Z_m) = \frac{1}{2} \sum_{i=1}^m |v_i|^2.
\end{equation}
\begin{theorem}\label{existence global flow}
Let $m\in \N$. There exists a family of measure-preserving maps $(\Psi^t_m)_{m\in \N}: \Dsig \rightarrow \Dsig$ such that
\begin{equation} \label{interaction flow}
    \begin{split}
        &\Psi^{t+s}_m Z_m = (\Psi^t_m \circ \Psi^s_m)(Z_m) = (\Psi^s_m \circ \Psi^t_m)(Z_m), \quad a.e. \text{ in } \Dsig, \quad \forall t,s\in \R, \\
        &E_m(\Psi^t_m Z_m) = E_m(Z_m), \quad a.e. \text{ in } \Dsig, \quad \forall t\in \R.
    \end{split}
\end{equation} 
Additionally, for all $\ell \in \{1,\cdots, M\}$ and when $m\geq \ell+1$ we have,
\begin{equation} \label{measure preserving bdry}
    \begin{split}
        &\Psi^t_m Z_m^{*,\ell+1} = \Psi^t_m Z_m, \quad \sigma-a.e \text{ on } \partial_{\ell+1}^{sc,ng} \Dsig, \quad \forall t \in \R,
    \end{split}
\end{equation}
where $\sigma$ is the surface measure induced on $\partial \Dsig$ and $Z^{*,\ell+1}_m$ is defined in Definition \ref{z* defn}.
\end{theorem}
\begin{proof}
For a proof see for example \cite{AmpatzoglouThesis}.

\end{proof}

\subsection{The Liouville Equation}
Upon establishing a global in time flow for  almost any initial configuration, we are now able  to derive the Liouville equation for our system. Let $m \geq M+1$. We first will define the $m$-particle flow operator and $m$-particle free-flow operator. We define the space of functions $L^{\infty}(\Dsig)$ by,
\begin{equation}
    L^{\infty}(\Dsig) = \{ g_N \in L^{\infty}(\R^{2dN}) : \text{supp}(g_N) \subset \Dsig \}. 
\end{equation}

We now define the $m$-particle flow operator $T^t_m:  L^{\infty}(\R^{2dm}) \rightarrow  L^{\infty}(\R^{2dm}) $,
\begin{equation} \label{flow operator}
    T^t_m g_m(Z_m) = 
    \begin{cases}
        g_m(\Psi^{-t}_m Z_m), \quad \text{ if } Z_m \in \Dsig \\
        0, \quad \text{ if } Z_m \notin \Dsig,
    \end{cases}
\end{equation}
where $\Psi_m$ is defined in Theorem \ref{existence global flow}.
We also define the $m$-particle free flow as the family of measure-preserving maps $(\Phi^t_m)_{t\in \R}:\R^{2dm}\rightarrow \R^{2dm}$ given by,
\begin{equation}\label{free flow}
    \Phi^t_m Z_m = \Phi^t(X_m,V_m) = (X_m + tV_m, V_m).
\end{equation}
Similarly, we define the $m$-particle free flow operator $S^t_m:  L^{\infty}(\R^{2dm}) \rightarrow  L^{\infty}(\R^{2dm})$ by,
\begin{equation}\label{free flow operator}
    S^t_m g_m (Z_m) = g_m(\Phi^{-t}_m Z_m) = g_m(X_m-t V_m,V_m).
\end{equation}
\par
Now, we consider the absolutely continuous Borel probability measure $P_0$ on $\R^{2dm}$ with a probability density function $f_{m,0}$ satisfying,
\begin{itemize}
    \item $f_{m,0}$ is supported in $\Dsig$,
    \item $f_{m,0}$ is symmetric with respect to the labelling of the $m$ particles.
\end{itemize}
\par
$P_0$ represents the initial distribution of the system. We wish to study the evolution of this measure under the $m$-particle $(\sigma_{\ell+1})_{\ell =1}^M$-flow. We denote the push-forward of $P_0$ as $P_t$,
\begin{equation}
    P_t(A) = P_0\big( \Psi^{-t}_m(A) \big), \quad A \subset \R^{2dm} \text{ Borel measurable.}
\end{equation}
The conservation of measure under the flow implies that $P_t$ is absolutely continuous with probability density given by, 
\begin{equation}
    f_m(t,Z_m) = 
\begin{cases}
f_{m,0} \circ \Psi^{-t}_m(Z_m), \quad a.e \text{ in } \Dsig, \\
0, \quad a.e \text{ in } \R^{2dm} \setminus \Dsig.
\end{cases}
\end{equation}
\par 
We have, 
\begin{equation}
    f_m(0,Z_m) = f_{m,0} \circ \Psi^0_m(Z_m) = f_{m,0}(Z_m), \quad Z_m \in \mathring{\mathcal{D}}_{m,\bm{\sigma_{M}}}.
\end{equation}
For post-collisional velocities we also have for all $\ell \in \{ 1, \cdots, M\}$, 
\begin{equation}
    f_{m}(t,Z_m^{*,\ell+1}) = f_{m,0} \circ \Psi^{-t}_m(Z_m^{*,\ell+1}) = f_m(t,Z_m), \quad \sigma - a.e. \text{ on } \partial_{\ell+1}^{sc,ng} \Dsig, \quad \forall t \geq 0.
\end{equation}
\par 
By the chain rule we obtain that $f_m$ satisfies the $m$-particle Liouville equation in $\Dsig$ given by,

\begin{equation}\label{lville}
	\begin{cases}
	\partial_t f_m + \sum_{i=1}^m v_i \cdot \nabla_{x_i} f_m = 0 \quad (t,Z_m) \in (0,\infty) \times \mathring{\mathcal{D}}_{m,\bm{\sigma_{M}}}  \\
	f_m(t,Z^{*,\ell+1}_m) = f_m(t,Z_m), \quad (t,Z_m) \in [0,\infty) \times \partial_{\ell+1}^{sc} \Dsig \quad \forall \ell \in \{ 1, \cdots, M\} \\
	f_{m,0}(0,Z_m) = f_{m,0}(Z_m), \quad Z_m \in \mathring{\mathcal{D}}_{m,\bm{\sigma_{M}}}.
	\end{cases} 
\end{equation}

\section{BBGKY and Boltzmann hierarchies}\label{sec::hier}
In this section we provide details of the calculation of the BBGKY hierarchy from the Liouville equation of the finite system given in \eqref{lville} and define the analogous Boltzmann hierarchy for the infinite system.  We go on to prove local well-posedness of these two hierarchies in Section \ref{sec:local} and state our main convergence result in Section \ref{sec_conv statement}.
\subsection{The BBGKY hierarchy.}
For $N\geq M+1$, we consider an $N$-particle system with interaction zones $\bm{\epsilon_M} = (\epsilon_2,\cdots, \epsilon_{M+1})$ such that $0 < \epsilon_2 << \cdots << \epsilon_{M+1} << 1$. For $s \in \N$, the $s$-marginal of the probability density function $f_N$ is given by
\begin{equation}	
	f^{(s)}_N(Z_s)
	=
	\begin{cases}
	\int_{ \mathbb{R}^{d ( N -s )} }  f_N (Z_N) dx_{s+1} \cdots dx_{N} d v_{s + 1 }\cdots d v_{N} \,  &  s < N  \\
	f_N & s =  N  \\ 
	0 & s > N 
	\end{cases} 
	\end{equation}
 where $Z_s = (X_s,V_s) \in \mathbb{R}^{2ds}$. Following Definition \ref{z* defn}, $Z_s^{*,\ell+1}$ denotes an $\ell$+1-order post-collisional velocity for $\ell \in \{ 1,\cdots, \}$. \par
 Starting from Liouville's equation we have,
 
 	\begin{equation}	
	\begin{cases}
	\partial_t f_N + \sum_{i=1}^N v_i \cdot \nabla_{x_i} f_N = 0 \quad (t,Z_N) \in (0,\infty) \times \mathring{\mathcal{D}}_{N,\bm{\epsilon_M}}  \\
	f_N(t,Z^{*,\ell+1}_N) = f_N(t,Z_N), \quad (t,Z_N) \in [0,\infty) \times \partial_{\ell+1}^{sc} \DN \quad \forall \ell \in \{ 1, \cdots, M\} \\
	f_{N,0}(0,Z_N) = f_{N,0}(Z_N), \quad Z_N \in \mathring{\mathcal{D}}_{N,\bm{\epsilon_M}}.
	\end{cases} 
	\end{equation}
\par
Consider a smooth test function $\phi_s$ compactly supported in $(0,\infty) \times {\mathcal{D}}_{N,\bm{\epsilon_M}}$ such that for all $\ell \in \{1,\cdots, M\}$ and $(i_1,\cdots,i_{\ell+1}) \in {I}^{\ell+1}_N$ with $i_{\ell} \leq s$ we have
\begin{equation}
    \phi_s(t,p_s (Z_N^{*,\ell+1})) = \phi_s(t,p_s ( Z_N)) = \phi_s(t,Z_s) \quad \forall (t,Z_N) \in (0,\infty) \times \Sigma^{sc,X}_{i_1,\cdots,i_{\ell+1}},
\end{equation}
where $p_s$ denotes the projection onto the first $s$-particles.
Multiplying $\phi_s$ by the Liouville equation and integrating we have
\begin{equation}
    \int_{(0,\infty)\times \DN} \bigg( \partial_t f_N (t,Z_N) + \sum_{i=1}^{N} v_i \cdot \nabla_{x_i}f_N(t,Z_N) \bigg) \phi_s(t,Z_s) dX_N dV_N dt = 0.
\end{equation}
This leads to the following BBGKY hierarchy. For details on these calculations see \cite{WarnerThesis}.
\begin{equation}	
 	\begin{cases}
 	\partial_t f_N^{(s)} + \sum_{i=1}^s v_i \cdot \nabla_{x_i} f_N^{(s)} &= \sum_{\ell = 1}^M \mathcal{C}^N_{s,s+\ell}f_N^{(s+\ell)} \quad (t,Z_s) \in (0,\infty)\times {\mathcal{D}}_{s,\bm{\epsilon_M}} \\
 	f^{(s)}_N(t,Z^{*,\ell+1}_s) &= f_N^{(s)}(t,Z_s) \quad (t,Z_s) \in [0,\infty) \times \partial_{\ell+1}^{sc} \mathcal{D}_{s,\bm{\epsilon_M}} \quad \forall \ell \in \{1,\cdots, M \}\\ 
 	f_N^{(s)}(0,Z_s) &= f_{N,0}^{(s)}(Z_s) \quad Z_s \in {\mathcal{D}}_{s,\bm{\epsilon_M}},
\end{cases} 
\end{equation}
where
\begin{equation}\label{BBGKY heir}
    \mathcal{C}^N_{s,s+\ell} = \mathcal{C}^{N,+}_{s,s+\ell} - \mathcal{C}^{N,-}_{s,s+\ell}
\end{equation}
and for  all $\ell \in \{1,\cdots,M\}$
 \begin{align}
     \mathcal{C}_{s,s+\ell}^{N,+} &= A^\ell_{N,\epsilon_{\ell+1}, s} \sum_{i=1}^s  \int_{\E\times \mathbb{R}^{\ell d}}  b_+(\bm{\omega}, v_{s+1}-v_{i}, \cdots,v_{s+\ell}- v_{i})  f_N^{(s+\ell)}(t,Z^{i,*,\ell+1}_{s+\ell,\epsilon_{\ell+1}}) d\bm{\omega} d v_{s+1} \cdots d v_{s+\ell} \\
         \mathcal{C}_{s,s+\ell}^{N,-} &= A^\ell_{N,\epsilon_{\ell+1}, s} \sum_{i=1}^s  \int_{\E\times \mathbb{R}^{\ell d}}  b_+(\bm{\omega}, v_{s+1}-v_{i}, \cdots,v_{s+\ell}- v_{i})  f_N^{(s+\ell)}(t,Z^{i}_{s+\ell,\epsilon_{\ell+1}}) d\bm{\omega} d v_{s+1} \cdots d v_{s+\ell}
\end{align}
\begin{equation}\label{A defn}
    A^\ell_{N,\epsilon_{\ell+1},s} = \epsilon_{\ell+1}^{\ell d-1}  {N-s \choose \ell}.
\end{equation}
Here, we use the notation
\begin{equation}\label{Zi defn}
\begin{split}
    Z_{s+\ell,\epsilon_{\ell+1}}^{i,*,\ell+1} &= (X_s,x_{i}+\epsilon_{\ell+1} \omega_1,\cdots, x_{i}+ \epsilon_{\ell+1} \omega_\ell, v_1,\cdots,v_{i - 1}, v_{i}^*,v_{i + 1}, \cdots, v_s,v_{s+1}^*,\cdots,v_{s+\ell }^*  )\\
    Z_{s+\ell,\epsilon_{\ell+1}}^{i} &= (X_s,x_{i}-\epsilon_{\ell+1} \omega_1,\cdots, x_{i}- \epsilon_{\ell+1} \omega_\ell, v_1,\cdots,v_s, v_{s+1},\cdots,v_{s+\ell}  ).
\end{split}
\end{equation}

\subsection{The Boltzmann hierarchy}
Formally, we derive the Boltzmann hierarchy as the limit of the BBGKY hierarchy as $N \rightarrow \infty$ and $\epsilon_{\ell+1} \rightarrow 0^+$ for all $\ell \in \{1,\cdots, M\}$ under the scaling
\begin{equation} \label{scaling}
    N\epsilon_{2}^{d-1} \simeq N\epsilon_{3}^{d-\frac{1}{2}} \simeq \cdots \simeq N\epsilon_{M+1}^{d-\frac{1}{M}} \simeq 1
\end{equation}
to ensure that $A^\ell_{N,\epsilon_{\ell+1},s} \rightarrow \frac{1}{\ell!}$. By this scaling, the interaction zones $\{\epsilon_{\ell+1} \}_{\ell = 1}^M$ satisfy
\begin{equation}
        \epsilon_{2}^{d-1} \simeq \epsilon_{3}^{d-\frac{1}{2}} \simeq \cdots \simeq \epsilon_{M+1}^{d-\frac{1}{M}}
\end{equation}
\par
For initial data $f_0^{(s)}$ the Boltzmann hierarchy is given by,
\begin{equation}	
 \begin{cases}
 	\partial_t f^{(s)} + \sum_{i=1}^s v_i \cdot \nabla_{x_i} f^{(s)} = \sum_{\ell = 1}^M \mathcal{C}^\infty_{s,s+\ell}f^{(s+\ell)} \quad (t,Z_s) \in (0,\infty)\times \mathbb{R}^{\ell ds} \\
 	f^{(s)}(0,Z_s) = f^{(s)}_0 (Z_s) \quad \forall Z_s \in \mathbb{R}^{\ell ds},
\end{cases} 
\end{equation}
where
\begin{equation} \label{Boltzmann heir}
    \mathcal{C}^\infty_{s,s+\ell} = \mathcal{C}^{\infty,+}_{s,s+\ell} - \mathcal{C}^{\infty,-}_{s,s+\ell}
\end{equation}
and 
 \begin{align}
     \mathcal{C}_{s,s+\ell}^{\infty,+} &= \frac{1}{\ell!}  \sum_{i=1}^s \int_{\E\times \mathbb{R}^{\ell d}}  b_+(\bm{\omega}, v_{s+1}-v_{i}, \cdots,v_{s+\ell}- v_{i})  f^{(s+\ell)}(t,Z^{i,*,\ell+1}_{s+\ell}) d\bm{\omega} d v_{s+1} \cdots d v_{s+\ell}\\
         \mathcal{C}_{s,s+\ell}^{\infty,-} &=  \frac{1}{\ell!}\sum_{i=1}^s \int_{\E\times \mathbb{R}^{\ell d}}  b_+(\bm{\omega}, v_{s+1}-v_{i}, \cdots,v_{s+\ell}- v_{i})  f^{(s+\ell)}(t,Z^{i}_{s+\ell}) d\bm{\omega} d v_{s+1} \cdots d v_{s+\ell}
\end{align}
where
\begin{equation}\label{Z inf hier}
\begin{split}
    Z^{i,*,\ell+1}_{s+\ell} &= (X_s, \underbrace{x_{i},\cdots,x_{i}}_{\ell-\text{copies}},v_1,\cdots,v_{i-1},v_{i}^*, v_{i+1},\cdots, v_s,v_{s+1}^*,\cdots,v_{s+\ell}^* ) \\
    Z^{i}_{s+\ell } &= (X_s,\underbrace{x_{i},\cdots,x_{i}}_{\ell-\text{copies}},V_s,v_{s+1},\cdots,v_{s+ \ell }).
\end{split}
\end{equation}

\subsection{The Boltzmann Equation}
In the case of independent initial data, the marginals take the factorized form,
\begin{equation}
    f_0^{(s)}(Z_s) := f_0^{\otimes s}(Z_s), \quad s \in \mathbb{N},
\end{equation}
for a given $f_0 \in L^1(\mathbb{R}^d)$, where we define
\begin{equation}
    f_0^{\otimes s}(Z_s) := \prod_{j=1}^s f_0(x_j,v_j).
\end{equation}
In this case, the solution to the Boltzmann hierarchy is given by
\begin{equation}
    f^{(s)}(t, Z_s) = f^{\otimes s}(t,Z_s) = \prod_{j=1}^s f(t,x_j,v_j) , \quad s \in \mathbb{N},
\end{equation}
where $f: [0, \infty) \times \mathbb{R}^d \times \mathbb{R}^d$ solves the $M+1$-nary Boltzmann equation,
\begin{equation}	
 \begin{cases}
 	\partial_t f + v \cdot \nabla_x f  = \sum_{\ell = 1}^M \frac{1}{\ell!} Q_{\ell+1} (f,\cdots ,f), \quad (t,x,v) \in (0,\infty)\times \mathbb{R}^{2d}, \\
 	f(0,x,v) = f_0 (x,v) \quad (x,v) \in \mathbb{R}^{2d},
\end{cases} 
\end{equation}
where, for $\ell = 1,\cdots, M $, the $\ell+1$-nary collisional operator $Q_{\ell+1}$ is given by
\begin{equation} \label{Boltzmann equation}
    Q_{\ell+1}(f,\cdots,f)(t,x,v) =  \int_{S_1^{\ell d-1}\times \mathbb{R}^{\ell d}} (f^{*}\cdot f^{*}_1 \cdots f^{*}_{\ell} - f \cdot  f_1 \cdots f_{\ell}) b^+_{\ell+1}(\bm{\omega}, v_1 - v, \cdots, v_\ell - v)  d\bm{\omega} dv_1 \cdots dv_{\ell},
\end{equation}
where,
\begin{equation}
\begin{split}
    &f^{*} = f(t,x,v^{*}), \quad f = f(t,x,v), \\
    &f^{*}_i = f(t,x,v^{*}_i), \quad f_i = f(t,x,v_i) \quad i\in \{1,\cdots, \ell \}.
\end{split}
\end{equation}

\section{Local well-posedness}\label{sec:local} 
We prove local well-posedness for short times in Maxwellian weighted $L^\infty$ spaces for the $M+1$-order BBGKY hierarchy, Boltzmann hierarchy and Boltzmann equation. 
 \subsection{LWP for the BBGKY hierarchy}\label{sub BBGKY well posedness}
 We first will define the spaces for which our hierarchies are well-posed.
  Consider $(N,\epsilon_2, \cdots, \epsilon_{M+1})$ in the scaling \eqref{scaling}, with $N\geq M+1$.  
  
  For $\beta> 0$ we define  the Banach space
 \begin{equation} \label{norm}
 X_{N,\beta,s}
 :=\left\{g_{N,s}\in L^\infty(\D)\text{ and }  |g_{N,s}|_{N,\beta,s}<\infty\right\},
 \end{equation}
with norm
$|g_{N,s}|_{N,\beta,s}=\sup_{Z_s\in\mathbb{R}^{2ds}}|g_{N,s}(Z_s)|e^{\beta E_s(Z_s)},$
where $E_s(Z_s)$ is the kinetic energy of the $s$-particles given by \eqref{kinetic energy}.
For $s>N$ we trivially define  
$X_{N,\beta,s}:=\left\{0\right\}.
$
\par
Consider as well $\mu\in\mathbb{R}$. We define the Banach space 
\begin{equation} \label{norm2}
X_{N,\beta,\mu}:=\left\{G_N=(g_{N,s})_{s\in\mathbb{N}}:\|G_N\|_{N,\beta,\mu}<\infty\right\},
\end{equation}
with norm
$\|G_N\|_{N,\beta,\mu}=\sup_{s\in\mathbb{N}}e^{\mu s}|g_{N,s}|_{N,\beta,s}=\max_{s\in\{1,...,N\}}e^{\mu s}|g_{N,s}|_{N,\beta,s}.$ 
\par
As in \cite{AmpatzoglouPavlovic2020}, the $s$-particle $(\sigma_2,\cdots, \sigma_{M+1})$-flow operator $T_s^t:X_{N,\beta,s}\to X_{N,\beta,s}$ given by \ref{flow operator} is an isometry with respect to the norm \ref{norm}. Likewise, the map $\mathcal{T}^t:X_{N,\beta,\mu}\to X_{N,\beta,\mu}$ given by
\begin{equation}\label{T_N definition}
\mathcal{T}^tG_N:=\left(T_s^tg_{N,s}\right)_{s\in\mathbb{N}},
\end{equation}
is an isometry with respect to the norm \ref{norm2}.
\par
Finally, we define the time-dependent spaces. Given $T>0$, $\beta_0> 0$, $\mu_0\in\mathbb{R}$ and $\bm{\beta},\bm{\mu}:[0,T]\to\mathbb{R}$ decreasing functions of time with $\bm{\beta}(0)=\beta_0$, $\bm{\beta}(T)> 0$, $\bm{\mu}(0)=\mu_0$, we define the Banach space 
\begin{equation*}
\bm{X}_{N,\bm{\beta},\bm{\mu}}:=C^0\left([0,T],X_{N,\bm{\beta}(t),\bm{\mu}(t)}\right),
\end{equation*}
with norm
$|||\bm{G_N}|||_{N,\bm{\beta},\bm{\mu}}=\sup_{t\in[0,T]}\|\bm{G_N}(t)\|_{N,\bm{\beta}(t),\bm{\mu}(t)}.$
\par
Similarly as  in Proposition 6.2. from \cite{AmpatzoglouThesis}, one can obtain the bounds:
\begin{proposition}\label{remark for initial} Let $T>0$, $\beta_0>0$, $\mu_0\in\mathbb{R}$ and $\bm{\beta},\bm{\mu}:[0,T]\to\mathbb{R}$ decreasing functions with $\beta_0=\bm{\beta}(0)$, $\bm{\beta}(T)> 0$ $\mu_0=\bm{\mu}(0)$. Then for any $G_N=\left(g_{N,s}\right)_{s\in\mathbb{N}}\in X_{N,\beta_0,\mu_0}$, the following estimates hold:
\begin{enumerate}
\item $|||G_N|||_{N,\bm{\beta},\bm{\mu}}\leq\|G_N\|_{N,\beta_0,\mu_0}$.\vspace{0.2cm}
\item $\left|\left|\left|\displaystyle\int_0^t\mathcal{T}^{\tau}G_N\,d\tau\right|\right|\right|_{N,\bm{\beta},\bm{\mu}}\leq T\|G_N\|_{N,\beta_0,\mu_0}.$
\end{enumerate}
\end{proposition}
By extending Lemma 5.1. in \cite{ternary}, we can derive the continuity estimates for the $M+1$-nary collisional operators respectively:
\begin{lemma}\label{BBGKY C est}
Let $s\in\mathbb{N}$, $\beta>0$.  For any $Z_s\in \D$ and $\ell\in\{1,\cdots,M\}$, the following  estimate holds:
\begin{equation*}
\left|\mathcal{C}_{s,s+\ell}^{N}g_{N,s+\ell}(Z_s)\right| \leq C_d  A^\ell_{N,\epsilon_{\ell+1}, s} \ell^3 \beta^{\frac{-\ell d}{2}}\bigg( \sum_{i=1}^s |v_i| + s \beta^{-\frac{1}{2}}   \bigg)e^{-\beta E_s(Z_s)} |g_{N,s+\ell}|_{N,\beta, s+\ell},\quad\forall g_{N,s+\ell}\in X_{N,\beta,s+\ell}.
\end{equation*}
\end{lemma}
\begin{proof}
We first note that by the Cauchy-Schwarz inequality we have
\begin{equation} \label{b+ est}
\begin{split}
    b_+(\bm{\omega}, v_2 - v_1, \cdots, v_{\ell+1} - v_1) &\leq \sum_{i = 1}^\ell |\omega_i| |v_{i+1}- v_1| + \sum_{1\leq i < j \leq \ell} |\omega_i - \omega_j|  |v_{i+1} - v_{j+1}| \\
    &\leq 4\ell^2 \sum_{i  =1}^{\ell+1} |v_i|,
    \end{split}
\end{equation}
since $\bm{\omega}$ is on the ellipsoid $\E$.
\par
Now, recalling \ref{BBGKY heir} and the scaling given in \ref{scaling} we arrive at the calculation,
\begin{align}
    |\mathcal{C}_{s,s+\ell}^{N}g_{N,s+\ell}(Z_s)| &\leq 2 A^\ell_{N,\epsilon_{\ell+1}, s}e^{-\beta E_s(Z_s)} |g_{N,s+\ell}|_{N,\beta, s+\ell} \\
    & \times \sum_{i=1}^s \int_{\mathbb{S}_1^{\ell d -1}\times \R^{\ell d}} b_+(\bm{\omega}, v_{s+1} - v_1, \cdots, v_{s+\ell} - v_1) e^{-\frac{\beta}{2}\sum_{k=1}^{\ell}|v_{s+k}|^2} d\bm{\omega} dv_{s+1} \cdots dv_{s+\ell} \\
    &\leq  2 A^\ell_{N,\epsilon_{\ell+1}, s}e^{-\beta E_s(Z_s)} |g_{N,s+\ell}|_{N,\beta, s+\ell}\\
    &\times 4 \ell^2 \sum_{i=1}^s  \int_{ \R^{\ell d}} (|v_i| + |v_{s+1}| + \cdots + |v_{s+\ell}|) e^{-\frac{\beta}{2}\sum_{k=1}^{\ell}|v_{s+k}|^2} dv_{s+1} \cdots dv_{s+\ell} \\
    &\leq 8 A^\ell_{N,\epsilon_{\ell+1}, s} \ell^2 e^{-\beta E_s(Z_s)} |g_{N,s+\ell}|_{N,\beta, s+\ell} \\
    &\times \int_{ \R^{\ell d}} \bigg( \sum_{i=1}^s |v_i| + s (|v_{s+1}| + \cdots + |v_{s+\ell}|)  \bigg) e^{-\frac{\beta}{2}\sum_{k=1}^{\ell}|v_{s+k}|^2} dv_{s+1} \cdots dv_{s+\ell} \\
    & \leq C_d  A^\ell_{N,\epsilon_{\ell+1}, s}\ell^3 \beta^{\frac{-\ell d}{2}}\bigg(  \sum_{i=1}^s |v_i| + s \beta^{-\frac{1}{2}}   \bigg) e^{-\beta E_s(Z_s)} |g_{N,s+\ell}|_{N,\beta, s+\ell}  ,
\end{align}
where we obtain the last step of the above estimate with the known integrals, 
\begin{equation}
    \begin{split}
    &\int_0^\infty e^{\frac{-\beta}{2}x^2 }dx \simeq \beta^{-\frac{1}{2}}, \\
    &\int_0^\infty x e^{\frac{-\beta}{2}x^2 }dx \simeq \beta^{-1},
    \end{split}
    \end{equation}
which we use to make the estimate
\begin{equation}
\int_{\R^d} |v| e^{-\frac{\beta}{2}|v|^2 } dv \lesssim \beta^{-d+1} \lesssim \beta^{-\frac{d+1}{2}}.
\end{equation}

\end{proof}
Let us now define mild solutions to the BBGKY hierarchy:
\begin{definition}\label{def of mild bbgky} Consider $T>0$, $\beta_0> 0$, $\mu_0\in\mathbb{R}$ and the decreasing functions $\bm{\beta},\bm{\mu}:[0,T]\to\mathbb{R}$ with $\bm{\beta}(0)=\beta_0$, $\bm{\beta}(T)> 0$, $\bm{\mu}(0)=\mu_0$.
Consider also  initial data $G_{N,0}=\left(g^{(s)}_{N,0}\right)\in X_{N,\beta_0,\mu_0}$. A map $\bm{G_N}=\left(g^{(s)}_{N}\right)_{s\in\mathbb{N}}\in\bm{X}_{N,\bm{\beta},\bm{\mu}}$ is a mild solution of the BBGKY hierarchy in $[0,T]$, with initial data $G_{N,0}$, if it satisfies:
\begin{equation}
\bm{G_N}(t)=\mathcal{T}^tG_{N,0}+\int_0^t \mathcal{T}^{t-\tau}\mathcal{C}_N\bm{G_N}(\tau)\,d\tau,
\end{equation}
where, given $\beta>0$, $\mu\in\mathbb{R}$ and $G_{N}=(g^{(s)}_{N})_{s\in\mathbb{N}}\in X_{N,\beta,\mu}$, we write
\begin{align}\label{C_N defn}
&\mathcal{C}_N G_N:=(\mathcal{C}_N^2 +\cdots +\mathcal{C}_N^{M+1}) G_N, \\
&\mathcal{C}_N^{\ell+1} G_N:=\left(\mathcal{C}_{s,s+\ell}^N g^{(s+\ell)}_{N}\right)_{s\in\mathbb{N}} \quad \ell \in \{ 1, \cdots, M\},
\end{align}
where $\mathcal{T}^t$ is given by \eqref{T_N definition}.
\end{definition}
\begin{remark}\label{rmk mild bbgky}
    Note that we are using an abuse of notation in Definition \ref{def of mild bbgky}. The collision operators $\mathcal{C}_{s,s+\ell}^N$ are ill-defined on $L^\infty$ because they involve integration over a measure zero set, the ellipsoid $\E$. However, by filtering the BBGKY hierarchy by the flow $T^{-t}_s$ we obtain a well-defined operator. For details of this method see the erratum of Chapter 5 in \cite{Gallagher}. 
\end{remark}
Using Lemma \ref{BBGKY C est}, we obtain the following a-priori bounds:
\begin{lemma}\label{BBGKY T est} Let $\beta_0> 0$, $\mu_0\in\mathbb{R}$, $T>0$ and $\lambda\in (0,\beta_0/T)$. Consider the functions $\bm{\beta}_\lambda,\bm{\mu}_\lambda:[0,T]\to\mathbb{R}$ given by
 \begin{equation}\label{beta_lambda-mu_lambda}
 \begin{aligned}
 \bm{\beta}_\lambda(t)=\beta_0-\lambda t,\quad\bm{\mu}_\lambda(t)=\mu_0-\lambda t.
 \end{aligned}
\end{equation}  
 Then  for any $\mathcal{F}(t)\subseteq [0,t]$ measurable,  $\bm{G_N}=\left(g_{N,s}\right)_{s\in\mathbb{N}}\in\bm{X}_{N,\bm{\beta}_\lambda,\bm{\mu}_\lambda}$ and $\ell\in\{1,\cdots, M\}$ the following bounds hold:
 \begin{align}
\left|\left|\left|\displaystyle\int_{\mathcal{F}(t)}\mathcal{T}^{t-\tau}\mathcal{C}_N^{\ell+1}\bm{G_N}(\tau)\,d\tau\right|\right|\right|_{N,\bm{\beta}_\lambda,\bm{\mu}_\lambda}&\leq A^\ell_{N,\epsilon_{\ell+1}} C_{\ell+1}|||\bm{G_N}|||_{N,\bm{\beta}_\lambda,\bm{\mu}_\lambda},\label{both bbkky with constant}
\end{align}
 \begin{align}
 C_{\ell+1}=C_{\ell+1}(d,\beta_0,\mu_0,T,\lambda)&=C_d  \ell^3 \lambda^{-1}e^{-\ell \bm{\mu}_\lambda(T)}\bm{\beta}_\lambda^{-\ell d/2}(T)\left(1+\bm{\beta}_{\lambda}^{-1/2}(T)\right)\label{constant of WP binary}.
 \end{align}
\end{lemma}
\begin{proof}
This proof follows the estimate in Lemma \ref{BBGKY C est} with similar calculations as done in \cite{AmpatzoglouThesis}.
 \end{proof}
 \begin{remark}
     When $\lambda = \frac{\beta_0}{2T}$, the constant $C_{\ell+1}$ becomes,
     \begin{align}
    C_{\ell+1}=C_d \ell^3 \frac{2T}{\beta_0} e^{-\ell (\mu_0 - \beta_0/2))} \left(\frac{\beta_0}{2}\right)^{-\ell d/2}\left(1+\left(\frac{\beta_0}{2}\right)^{-1/2}\right).
 \end{align}
 Keeping in mind \eqref{A defn}, and in particular according to the scaling \eqref{scaling}
 \begin{align}
     A^\ell_{N,\epsilon_{\ell+1}} \nearrow \frac{1}{\ell !},
 \end{align}
 we have the following bound for $A^\ell_{N,\epsilon_{\ell+1}} C_{\ell+1}$,
 \begin{equation}
    A^\ell_{N,\epsilon_{\ell+1}} C_{\ell+1}\leq C_{d,\beta_0} \frac{1}{\ell!} \ell^3 e^{\ell (| \mu_0 |+ \beta_0/2))}T.
 \end{equation}
 By setting the constant, 
 \begin{equation}
     C_{d,\beta_0, \mu_0} := C_{d,\beta_0} 
 \max_{\ell \in \N} \{ 2^\ell \ell^3 e^{\ell (| \mu_0 |+ \beta_0/2))} \, : \, 2^\ell \ell^3 e^{\ell (| \mu_0 |+ \beta_0/2))} > \ell!\},
 \end{equation}
 we have the following bound,
 \begin{equation}
     2^\ell A^\ell_{N,\epsilon_{\ell+1}}  C_{\ell+1} \leq C_{d,\beta_0,\mu_0} \frac{T}{\ell!} .
 \end{equation}
 Therefore, the sum of $C_{\ell+1}$ can be bounded independent of the maximum collision order $M$,
 \begin{equation}
     \sum_{\ell = 1 }^M 2^\ell A^\ell_{N,\epsilon_{\ell+1}} C_{\ell+1} \leq C_{d,\beta_0,\mu_0} (e-1) T = \tilde{C}_{d,\beta_0, \mu_0} T.
 \end{equation}
 \end{remark}
  We now can prove well-posedness of the BBGKY hierarchy up to a short time by applying Lemma \ref{BBGKY T est} and setting $\lambda=\beta_0/2T$.
 \begin{theorem}\label{well posedness BBGKY}
 Let $\beta_0> 0$, $\mu_0\in\mathbb{R}$. There exists $T=T(d,\beta_0,\mu_0)>0$ such that for any initial datum $F_{N,0}=(f_{N,0}^{(s)})_{s\in\mathbb{N}}\in X_{N,\beta_0,\mu_0}$ there is unique mild solution $\bm{F_N}=(f_N^{(s)})_{s\in\mathbb{N}}\in\bm{X}_{N,\bm{\beta},\bm{\mu}}$ to the BBGKY hierarchy in $[0,T]$ for the functions $\bm{\beta},\bm{\mu}:[0,T]\to\mathbb{R}$ given by
  \begin{equation}\label{beta mu given lambda}
  \begin{aligned}
  \bm{\beta}(t)&=\beta_0-\frac{\beta_0}{2T}t,\quad \bm{\mu}(t)&=\mu_0-\frac{\beta_0}{2T}t.
  \end{aligned}
  \end{equation}
   The solution $\bm{F_N}$ satisfies the bound:
 \begin{equation}
  \label{a priori bound F_N,0} |||\bm{F_N}|||_{N,\bm{\beta},\bm{\mu}}\leq 2 \|F_{N,0}\|_{N,\beta_0,\mu_0}.
 \end{equation}
 Moreover, for any $\mathcal{F}(t)\subseteq[0,t]$ measurable and $\ell\in\{1, \cdots, M\}$, the following bound holds:
 \begin{align}
 \label{a priori binary bound F_N}\left|\left|\left|\int_{\mathcal{F}(t)}\mathcal{T}^{t-\tau}C_N^{\ell+1}\bm{G_N}(\tau)\,d\tau\right|\right|\right|_{N,\bm{\beta},\bm{\mu}}&\leq\frac{1}{2 (e-1)\ell!}|||G_N|||_{{N,\bm{\beta},\bm{\mu}}},\quad\forall G_N\in\bm{X}_{N,\bm{\beta},\bm{\mu}},
 \end{align}
The time $T$ is explicitly given by:
\begin{equation}\label{time}
    T = \frac{1}{2 \tilde{C}_{d,\beta_0,\mu_0}}.
\end{equation}
 \end{theorem}
\begin{proof}
    Fix  $F_{N,0} \in X_{N,\beta_0,\mu_0}$. We define the operator $\mathcal{L}: \bm{X}_{N,\bm{\beta},\bm{\mu}} \rightarrow \bm{X}_{N,\bm{\beta},\bm{\mu}}$ by:
    \begin{equation}
        \mathcal{L} \bm{G_N}(t) = \mathcal{T}^tG_{N,0}+\int_0^t \mathcal{T}^{t-\tau}\mathcal{C}_N\bm{G_N}(\tau)\,d\tau.
    \end{equation}
    Define $T$ as in \eqref{time}.
    Therefore, for $t\in [0,T]$ and $\mathcal{F}(t) \subset [0,t]$, by Lemma \ref{BBGKY T est}, 
    \begin{equation}\label{wp pf bound}
\left|\left|\left|\displaystyle\int_{\mathcal{F}(t)}\mathcal{T}^{t-\tau}\mathcal{C}_N^{\ell+1}\bm{G_N}(\tau)\,d\tau\right|\right|\right|_{N,\bm{\beta}_\lambda,\bm{\mu}_\lambda} \leq \frac{1}{2 (e-1)\ell!} |||\bm{G_N}|||_{N,\bm{\beta}_\lambda,\bm{\mu}_\lambda},
    \end{equation}
    and thus, recalling \eqref{C_N defn}, we have
    \begin{equation}
\left|\left|\left|\displaystyle\int_{\mathcal{F}(t)}\mathcal{T}^{t-\tau}\mathcal{C}_N \bm{G_N}(\tau)\,d\tau\right|\right|\right|_{N,\bm{\beta}_\lambda,\bm{\mu}_\lambda} \leq \frac{1}{ 2} |||\bm{G_N}|||_{N,\bm{\beta}_\lambda,\bm{\mu}_\lambda}.
    \end{equation}
    For the case $\mathcal{F}(t) = [0,t]$, the operator $\mathcal{L}$ is a contraction with unique fixed point $\bm{F_N} \in \bm{X}_{N,\bm{\beta}, \bm{\mu}}$. Therefore, $\bm{F_N}$ is the unique mild solution to the BBGKY hierarchy in $[0,T]$ with initial datum $F_{N,0}$ satisfying \eqref{a priori binary bound F_N}. 
    Finally, by \eqref{a priori binary bound F_N} and part (1) of Proposition \ref{remark for initial}, for all $\ell \in \{1, \cdots, M \}$ we have
    \begin{align}
        |||\bm{F_N}|||_{N,\bm{\beta},\bm{\mu}(t)} &\leq ||\mathcal{T}^t F_{N,0}||_{N,\bm{\beta}(t), \bm{\mu}(t)} + \left|\left|\left|\displaystyle\int_{0}^t\mathcal{T}^{t-\tau}\mathcal{C}_N \bm{F_N}(\tau)\,d\tau\right|\right|\right|_{N,\bm{\beta}_\lambda,\bm{\mu}_\lambda}\\
        &\leq ||F_{N,0}||_{N,\beta_0,\mu_0} + \frac{1}{2} |||\bm{F_N}|||_{N,\bm{\beta}, \bm{\mu}}.
    \end{align}
    Thus, we prove \eqref{a priori bound F_N,0},
\begin{equation}
|||\bm{F_N}|||_{N,\bm{\beta},\bm{\mu}}\leq 2 \|F_{N,0}\|_{N,\beta_0,\mu_0}.
 \end{equation}
\end{proof}

 \subsection{LWP for the Boltzmann hierarchy} Similary to Subsection \ref{sub BBGKY well posedness}, here we establish a-priori bounds and local well-posedness for the Boltzmann hierarchy. Without loss of generality, we will omit the proofs since they are identical to the BBGKY hierarchy case.  Given $s\in\mathbb{N}$ and $\beta> 0$, we define  the Banach space
 \begin{equation*}
 X_{\infty,\beta,s}
 :=\left\{g_{s}\in L^\infty(\mathbb{R}^{2ds}):|g_{s}|_{\infty,\beta,s}<\infty\right\},
 \end{equation*}
with norm
$|g_{s}|_{\infty,\beta,s}=\sup_{Z_s\in\mathbb{R}^{2ds}}|g_{s}(Z_s)|e^{\beta E_s(Z_s)},$
where $E_s(Z_s)$ is the kinetic energy of the $s$-particles given by \eqref{kinetic energy}. 
\begin{remark}\label{S_s isometry} Given $t\in\mathbb{R}$ and $s\in\mathbb{N}$, conservation of energy under the free flow  implies that the $s$-particle  free  flow operator $S_s^t:X_{\infty,\beta,s}\to X_{\infty,\beta,s}$, given in \eqref{free flow operator}, is an isometry i.e.
\begin{equation*}
|S_s^tg_{s}|_{\infty,\beta,s}=|g_{s}|_{\infty,\beta,s},\quad\forall g_{s}\in X_{\infty,\beta,s}.
\end{equation*}
\end{remark}

Consider as well $\mu\in\mathbb{R}$. We define the Banach space 
\begin{equation*}
X_{\infty,\beta,\mu}:=\left\{G=(g_{s})_{s\in\mathbb{N}}:\|G\|_{\infty,\beta,\mu}<\infty\right\},\end{equation*}
with norm
$\|G\|_{\infty,\beta,\mu}=\sup_{s\in\mathbb{N}}e^{\mu s}|g_{s}|_{\infty,\beta,s}.$
\begin{remark}\label{S isometry} Given $t\in\mathbb{R}$, Remark \ref{S_s isometry} implies that the map $\mathcal{S}^t:X_{\infty,\beta,\mu}\to X_{\infty,\beta,\mu}$ given by
\begin{equation}\label{S definition}
\mathcal{S}^tG:=\left(S_s^tg_{s}\right)_{s\in\mathbb{N}},
\end{equation}
is an isometry i.e.
$
\|\mathcal{S}^tG\|_{\infty,\beta,\mu}=\|G\|_{\infty,\beta,\mu},$ for any  $G\in X_{\infty,\beta,\mu}.$
\end{remark}
Finally, given $T>0$, $\beta_0> 0$, $\mu_0\in\mathbb{R}$ and $\bm{\beta},\bm{\mu}:[0,T]\to\mathbb{R}$ decreasing functions of time with $\bm{\beta}(0)=\beta_0$, $\bm{\beta}(T)> 0$, $\bm{\mu}(0)=\mu_0$, we define the Banach space 
\begin{equation*}
\bm{X}_{\infty,\bm{\beta},\bm{\mu}}:=C^0\left([0,T],X_{\infty,\bm{\beta}(t),\bm{\mu}(t)}\right),
\end{equation*}
with norm
$|||\bm{G}|||_{\infty,\bm{\beta},\bm{\mu}}=\sup_{t\in[0,T]}\|\bm{G}(t)\|_{\infty,\bm{\beta}(t),\bm{\mu}(t)}.$
\begin{proposition}\label{remark for initial boltzmann hierarchy} Let $T>0$, $\beta_0>0$, $\mu_0\in\mathbb{R}$ and $\bm{\beta},\bm{\mu}:[0,T]\to\mathbb{R}$ decreasing functions with $\beta_0=\bm{\beta}(0)$, $\bm{\beta}(T)> 0$ $\mu_0=\bm{\mu}(0)$. Then for any $G=\left(g_{s}\right)_{s\in\mathbb{N}}\in X_{\infty,\beta_0,\mu_0}$, the following estimates hold:
\begin{enumerate}
\item $|||G|||_{\infty,\bm{\beta},\bm{\mu}}\leq\|G\|_{\infty,\beta_0,\mu_0}$.\vspace{0.2cm}
\item $\left|\left|\left|\displaystyle\int_0^t\mathcal{S}^{\tau}G\,d\tau\right|\right|\right|_{\infty,\bm{\beta},\bm{\mu}}\leq T\|G\|_{\infty,\beta_0,\mu_0}.$
\end{enumerate}
\end{proposition}

\begin{lemma}\label{a priori lemma for C Boltzmann}
Let $s\in\mathbb{N}$ and $\beta>0$.  For any $Z_s\in\mathbb{R}^{2ds}$ and $\ell\in\{1, \cdots, M\}$, the following continuity estimate holds:
\begin{equation}
\left|\mathcal{C}_{s,s+\ell}^{\infty}g_{s+\ell}(Z_s)\right|\leq C_d \frac{1}{\ell!} \ell^3  \beta^{-\ell d/2}\left(s\beta^{-1/2}+\sum_{i=1}^s|v_i|\right)e^{-\beta E_s(Z_s)}|g_{s+\ell}|_{\infty,\beta,s+\ell},\quad\forall g_{s+\ell}\in X_{\infty,\beta,s+\ell}.\label{cont estimate both boltzmann}
\end{equation}
\end{lemma}
Let us now define mild solutions to the Boltzmann hierarchy:
\begin{definition}\label{def mild solution boltzmann} Consider $T>0$, $\beta_0> 0$, $\mu_0\in\mathbb{R}$ and the decreasing functions $\bm{\beta},\bm{\mu}:[0,T]\to\mathbb{R}$ with $\bm{\beta}(0)=\beta_0$, $\bm{\beta}(T)> 0$, $\bm{\mu}(0)=\mu_0$.
Consider also  initial data $G_{0}=\left(g^{(s)}_0\right)\in X_{\infty,\beta_0,\mu_0}$. A map $\bm{G}=\left(g^{(s)}\right)_{s\in\mathbb{N}}\in\bm{X}_{\infty,\bm{\beta},\bm{\mu}}$ is a mild solution of the Boltzmann hierarchy in $[0,T]$, with initial data $G_0$, if it satisfies:
\begin{equation*}\bm{G}(t)=\mathcal{S}^tG_{0}+\int_0^t \mathcal{S}^{t-\tau}\mathcal{C}_\infty\bm{G}(\tau)\,d\tau,\end{equation*}
where, given $\beta>0$, $\mu\in\mathbb{R}$ and ${G}=({g}^{(s)})_{s\in\mathbb{N}}\in X_{\infty,\beta,\mu}$, we write
\begin{align*}
&\mathcal{C}_\infty G:=(\mathcal{C}_\infty^2 +\cdots +\mathcal{C}_\infty^{M+1}) G_N, \\
&\mathcal{C}_\infty^{\ell+1} G:=\left(\mathcal{C}_{s,s+\ell}^\infty g^{(s+\ell)}\right)_{s\in\mathbb{N}} \quad \ell \in \{ 1, \cdots, M\},
\end{align*}
and $\mathcal{S}^t$ is given by \eqref{S definition}.
\end{definition}
\begin{remark}
    As in Remark \ref{rmk mild bbgky}, the above mild formulation is an abuse of notation because the operators $\mathcal{C}_{s,s+\ell}^\infty$ are ill-defined on $L^\infty$. In this case, we obtain a well-defined mild solution to the Boltzmann hierarchy by filtering the Boltzmann hierarchy by the flow $S^{-t}_s$.
\end{remark}
Using Lemma \ref{a priori lemma for C Boltzmann}, we obtain the following a-priori bounds:
\begin{lemma}\label{a priori lemma for S boltzmann} Let $\beta_0> 0$, $\mu_0\in\mathbb{R}$, $T>0$ and $\lambda\in (0,\beta_0/T)$. Consider the functions $\bm{\beta}_\lambda,\bm{\mu}_\lambda:[0,T]\to\mathbb{R}$ given by
\eqref{beta_lambda-mu_lambda}.
 Then  for any $\mathcal{F}(t)\subseteq [0,t]$ measurable,  $\bm{G}=\left(g_{s}\right)_{s\in\mathbb{N}}\in\bm{X}_{\infty,\bm{\beta}_\lambda,\bm{\mu}_\lambda}$ and $\ell\in\{1, \cdots, M\}$, the following bound holds:
 \begin{align}
\left|\left|\left|\displaystyle\int_{\mathcal{F}(t)}\mathcal{S}^{t-\tau}\mathcal{C}_\infty^{\ell+1}\bm{G}(\tau)\,d\tau\right|\right|\right|_{\infty,\bm{\beta}_\lambda,\bm{\mu}_\lambda}&\leq \frac{1}{\ell!}C_{\ell+1}|||\bm{G}|||_{\infty,\bm{\beta}_\lambda,\bm{\mu}_\lambda},\label{both boltzmann with constant}
\end{align}
 where the constant $C_{\ell+1}=C_{\ell+1}(d,\beta_0,\mu_0,T,\lambda)$ is given by \eqref{constant of WP binary}.
\end{lemma}

  Choosing $\lambda=\beta_0/2T$, Lemma \ref{a priori lemma for S boltzmann} directly implies well-posedness of the Boltzmann hierarchy up to short time.
 \begin{theorem}\label{well posedness boltzmann}
 Let $\beta_0> 0$, $\mu_0\in\mathbb{R}$. There exists $T=T(d,\beta_0,\mu_0)>0$ such that for any initial datum $F_{0}=(f_{0}^{(s)})_{s\in\mathbb{N}}\in X_{\infty,\beta_0,\mu_0}$ there is unique mild solution $\bm{F}=(f^{(s)})_{s\in\mathbb{N}}\in\bm{X}_{\infty,\bm{\beta},\bm{\mu}}$ to the Boltzmann hierarchy in $[0,T]$ for the functions $\bm{\beta},\bm{\mu}:[0,T]\to\mathbb{R}$ given by \eqref{beta mu given lambda}. The solution $\bm{F}$ satisfies the bound:
 \begin{equation}
 \label{a priori bound F_0 Boltzmann}|||\bm{F}|||_{\infty,\bm{\beta},\bm{\mu}}\leq 2\|F_{0}\|_{\infty,\beta_0,\mu_0}.
 \end{equation}
 Moreover, for any $\mathcal{F}(t)\subseteq[0,t]$ measurable and $\ell\in\{1,\cdots, M\}$, the following bound holds:
 \begin{align}
 \label{a priori binary bound F Boltzmann}\left|\left|\left|\int_{\mathcal{F}(t)}\mathcal{S}^{t-\tau}C_\infty^{\ell+1}\bm{G}(\tau)\,d\tau\right|\right|\right|_{\infty,\bm{\beta},\bm{\mu}}&\leq\frac{1}{2(e-1)\ell!}|||G|||_{{\infty,\bm{\beta},\bm{\mu}}},\quad\forall G\in\bm{X}_{\infty,\bm{\beta},\bm{\mu}},
 \end{align}
 and the time $T$ is explicitly given by \eqref{time}.
 \end{theorem}
 \begin{proof}
     The proof follows the same strategy as the proof for Theorem \ref{well posedness BBGKY}.
 \end{proof}
 \subsection{LWP for the M+1-nary Boltzmann equation and propagation of chaos}
For $\beta>0$ let us define the Banach space
\begin{equation*}
X_{\beta,\mu}:=\left\{g\in L^\infty(\mathbb{R}^{2d}):|g|_{\beta,\mu}<\infty\right\},
\end{equation*}
with norm
$
|g|_{\beta,\mu}=\sup_{(x,v)\in\mathbb{R}^{2d}} |g(x,v)|e^{\mu+\frac{\beta}{2} |v|^2}.
$
Notice that for any $t\in[0,T]$, the map $S_1^t:X_{\beta,\mu}\to X_{\beta,\mu}$ is an isometry.

Consider $\beta_0>0$, $\mu_0\in\mathbb{R}$, $T>0$ and $\bm{\beta},\bm{\mu}:[0,T]\to\mathbb{R}$ decreasing functions of time with $\bm{\beta}(0)=\beta_0$, $\bm{\beta}(T)>0$ and $\bm{\mu}(0)=\mu_0$.
We define the Banach space
\begin{equation*}
\bm{X}_{\bm{\beta},\bm{\mu}}:=C^0\left([0,T],X_{\bm{\beta}(t),\bm{\mu}(t)}\right),
\end{equation*} 
with norm
$
\|\bm{g}\|_{\bm{\beta},\bm{\mu}}=\sup_{t\in[0,T]}|\bm{g}(t)|_{\bm{\beta}(t),\bm{\mu}(t)}.
$
\begin{remark}\label{remark for initial equation} Let $T>0$, $\beta_0>0$, $\mu_0\in\mathbb{R}$ and $\bm{\beta},\bm{\mu}:[0,T]\to\mathbb{R}$ decreasing functions with $\beta_0=\bm{\beta}(0)$, $\bm{\beta}(T)> 0$ $\mu_0=\bm{\mu}(0)$. Then for any $g\in X_{\beta_0,\mu_0}$, the following estimate holds:
\begin{equation*}
\|g\|_{\bm{\beta},\bm{\mu}}\leq |g|_{\beta_0,\mu_0}.
\end{equation*}
\end{remark}


\begin{lemma}\label{continuity boltzmann lemma} Let $\beta>0$, $\mu\in\mathbb{R}$. Then for any $g,h\in X_{\beta,\mu}$ and $(x,v)\in\mathbb{R}^{2d}$, the following nonlinear continuity estimate holds:\begin{align}
    \bigg|& \bigg[\sum_{\ell = 1}^M \frac{1}{\ell!} Q_{\ell+1}(g,\cdots, g)\bigg](x,v) - \bigg[\sum_{\ell =1 }^M \frac{1}{\ell!}Q_{\ell+1}(h,\cdots, h) \bigg](x,v) \bigg|\\
     & \lesssim  \big(\sum_{\ell =1 }^M \frac{1}{\ell!} \ell^3 e^{-(\ell+1) \mu} \beta^{\frac{-\ell d}{2}} (|g|_{\beta,\mu} + |h|_{\beta,\mu})^\ell  \big) (|v| + \beta^{-\frac{1}{2}}) e^{-\frac{\beta}{2}|v|^2} |g-h|_{\beta,\mu}.
\end{align}
\end{lemma}
\begin{proof}
First, we will estimate $|Q_{\ell+1}(g,\cdots, g) - Q_{\ell+1}(h,\cdots, h)|$ for all $\ell \in \{1,\cdots, M \}$. By estimate \ref{b+ est} we have, 
\begin{equation}
    \begin{split}
        |Q_{\ell+1}(g,&\cdots, g)(x,v) - Q_{\ell+1}(h,\cdots, h)(x,v)| \\
        &\leq 4 \ell^2  \int_{\E \times \R^{\ell d}} \big(|v| + |v_1| + \cdots |v_{\ell}| \big) \big[ |g^* g^*_1 \cdots g^*_{\ell}- h^* h^*_1 \cdots h^*_{\ell}| +|g g_1 \cdots g_{\ell}- h h_1 \cdots h_{\ell}|\big] d\bm{\omega} dV_{\ell}.
    \end{split}
\end{equation}
By utilizing the identity 
\begin{equation}
    g_0 \cdots g_{\ell} - h_0 \cdots h_{\ell} = \sum_{i=0}^{\ell} (g_i - h_i) h_0 \cdots h_{i - 1} g_{i+1} \cdots g_{\ell}, 
\end{equation}
where we denote
\begin{equation}
    \begin{split}
        g_0 = g, \quad h_0 = h,
    \end{split}
\end{equation}
and performing the change of variables $(v^*, v_1^*, \cdots, v_\ell^*) \mapsto (v, v_1, \cdots, v_\ell)$ on the post-collisional part we arrive at,
\begin{align}
    |Q_{\ell+1}(g,&\cdots, g)(x,v) - Q_{\ell+1}(h,\cdots, h)(x,v)| \\
        &\lesssim \ell^2  \int_{\E\times \R^{\ell d}} \big(|v| + |v_1| + \cdots |v_{\ell}| \big) \sum_{i=0}^{\ell} (g_i - h_i) h_0 \cdots h_{i - 1} g_{i+1} \cdots g_{\ell} d\bm{\omega} dV_{\ell}.
\end{align}
 We now recall the definition of the norms $|g|_{\beta, \mu}$ and $|h|_{\beta, \mu}$ to derive
\begin{align}
    |Q_{\ell+1}(g,&\cdots, g)(x,v) - Q_{\ell+1}(h,\cdots, h)(x,v)| \\
    &\lesssim \ell^2 e^{- (\ell + 1)\mu - \frac{\beta}{2}|v|^2} |g - h|_{\beta, \mu}  \sum_{i=0}^{\ell-1} \big(|g|^i_{\beta,\mu} |h|^{\ell-1 - i}_{\beta,\mu} \big)  \int_{\E\times \R^{\ell d}} \big(|v| + |v_1| + \cdots |v_{\ell}| \big) e^{- \frac{\beta}{2} \sum_{k = 1}^\ell |v_k|^2}  d\bm{\omega} dV_\ell,
\end{align}
where we note the estimate 
\begin{equation}
    \sum_{i=0}^{\ell-1} \big(|g|^i_{\beta,\mu} |h|^{\ell-1 - i}_{\beta,\mu} \big) \lesssim  (|g|_{\beta,\mu}+|h|_{\beta,\mu})^\ell,
\end{equation}
to finally arrive at, 
\begin{align}
    |Q_{\ell+1}(g,&\cdots, g)(x,v) - Q_{\ell+1}(h,\cdots, h)(x,v)| \\
    &\lesssim \ell^2 e^{- (\ell + 1)\mu - \frac{\beta}{2}|v|^2} |g - h|_{\beta, \mu}  (|g|_{\beta,\mu}+|h|_{\beta,\mu})^\ell \int_{\E\times \R^{\ell d}} \big(|v| + |v_1| + \cdots |v_{\ell}| \big) e^{- \frac{\beta}{2} \sum_{k = 1}^\ell |v_k|^2}  d\bm{\omega} dV_\ell.
\end{align}
By applying the same procedure as is done in the proof of Lemma \ref{BBGKY C est}, we arrive at the estimate,
\begin{align}
    |Q_{\ell+1}(g,&\cdots, g)(x,v) - Q_{\ell+1}(h,\cdots, h)(x,v)| \\
    &\lesssim  \ell^3 \beta^{\frac{-\ell d}{2}}\bigg(  |v| + \beta^{-\frac{1}{2}}   \bigg)e^{-(\ell + 1)\mu - \frac{\beta}{2} |v|^2}  (|g|_{\beta,\mu}+|h|_{\beta,\mu})^\ell |g - h|_{\beta, \mu}
\end{align}
By summing these differences over $\ell$, we obtain the desired estimate, 
\begin{align}
    \bigg|& \bigg[\sum_{\ell = 1}^M \frac{1}{\ell!} Q_{\ell+1}(g,\cdots, g)\bigg](x,v) - \bigg[\sum_{\ell =1 }^M \frac{1}{\ell!}Q_{\ell+1}(h,\cdots, h) \bigg](x,v) \bigg|\\
     & \lesssim  \big(\sum_{\ell =1 }^M \frac{1}{\ell!} \ell^3 e^{-(\ell+1) \mu} \beta^{\frac{-\ell d}{2}} (|g|_{\beta,\mu} + |h|_{\beta,\mu})^\ell  \big) (|v| + \beta^{-\frac{1}{2}}) e^{-\frac{\beta}{2}|v|^2} |g-h|_{\beta,\mu}.
\end{align}

\end{proof}

We define mild solutions to the $M+1$-nary Boltzmann equation \eqref{Boltzmann equation} as follows:
\begin{definition}
Consider $T>0$, $\beta_0> 0$, $\mu_0\in\mathbb{R}$ and  $\bm{\beta},\bm{\mu}:[0,T]\to\mathbb{R}$  decreasing functions of time, with $\bm{\beta}(0)=\beta_0$, $\bm{\beta}(T)> 0$, $\bm{\mu}(0)=\mu_0$.
Consider also initial data $g_{0}\in X_{\beta_0,\mu_0}$. A map $\bm{g}\in\bm{X}_{\bm{\beta},\bm{\mu}}$ is a mild solution to the $M+1$-nary Boltzmann equation \eqref{Boltzmann equation} in $[0,T]$, with initial data $g_0\in X_{\beta_0,\mu_0}$, if it satisfies

\begin{equation}\label{mild boltzmann equation}
\bm{g}(t)=S_1^tg_0+\int_0^tS_1^{t-\tau} \bigg[\sum_{\ell = 1 }^M \frac{1}{\ell!}Q_{\ell+1}(\bm{g}, \cdots, \bm{g})\bigg] (\tau)\,d\tau.
\end{equation} 

where $S_1^t$ denotes the free flow of  one particle given in \eqref{free flow operator}.
\end{definition}

A similar proof to Lemma \ref{BBGKY T est} gives the following:
\begin{lemma}\label{Lemma for integral boltzmann}
Let $\beta_0> 0$, $\mu_0\in\mathbb{R}$, $T>0$ and $\lambda\in (0,\beta_0/T)$. Consider the functions $\bm{\beta}_\lambda,\bm{\mu}_\lambda:[0,T]\to\mathbb{R}$ given by \eqref{beta_lambda-mu_lambda}.  
 Then for any $\bm{g},\bm{h}\in\bm{X}_{\bm{\beta}_\lambda,\bm{\mu}_\lambda}$ the following bounds hold:
\begin{equation}
 \left\|\int_0^tS_1^{t-\tau}\bigg[\sum_{\ell = 1 }^M \frac{1}{\ell!}Q_{\ell+1}(\bm{g}-\bm{h}, \cdots, \bm{g}-\bm{h})\bigg](\tau)\,d\tau\right\|_{\bm{\beta}_\lambda,\bm{\mu}_\lambda} \leq \bigg(\sum_{\ell = 1}^M \frac{1}{\ell!}C_{\ell + 1} (|\bm{g}|_{\bm{\beta}_\lambda,\bm{\mu}_\lambda}+ |\bm{h}|_{\bm{\beta}_\lambda,\bm{\mu}_\lambda})^\ell \bigg)
 |\bm{g}-\bm{h}|_{\bm{\beta}_\lambda,\bm{\mu}_\lambda},
\end{equation} 
  where for $\ell = 1, \cdots, M$, the constants $C_{\ell+1}$ are given by  \eqref{constant of WP binary}.
\end{lemma}

Choosing $\lambda=\beta_0/2T$, this estimate implies local well-posedness of the generalized Boltzmann equation \eqref{Boltzmann equation} up to short times. 
Let us write $B_{\bm{X}_{\bm{\beta},\bm{\mu}}}$ for the unit ball of $\bm{X}_{\bm{\beta},\bm{\mu}}$.

\begin{theorem}[LWP for the $M+1$-nary Boltzmann equation]\label{lwp boltz eq}
 Let $\beta_0> 0$, $\mu_0\in\mathbb{R}$. Then there exists $T=T(d,\beta_0,\mu_0)>0$ such that for any initial data $f_0\in X_{\beta_0,\mu_0}$, with $|f_0|_{\beta_0,\mu_0} \leq 1/2$, there is a unique mild solution $\bm{f}\in B_{\bm{X}_{\bm{\beta},\bm{\mu}}}$ to the $M+1$-nary Boltzmann equation in $[0,T]$ with initial data $f_0$ \eqref{Boltzmann equation}, where $\bm{\beta},\bm{\mu}:[0,T]\to\mathbb{R}$ are the functions given by \eqref{beta mu given lambda}. The solution $\bm{f}$ satisfies the bound:
  \begin{equation}\label{bound on initial data boltzmann equ}
 \|\bm{f}\|_{\bm{\beta},\bm{\mu}}\leq 2 |f_0|_{\beta_0,\mu_0}.
 \end{equation}
 Moreover, for any $\bm{g,h}\in\bm{X}_{\bm{\beta},\bm{\mu}}$, the  following estimates hold:
 \begin{align}\label{a-priori BE 1}
 \left\|\int_0^tS_1^{t-\tau}\sum_{\ell = 1 }^M \frac{1}{\ell!} Q_{\ell+1}(\bm{g}-\bm{h}, \cdots, \bm{g}-\bm{h})(\tau)\,d\tau\right\|_{\bm{\beta},\bm{\mu}}
 \leq \frac{1}{2} 
 \|\bm{g}-\bm{h}\|_{\bm{\beta},\bm{\mu}}.
 \end{align}
The time $T$ is given in \eqref{time}.
\end{theorem}
\begin{proof}
Choosing $T$  as in \eqref{time},
we obtain $\sum_{\ell = 1}^M \frac{2^\ell}{\ell!} C_{\ell + 1} \leq \frac{1}{2}$.
Thus,  Lemma \ref{Lemma for integral boltzmann} implies estimate \eqref{a-priori BE 1}.
Therefore, for any $g\in B_{\bm{X}_{\bm{\beta},\bm{\mu}}}$, using \eqref{a-priori BE 1} for $\bm{h}=0$, we obtain
\begin{equation}\label{inequality for cubic boltzmann 1}
 \left\|\int_0^tS_1^{t-\tau}\left[\sum_{\ell = 1 }^M \frac{1}{\ell!}Q_{\ell+1}(\bm{g}, \cdots, \bm{g})\right](\tau)\,d\tau\right\|_{\bm{\beta}_\lambda,\bm{\mu}_\lambda}\leq \sum_{\ell = 1}^M \frac{1}{\ell!} C_{\ell+ 1} \|\bm{g}\|_{\bm{\beta},\bm{\mu}}^{\ell+1}\leq \frac{1}{2}.
\end{equation}
Let us define the nonlinear operator $\mathcal{L}:\bm{X}_{\bm{\beta},\bm{\mu}}\to\bm{X}_{\bm{\beta},\bm{\mu}}$ by
\begin{equation*}
\mathcal{L}\bm{g}(t)=S_1^tf_0+\int_0^tS_1^{t-\tau}\sum_{\ell = 1 }^M \frac{1}{\ell!}Q_{\ell+1}(\bm{g}, \cdots, \bm{g})(\tau)\,d\tau.
\end{equation*}
By the triangle inequality, the isometry of the free flow, Remark \ref{remark for initial equation}, bound \eqref{inequality for cubic boltzmann 1}, and the assumption $|f_0|_{\beta_0,\mu_0}\leq  1/2$, for any $\bm{g}\in B_{\bm{X}_{\bm{\beta},\bm{\mu}}}$ and $t\in[0,T]$, we have
\begin{equation*}
\begin{aligned}
|\mathcal{L}\bm{g}|_{\bm{\beta}(t),\bm{\mu}(t)}\leq |S_1^tf_0|_{\bm{\beta}(t),\bm{\mu}(t)}+ \frac{1}{2}=|f_0|_{\bm{\beta}(t),\bm{\mu}(t)}+ \frac{1}{2} \leq 1
\end{aligned}
\end{equation*}
 Thus, the operator $\mathcal{L}$ maps into the ball, $\mathcal{L}:B_{\bm{X}_{\bm{\beta},\bm{\mu}}}\to B_{\bm{X}_{\bm{\beta},\bm{\mu}}}$.
Moreover, for any $\bm{g},\bm{h}\in B_{\bm{X}_{\bm{\beta},\bm{\mu}}}$, using \eqref{a-priori BE 1}, we obtain
\begin{equation}\label{pre-triang-bolt}
\left\|\mathcal{L}\bm{g}-\mathcal{L}\bm{h}\right\|_{\bm{\beta},\bm{\mu}}
\leq \frac{1}{2}\|\bm{g}-\bm{h}\|_{\bm{\beta},\bm{\mu}}.
\end{equation}
The operator $\mathcal{L}:B_{\bm{X}_{\bm{\beta},\bm{\mu}}}\to B_{\bm{X}_{\bm{\beta},\bm{\mu}}}$ is a contraction with a unique fixed point $\bm{f}\in B_{\bm{X}_{\bm{\beta},\bm{\mu}}}$ which is the unique mild solution of the $M+1$-nary Boltzmann equation in $[0,T]$ with initial data $f_0$.

To  prove \eqref{bound on initial data boltzmann equ}, we use the fact that $\bm{f}=\mathcal{L}\bm{f}$. Then for any $t\in[0,T]$, triangle inequality, definition of $\mathcal{L}$, estimate \eqref{pre-triang-bolt}(for $\bm{g}=\bm{f}$ and $\bm{g}=0$), free flow being isometric,  and Remark \ref{remark for initial equation} yields
\begin{align*}
|\bm{f}|_{\bm{\beta}(t),\bm{\mu}(t)}=|\mathcal{L}\bm{f}|_{\bm{\beta}(t),\bm{\mu}(t)}&\leq |\mathcal{L}\bm{0}|_{\bm{\beta}(t),\bm{\mu}(t)}+|\mathcal{L}\bm{f}-\mathcal{L}\bm{0}|_{\bm{\beta}(t),\bm{\mu}(t)} \leq|S_1^tf_0|_{\bm{\beta}(t),\bm{\mu}(t)}+\frac{1}{2}\|\bm{f}\|_{\bm{\beta},\bm{\mu}}\\
&=|f_0|_{\bm{\beta}(t),\bm{\mu}(t)}+\frac{1}{2}\|\bm{f}\|_{\bm{\beta},\bm{\mu}}
\leq |f_0|_{\beta_0,\mu_0}+\frac{1}{2}\|\bm{f}\|_{\bm{\beta},\bm{\mu}}.
\end{align*}
Thus, $\|\bm{f}\|_{\bm{\beta},\bm{\mu}}\leq |f_0|_{\beta_0,\mu_0}+\displaystyle\frac{1}{2}\|\bm{f}\|_{\bm{\beta},\bm{\mu}},$ and \eqref{bound on initial data boltzmann equ} follows.
\end{proof}

We can now prove that chaos is propagated by the Boltzmann hierarchy.

\begin{theorem}[Propagation of chaos]\label{theorem propagation of chaos}
Let $\beta_0>0$, $\mu_0\in\mathbb{R}$, $T>0$ be the time given in \eqref{time}, and $\bm{\beta},\bm{\mu}:[0,T]\to\mathbb{R}$ the functions defined by \eqref{beta mu given lambda}. Consider $f_0\in X_{\beta_0,\mu_0}$ with
$|f_0|_{\beta_0,\mu_0}\leq 1/2$.
Assume $\bm{f}\in B_{\bm{X}_{\bm{\beta},\bm{\mu}}}$ is the corresponding mild solution of the $M+1$-nary Boltzmann equation in $[0,T]$, with initial data $f_0$ given by Theorem \ref{lwp boltz eq}. Then the following hold:
\begin{enumerate}
\item $F_0=(f_0^{\otimes s})_{s\in\mathbb{N}}\in X_{\infty,\beta_0,\mu_0}$.
\item $\bm{F}=(\bm{f}^{\otimes s})_{s\in\mathbb{N}}\in\bm{X}_{\infty,\bm{\beta},\bm{\mu}}$.
\item $\bm{F}$ is the unique mild solution of the Boltzmann hierarchy in $[0,T]$, with initial data $F_0$.
\end{enumerate}

\end{theorem}
\begin{proof}
\textit{(i)} is trivially verified by the bound on the initial data \eqref{bound on initial data boltzmann equ} and the definition of the norms. By the same bound again, we may apply Theorem \ref{lwp boltz eq} to obtain the unique mild solution $\bm{f}\in B_{\bm{X}_{\bm{\beta},\bm{\mu}}}$ of the corresponding $M+1$-nary Boltzmann equation. Since $\|\bm{f}\|_{\bm{\beta},\bm{\mu}}\leq 1$, the definition of the norms directly imply \textit{(ii)}. It is also straight forward to verify that $\bm{F}$ is a mild solution of the Boltzmann hierarchy in $[0,T]$, with initial data $F_0$. Uniqueness of the mild solution to the Boltzmann hierarchy, obtained by Theorem \ref{well posedness boltzmann}, implies that $\bm{F}$ is  the unique mild solution.
\end{proof}
 
\section{Statement of the Main Result}
\label{sec_conv statement}

In this section, we define the relevant notion of convergence, the convergence in observables, and state the main result of the paper.

\subsection{Approximation of Boltzmann initial data} 
Here, we approximate Boltzmann hierarchy initial data by BBGKY hierarchy initial data. 
\begin{definition}\label{approx bbgky hier}
    Let $\beta_0 > 0$, $\mu_0 \in \R$ and $G_0 = (g_{s,0})_{s\in \N} \in X_{\infty,\beta_0, \mu_0}$. A sequence $G_{N,0} = (g_{N,s,0})_{s\in \N} \in X_{N, \beta_0, \mu_0}$ is called a BBGKY hierarchy sequence approximating $G_0$ if the following conditions hold:
    \begin{enumerate}
        \item $\sup\limits_{N\in \N} ||G_{N,\beta_0,\mu_0}||_{N,\beta_0,\mu_0} < \infty$.
        \item For any $s\in \N$, there holds $\lim\limits_{N \rightarrow \infty} ||g_{N,s,0} - g_{s,0}||_{L^\infty(\mathcal{D}_{s,\bm{\epsilon}_M})} = 0$.
    \end{enumerate}
\end{definition}
\begin{definition}\label{cond bbgky hier}
    Let $g_0 \in X_{\beta_0,\mu_0+1}$ be a positive probability density and denote $G_0 = (g_0^{\otimes s}) \in X_{\infty,\beta_0,\mu_0+1}$. We define the conditioned BBGKY hierarchy sequence $G_{N,0} = (g_{N,0}^{(s)})_{s\in \N}$ of $G_0$ as:
    \begin{equation*}
        g^{(s)}_{N,0}(X_s, V_s) = 
        \begin{cases}
            \mathcal{Z}_N^{-1} \int_{\R^{2d(N-s)}} \ind_{\DN} g_0^{\otimes N}(x_1, \cdots, x_N, v_1, \cdots, v_N) \,dx_{s+1},dv_{s+1} \cdots, dx_{N} dv_N, \quad 1 \leq s < N ,\\
            \mathcal{Z}_N^{-1} \ind_{\DN} g_0^{\otimes N} (Z_N), \quad s = N, \\
            0, \quad s>N,
        \end{cases}
    \end{equation*}
    where the partition function ensures normalization:
    \begin{equation}
        \mathcal{Z}_N = \int_{\R^{2dN}} \ind_{\DN} g_0^{\otimes N}(X_N, V_N) \, dX_N dV_N, \quad N \in \N.
    \end{equation}
\end{definition}
\begin{proposition}
    Let $g_0 \in X_{\beta_0,\mu_0 + 1}$ be a positive probability density with $|g_0|_{\beta_0,\mu_0+1} \leq 1$ and $G_0 = (g_0^{\otimes s})_{s\in \N} \in X_{\infty,\beta_0,\mu_0+1} \subseteq X_{\infty,\beta_0,\mu_0}$. Let $G_{N,0} = (g_{N,0}^{(s)})_{s\in \N}$ be the conditioned BBGKY hierarchy sequence of the tensorized initial data $G_0$ given in Definition \ref{cond bbgky hier}. Then, $G_{N,0}$ is a BBGKY hierarchy sequence approximating $G_0$ in the scaling \eqref{scaling}. For all $(N,\epsilon_2,\cdots, \epsilon_{M+1})$ in the scaling \eqref{scaling}, we have the estimate,
    \begin{equation}
        ||g^{(s)}_{N,0} - g^{\otimes s}_{0} ||_{L^\infty(\mathcal{D}_{s,\bm{\epsilon}_M})} \leq C_{d,s,\beta_0,\mu_0} \epsilon_{M+1}^{1/2} ||G_0||_{\infty, \beta_0, \mu_0},
    \end{equation}
    for an appropriate constant $C_{d,s,\beta_0,\mu_0}$.
\end{proposition}
\begin{proof}
    The proof follows the similar steps as given in \cite{AmpatzoglouPavlovic2020}.
\end{proof}
\subsection{Convergence in observables.}
For $\theta > 0$, we define well-separated spatial configurations for $m \in \N$ as
\begin{equation}
    \Delta^X_m(\theta) = \{X_m \in \R^{dm} \, : |x_i - x_j| > \theta, \quad \forall 1 \leq i < j \leq m  \}, \quad m \geq 2, \quad \Delta_1^X(\theta) = \R^d.
\end{equation}
We define the set of well-separated configurations, including velocity, as
\begin{equation}
    \Delta_m(\theta) = \Delta^X_m(\theta) \times \R^{dm}.
\end{equation}


For $s\in \N$, we define the space of test functions 
\begin{equation}
    C_c(\R^{ds}) = \{\phi_s: \R^{ds} \rightarrow \R : \phi_s \text{ continuous and compactly supported } \}.
\end{equation}
In order to define convergence of observables we must first define the $s$-observable functional. 
\begin{definition}
Consider $T > 0$, $s \in \N$ and $g_s \in C^0\big([0,T], L^\infty(\R^{2ds}) \big)$. Given $\phi_s \in C_c(\R^{ds})$, we define the $s$-observable functional as 
\begin{equation}
    I_{\phi_s}g_s(t)(X_s) = \int_{\R^{ds}} \phi_s(V_s) g_s(t,X_s,V_s) dV_s. 
\end{equation}
\end{definition}
We now can define convergence in observables. 
\begin{definition}
Consider $T > 0$, $N \in \N$, $\bm{G_N} = (g_{N,s})_{s\in \N}, \bm{G} = (g_s)_{s\in \N} \in \Pi_{s=1}^\infty C^0 \big([0,T], L^\infty(\R^{2ds}) \big)$. We say that $(\bm{G_N})_{N \in \N}$ converge in observables to $\bm{G}$, and write
\begin{equation}
    \bm{G_N} \xrightarrow{\sim} \bm{G},
\end{equation}
if there exists $s \in \N$, $\theta >0$, and $\phi_s \in C_c(\R^{ds})$ such that
\begin{equation}
    \lim_{N \rightarrow \infty} ||I_{\phi_s} g_{N,s}(t) - I_{\phi_s} g_s(t)||_{L^\infty(\Delta_s^X(\theta))} = 0, \quad \text{uniformly in } [0,T].
\end{equation}
\end{definition}
 
\subsection{Statement of the main result.}
 \begin{theorem}\label{convergence}
 Let $\beta_0 > 0$, $\mu_0 \in \R$, and fix $T$ as in \ref{time}. Consider initial Boltzmann hierarchy data $F_0 = (f_0^{(s)})_{s\in \N} \in X_{\infty,\beta_0,\mu_0}$ with approximating BBGKY hierarchy sequence $(F_{N,0})_{N\in\N}$. Assume additionally that,
 \begin{itemize}
     \item for each $N \in \N$, $\bm{F_N} \in \bm{X}_{N,\bm{\beta}, \bm{\mu}}$ is the mild solution of the BBGKY hierarchy in $[0,T]$ with initial data $F_{N,0}$.
     \item $\bm{F}\in \bm{X}_{\infty,\bm{\beta},\bm{\mu}}$ is the mild solution of the Boltzmann hierarchy in [0,T] with initial data $F_0$.
     \item $\exists C > 0$ such that for all $\zeta > 0$ there is $q = q(\zeta) > 0$ such that for all $s\in \N$ and for all $Z_s, Z_s' \in \R^{2ds}$ with $|Z_s - Z_s'|<q$, we have 
     \end{itemize}
     \begin{equation}
         |f_0^{(s)}(Z_s) - f_0^{(s)}(Z'_s)| < C^{s-1} \zeta.
     \end{equation}
     Then, $\bm{F_N}$ converges in observables to $\bm{F}$.
 \end{theorem}

\begin{remark}To prove Theorem \ref{convergence} it suffices to prove
$$\|I_s^N(t)-I_s^\infty(t)\|_{L^\infty(\Delta^X_s(\theta))}\overset{N\to\infty}\longrightarrow 0,\text{ uniformly in $[0,T]$},
$$
 for any $s\in\mathbb{N}$, $\phi_s\in C_c(\mathbb{R}^{ds})$ and $\theta >0$, where 
\begin{align}
I_s^N(t)(X_s)&:=I_{\phi_s}f_N^{(s)}(t)(X_s)=\int_{\mathbb{R}^{ds}}\phi_s(V_s)f_N^{(s)}(t,X_s,V_s)\,dV_s, \label{def-I-Ns} \\
I_s^\infty(t)(X_s)&:=I_{\phi_s}f^{(s)}(t)(X_s)=\int_{\mathbb{R}^{ds}}\phi_s(V_s)f^{(s)}(t,X_s,V_s)\,dV_s.  \label{def-I-s}
\end{align}
\end{remark}

 \begin{corollary}\label{cor prop of chaos}
     Let $\beta_0 > 0$, $\mu_0 \in \R$, and $f_0 \in X_{\beta_0,\mu_0+1}$ be a H\"older continuous $C^{0,\gamma}$, $\gamma \in (0,1]$ probability density with $|f_0|_{\beta_0, \mu_0+1}\leq 1/2$. Let $F_0 = (f_0^{\otimes s})_{s \in \N}\in X_{\infty, \beta_0,\mu_0 + 1}$ and $F_{N,0} = (f^{(s)_{N,0}})_{s\in \N}$ be the conditioned BBGKY hierarchy sequence given in Definition \ref{cond bbgky hier} approximating the tensorized data $F_0$. Then, for any $\theta > 0$, $s\in\N$, and $\phi_s \in C_c(\R^{ds})$, we have the rate of convergence
     \begin{equation}
         ||I_{\phi_s} f^{(s)}_N(t) - I_{\phi_s} f^{\otimes s}|(t) ||_{L^\infty(\Delta_s^X(\theta))} = O(\epsilon_{M+1}^r), \quad \text{uniformly in } [0,T],
     \end{equation}
     for any $0<r<\min \{1/2, \gamma \}$, where $\bm{F_N} = (f^{(s)}_N)_{s\in \N} \in \bm{X_{N,\beta,\mu}}$ is the mild solution of the BBGKY hierarchy \eqref{BBGKY heir} in $[0,T]$ with initial data $F_{N,0}$ and $f$ is the mild solution to the Boltzmann equation \eqref{Boltzmann equation} in $[0,T]$ with initial data $f_0$.
 \end{corollary}

\section{Reduction to term by term convergence}\label{sec: series expansion}

Throughout this section,  we consider $\beta_0>0$, $\mu_0\in\mathbb{R}$, the functions $\bm{\beta},\bm{\mu}:[0,T]\to\mathbb{R}$ defined by \eqref{beta mu given lambda}, $(N, \epsilon_2, \cdots, \epsilon_{M+1})$ in the scaling \eqref{scaling} and initial data $F_{N,0}\in X_{N,\beta_0,\mu_0}$, $F_0\in X_{\infty,\beta_0,\mu_0}$. Let $\bm{F_N}=(f_N^{(s)})_{s\in\mathbb{N}}\in\bm{X}_{N,\bm{\beta},\bm{\mu}}$, $\bm{F}=(f^{(s)})_{s\in\mathbb{N}}\in\bm{X}_{\infty,\bm{\beta},\bm{\mu}}$ be the mild solutions of the corresponding BBGKY and Boltzmann hierarchies, respectively, in $[0,T]$, given by Theorems \ref{well posedness BBGKY} and Theorem \ref{well posedness boltzmann}. Let us note that by \eqref{beta mu given lambda}, we obtain
\begin{equation}\label{non dependence}
\bm{\beta}(T)=\frac{\beta_0}{2},\quad\bm{\mu}(T)=\mu_0-\frac{\beta_0}{2},
\end{equation}
thus $\bm{\beta}(T),\bm{\mu}(T)$ do not depend on $T$.

For convenience, we introduce the following notation. Given $k\in\mathbb{N}$ and $t\geq 0$, we denote
\begin{equation}\label{collision times}
\mathcal{T}_k(t):=\left\{(t_1,...,t_k)\in\mathbb{R}^k:0\leq t_k <...< t_1\leq t\right\}.
\end{equation}
Given $k\geq 1$, we denote
\begin{equation}\label{S_k}
S_k:=\left\{\sigma=(\sigma_1,...,\sigma_k):\sigma_i\in\left\{1,\cdots, M\right\},\quad\forall i=1,...,k\right\}.
\end{equation}
Notice that the cardinality of $S_k$ is given by:
\begin{equation}\label{cardinality of S_k}
|S_k|=M^k,\quad\forall k\geq 1.
\end{equation}
Given $k\in\mathbb{N}$ and $\sigma\in S_k$, for any $1\leq j \leq k$ we write
\begin{equation}\label{sigma tilde}
\widetilde{\sigma}_j=\sum_{i=1}^j \sigma_i.
\end{equation}
We also write
$
\widetilde{\sigma}_0:=0.
$
Notice that
\begin{equation}\label{bound on sigma}
k\leq\widetilde{\sigma}_k\leq  M k,\quad\forall k\in\mathbb{N}.
\end{equation}
\subsection{Series expansion}
Now, we make a series expansion for the mild solution $\bm{F_N}=(f_N^{(s)})_{s\in\mathbb{N}}$ of the BBGKY hierarchy with respect to the initial data $F_{N,0}$. By Definition \ref{def of mild bbgky}, for any $s \in\mathbb{N}$, we have  Duhamel's formula:
\begin{equation*}
f^{(s)}_N(t)=T_s^tf^{(s)}_{N,0}+\int_0^tT_s^{t-t_1}\left[\sum_{\ell =1}^{M}\mathcal{C}_{s,s+\ell}^{N}f^{(s+\ell)}_N \right](t_1)\,dt_1.
\end{equation*}
Let $n\in\mathbb{N}$. Iterating  Duhamel's formula $n$-times,  we obtain
\begin{equation}\label{function plus remainder BBGKY}
f^{(s)}_N(t)=\sum_{k=0}^nf^{(s,k)}_N(t)+R_{N}^{(s,n+1)}(t),
\end{equation}
where we use the notation:
\begin{equation}\label{function expansion BBGKY}
\begin{aligned}
  f^{(s,k)}_N(t)&:=\sum_{\sigma\in S_k}f^{(s,k,\sigma)}_N(t),\text{ for } 1\leq k\leq n,\quad f^{(s,0)}_N(t):=T_s^tf^{(s)}_{N,0}.
\end{aligned}
\end{equation}

\begin{equation}\label{function expansion with indeces BBGKY}
\begin{aligned}
 f^{(s,k,\sigma)}_N(t)=\int_{\mathcal{T}_k(t)}T_s^{t-t_1}\mathcal{C}_{s,s+\widetilde{\sigma}_1}^{N}T_{s+\widetilde{\sigma}_1}^{t_1-t_2}\mathcal{C}_{s+\widetilde{\sigma}_1,s+\widetilde{\sigma}_2}^{N}T_{s+\widetilde{\sigma}_2}^{t_2-t_3}...
T_{s+\widetilde{\sigma}_{k-1}}^{t_{k-1}-t_k}\mathcal{C}_{s+\widetilde{\sigma}_{k-1},s+\widetilde{\sigma}_k}^{N}T_{s+\widetilde{\sigma}_k}^{t_k}f_{N,0}^{(s+\widetilde{\sigma}_k)}\,dt_k...\,dt_1,
\end{aligned}
\end{equation}
\begin{equation}\label{remainder BBGKY}
\begin{aligned}
R_N^{(s,n+1)}(t):=\sum_{\sigma\in S_{n+1}}R_N^{(s,n+1,\sigma)}(t),
\end{aligned}
\end{equation}

\begin{equation}\label{remainder BBGKY with indeces}
\begin{aligned}
R_N^{(s,n+1,\sigma)}(t):=\int_{\mathcal{T}_{n+1}(t)}T_s^{t-t_1}\mathcal{C}_{s,s+\widetilde{\sigma}_1}^NT_{s+\widetilde{\sigma}_1}^{t_1-t_2}\mathcal{C}_{s+\widetilde{\sigma}_1,s+\widetilde{\sigma}_2}^NT_{s+\widetilde{\sigma}_2}^{t_2-t_3}...&\\
T_{s+\widetilde{\sigma}_{n-1}}^{t_{n-1}-t_n}\mathcal{C}_{s+\widetilde{\sigma}_{n-1},s+\widetilde{\sigma}_n}^NT_{s+\widetilde{\sigma}_n}^{t_n-t_{n+1}}\mathcal{C}_{s+\widetilde{\sigma}_n,s+\widetilde{\sigma}_{n+1}}^Nf^{(s+\widetilde{\sigma}_{n+1})}_N(t_{n+1})\,dt_{n+1}\,dt_n...\,dt_1.
&
\end{aligned}
\end{equation}

Similarly by iterating Duhamel's formula for the mild solution of the Boltzmann hierarchy, see Definition \ref{def mild solution boltzmann}, 
\begin{equation}\label{function plus remainder bh}
f^{(s)}(t)=\sum_{k=0}^nf^{(s,k)}(t)+R^{(s,n+1)}(t),
\end{equation}
where we use the notation:
\begin{equation}\label{function expansion bh}
\begin{aligned}
  f^{(s,k)}(t)&:=\sum_{\sigma\in S_k}f^{(s,k,\sigma)}(t),\text{ for } 1\leq k\leq n,\quad f^{(s,0)}(t):=T_s^tf^{(s)}_{0}.
\end{aligned}
\end{equation}

\begin{equation}\label{function expansion with indeces bh}
\begin{aligned}
 f^{(s,k,\sigma)}(t)=\int_{\mathcal{T}_k(t)}S_s^{t-t_1}\mathcal{C}_{s,s+\widetilde{\sigma}_1}^{\infty}S_{s+\widetilde{\sigma}_1}^{t_1-t_2}\mathcal{C}_{s+\widetilde{\sigma}_1,s+\widetilde{\sigma}_2}^{\infty}S_{s+\widetilde{\sigma}_2}^{t_2-t_3}...
S_{s+\widetilde{\sigma}_{k-1}}^{t_{k-1}-t_k}\mathcal{C}_{s+\widetilde{\sigma}_{k-1},s+\widetilde{\sigma}_k}^{\infty}S_{s+\widetilde{\sigma}_k}^{t_k}f_{0}^{(s+\widetilde{\sigma}_k)}\,dt_k...\,dt_1,
\end{aligned}
\end{equation}
\begin{equation}\label{remainder bh}
\begin{aligned}
R^{(s,n+1)}(t):=\sum_{\sigma\in S_{n+1}}R^{(s,n+1,\sigma)}(t),
\end{aligned}
\end{equation}
\begin{equation}\label{remainder bh with indeces}
\begin{aligned}
R^{(s,n+1,\sigma)}(t):=\int_{\mathcal{T}_{n+1}(t)}S_s^{t-t_1}\mathcal{C}_{s,s+\widetilde{\sigma}_1}^\infty S_{s+\widetilde{\sigma}_1}^{t_1-t_2}\mathcal{C}_{s+\widetilde{\sigma}_1,s+\widetilde{\sigma}_2}^\infty S_{s+\widetilde{\sigma}_2}^{t_2-t_3}...&\\
S_{s+\widetilde{\sigma}_{n-1}}^{t_{n-1}-t_n}\mathcal{C}_{s+\widetilde{\sigma}_{n-1},s+\widetilde{\sigma}_n}^\infty S_{s+\widetilde{\sigma}_n}^{t_n-t_{n+1}}\mathcal{C}_{s+\widetilde{\sigma}_n,s+\widetilde{\sigma}_{n+1}}^\infty f^{(s+\widetilde{\sigma}_{n+1})}(t_{n+1})\,dt_{n+1}\,dt_n...\,dt_1.
&
\end{aligned}
\end{equation}

Given $\phi_s\in C_c(\mathbb{R}^{ds})$ and $k\in\mathbb{N}$, let us denote 
\begin{equation}\label{bbgky observ k}
I_{s,k}^N(t)(X_s):=
\int_{\mathbb{R}^{ds}}\phi_s(V_s)f_N^{(s,k)}(t,X_s,V_s)\,dV_s,
\end{equation}
\begin{equation}\label{boltz observ k}
I_{s,k}^\infty(t)(X_s):=
\int_{\mathbb{R}^{ds}}\phi_s(V_s)f^{(s,k)}(t,X_s,V_s)\,dV_s.
\end{equation}
We obtain the following estimates which, thanks to the a priori bounds in Section \ref{sec:local}, are independent of $M$.
\begin{lemma}\label{term by term} For any $s,n\in\mathbb{N}$ and $t\in [0,T]$, the following estimates hold:
\begin{equation*}\|I_s^N(t)-\sum_{k=0}^nI_{s,k}^N(t)\|_{L^\infty_{X_s}}\leq 2^{-(n+1)}  C_{s,\beta_0,\mu_0} \|\phi_s\|_{L^\infty_{V_s}} \|F_{N,0}\|_{N,\beta_0,\mu_0},
\end{equation*}
\begin{equation*}\|I_s^\infty(t)-\sum_{k=0}^nI_{s,k}^\infty(t)\|_{L^\infty_{X_s}}\leq 2^{-(n+1)}  C_{s,\beta_0,\mu_0} \|\phi_s\|_{L^\infty_{V_s}}\|F_{0}\|_{\infty,\beta_0,\mu_0},
\end{equation*}
where the observables $I_s^N$, $I_s^\infty$ defined in \eqref{def-I-Ns}-\eqref{def-I-s}.
\end{lemma}
\begin{proof}
For the proof see \cite{WarnerThesis}.
\end{proof}
\subsection{High energy truncation}
We now truncate energy and focus on bounded energy domains. 
Let us fix $s,n\in\mathbb{N}$ and  $R>1$. As usual, we denote $B_R^{2d}$ to be the $2d$-ball of radius $R$ centered at the origin. 

We first define the truncated BBGKY and Boltzmann  collisional operators. For $s \in\mathbb{N}$ we define,
\begin{equation}\label{velocity truncation of operators}
\mathcal{C}_{s,s+\ell}^{N,R}g_{s+\ell} :=\mathcal{C}_{s,s+\ell}^N(g_{s+\ell}\mathds{1}_{[E_{s+\ell}\leq R^2]}), \quad \ell \in \{1,\cdots, M \},
\end{equation}
\begin{equation}\label{velocity truncation of operators boltzmann}
\mathcal{C}_{s,s+\ell}^{\infty,R}g_{s+\ell} :=\mathcal{C}_{s,s+\ell}^\infty(g_{s+\ell}\mathds{1}_{[E_{s+\ell}\leq R^2]}), \quad \ell \in \{1,\cdots, M \}.
\end{equation}
For the BBGKY hierarchy, we define
\begin{equation*}
f_{N,R}^{(s,k)}(t,Z_s):=\sum_{\sigma\in S_k}f_{N,R}^{(s,k,\sigma)}(t,Z_s),\text{ for }1\leq k\leq n,\quad f_{N,R}^{(s,0)}(t,Z_s):=T_s^t(f_{N,0}\mathds{1}_{[E_s\leq R^2]})(Z_s),
\end{equation*}
where given $k\geq 1$ and $\sigma\in S_k$, we denote
\begin{equation*}
\begin{aligned}
f_{N,R}^{(s,k,\sigma)}(t,Z_s&):=\int_{\mathcal{T}_k(t)}T_s^{t-t_1}\mathcal{C}_{s,s+\widetilde{\sigma}_1}^{N,R} T_{s+\widetilde{\sigma}_1}^{t_1-t_2}...\mathcal{C}_{s+\widetilde{\sigma}_{k-1},s+\widetilde{\sigma}_k}^{N,R} T_{s+\widetilde{\sigma}_k}^{t_k}f_{N,0}^{(s+\widetilde{\sigma}_k)}(Z_s)\,dt_k...\,dt_{1}.
\end{aligned}
\end{equation*}
We can similarly define the truncated Boltzmann hierarchy. 
\begin{equation*}
f_{R}^{(s,k)}(t,Z_s):=\sum_{\sigma\in S_k}f_{R}^{(s,k,\sigma)}(t,Z_s),\text{ for }1\leq k\leq n,\quad f_{R}^{(s,0)}(t,Z_s):=S_s^t(f_{0}\mathds{1}_{[E_s\leq R^2]})(Z_s),
\end{equation*}
where given $k\geq 1$ and $\sigma\in S_k$, we denote
\begin{align*}
f_{R}^{(s,k,\sigma)}(t,Z_s&):=\int_{\mathcal{T}_k(t)}S_s^{t-t_1}\mathcal{C}_{s,s+\widetilde{\sigma}_1}^{\infty,R} S_{s+\widetilde{\sigma}_1}^{t_1-t_2}...\mathcal{C}_{s+\widetilde{\sigma}_{k-1},s+\widetilde{\sigma}_k}^{\infty,R} S_{s+\widetilde{\sigma}_k}^{t_k} f_{0}^{(s+\widetilde{\sigma}_k)}(Z_s)\,dt_k...\,dt_{1}.
\end{align*}

Given $\phi_s\in C_c(\mathbb{R}^{ds})$ and $k\in\mathbb{N}$, let us denote
\begin{equation}\label{BBGKY bounded energy}
I_{s,k,R}^N(t)(X_s):=
\int_{\mathbb{R}^{ds}}\phi_s(V_s)f_{N,R}^{(s,k)}(t,X_s,V_s)\,dV_s=\int_{B_R^{ds}}\phi_s(V_s)f_{N,R}^{(s,k)}(t,X_s,V_s)\,dV_s,
\end{equation}
\begin{equation}\label{Botlzmann bounded energy}
I_{s,k,R}^\infty(t)(X_s):=
\int_{\mathbb{R}^{ds}}\phi_s(V_s)f_R^{(s,k)}(t,X_s,V_s)\,dV_s=\int_{B_R^{ds}}\phi_s(V_s)f_{R}^{(s,k)}(t,X_s,V_s)\,dV_s.
\end{equation}
Recalling the observables $I_{s,k}^N$, $I_{s,k}^\infty$, defined in \eqref{bbgky observ k}-\eqref{boltz observ k}, 
we obtain the following estimates:
\begin{lemma}\label{energy truncation} For any $s,n\in\mathbb{N}$, $R>1$ and $t\in [0,T]$, the following estimates hold:
\begin{align}\label{bbgky trunc energy est}
\sum_{k=0}^n\|I_{s,k,R}^N(t)-I_{s,k}^N(t)\|_{L^\infty_{X_s}} &\leq C_{s,\beta_0,\mu_0,T} e^{-\frac{\beta_0}{3}R^2}\|\phi_s\|_{L^\infty_{V_s}} \|F_{N,0}\|_{N,\beta_0,\mu_0}, \\
\label{boltz trunc energy est} \sum_{k=0}^n\|I_{s,k,R}^\infty(t)-I_{s,k}^\infty(t)\|_{L^\infty_{X_s}} &\leq C_{s,\beta_0,\mu_0,T}  e^{-\frac{\beta_0}{3}R^2}\|\phi_s\|_{L^\infty_{V_s}}\|F_{0}\|_{\infty,\beta_0,\mu_0}.
\end{align}
\end{lemma}
\begin{proof}
The proof follows in a similar way as in \cite{AmpatzoglouPavlovic2020}.
\end{proof}

\subsection{Separation of collision times} 
We will now separate the time intervals we are integrating at, so that collisions occurring are separated in time. For this purpose consider a small time parameter $\delta>0$. 

For convenience, given $t\geq 0$ and $k\in\mathbb{N}$, we define  
\begin{equation}\label{separated collision times}
\mathcal{T}_{k,\delta}(t):=\left\{(t_1,...,t_k)\in\mathcal{T}_k(t):\quad 0\leq t_{i+1}\leq t_i-\delta,\quad\forall i\in [0,k]\right\},
\end{equation}
where we denote $t_{k+1}=0$, $t_0=t$.

For the BBGKY hierarchy, we define
\begin{equation*}
f_{N,R,\delta}^{(s,k)}(t,Z_s):=\sum_{\sigma\in S_k}f_{N,R,\delta}^{(s,k,\sigma)}(t,Z_s),\text{ for } 1\leq k\leq n,\quad f_{N,R,\delta}^{(s,0)}(t,Z_s):=T_s^t(f_{N,0}\mathds{1}_{[E_s\leq R^2]})(Z_s),
\end{equation*}
where, given $k\geq 1$ and $\sigma\in S_k$, we denote
\begin{equation*}
f_{N,R,\delta}^{(s,k,\sigma)}(t,Z_s):=\int_{\mathcal{T}_{k,\delta}(t)} T_s^{t-t_1}\mathcal{C}_{s,s+\widetilde{\sigma}_1}^{N,R} T_{s+\widetilde{\sigma}_1}^{t_1-t_2}
...\mathcal{C}_{s+\widetilde{\sigma}_{k-1},s+\widetilde{\sigma}_k}^{N,R} T_{s+\widetilde{\sigma}_k}^{t_k}f_{N,0}^{(s+\widetilde{\sigma}_k)}(Z_s)\,dt_k,...\,dt_{1}.
\end{equation*}
In the same spirit, for the Boltzmann hierarchy we define
\begin{equation*}
f_{R,\delta}^{(s,k)}(t,Z_s):=\sum_{\sigma\in S_k}f_{R,\delta}^{(s,k,\sigma)}(t,Z_s),\text{ for } 1\leq k\leq n,\quad f_{R,\delta}^{(s,0)}(t,Z_s):=S_s^t(f_{0}\mathds{1}_{[E_s\leq R^2]})(Z_s),
\end{equation*}
where, given $k\geq 1$ and $\sigma\in S_k$, we denote
\begin{equation*}
f_{R,\delta}^{(s,k,\sigma)}(t,Z_s):=\int_{\mathcal{T}_{k,\delta}(t)}S_s^{t-t_1}\mathcal{C}_{s,s+\widetilde{\sigma}_1}^{\infty,R} S_{s+\widetilde{\sigma}_1}^{t_1-t_2}
...\mathcal{C}_{s+\widetilde{\sigma}_{k-1},s+\widetilde{\sigma}_k}^{\infty,R} S_{s+\widetilde{\sigma}_k}^{t_m}f_{0}^{(s+\widetilde{\sigma}_k)}(Z_s)\,dt_k,...\,dt_{1}.
\end{equation*}
Given $\phi_s\in C_c(\mathbb{R}^{ds})$ and $k\in\mathbb{N}$, we define
\begin{equation}\label{bbgky truncated time}
I_{s,k,R,\delta}^N(t)(X_s):=
\int_{\mathbb{R}^{ds}}\phi_s(V_s)f_{N,R,\delta}^{(s,k)}(t,X_s,V_s)\,dV_s=\int_{B_R^{ds}}\phi_s(V_s)f_{N,R,\delta}^{(s,k)}(t,X_s,V_s)\,dV_s,
\end{equation}
\begin{equation}\label{boltzmann truncated time}
I_{s,k,R,\delta}^\infty(t)(X_s):=
\int_{\mathbb{R}^{ds}}\phi_s(V_s)f_{R,\delta}^{(s,k)}(t,X_s,V_s)\,dV_s=\int_{B_R^{ds}}\phi_s(V_s)f_{R,\delta}^{(s,k)}(t,X_s,V_s)\,dV_s.
\end{equation}
\begin{remark}\label{time functionals trivial} For $0\leq t\leq\delta$, we trivially obtain $\mathcal{T}_{k,\delta}(t)=\emptyset$. In this case the functionals $I_{s,k,R,\delta}^N(t), I_{s,k,R,\delta}^\infty(t)$ are identically zero.
\end{remark}

Recalling the observables $I_{s,k,R}^N$, $I_{s,k,R}^\infty$ defined in \eqref{BBGKY bounded energy}-\eqref{Botlzmann bounded energy},  we obtain the following estimates:
\begin{lemma}\label{time sep}
For any $s,n\in\mathbb{N}$, $R>0$, $\delta>0$ and $t\in[0,T]$, the following estimates hold:
\begin{equation*}
\sum_{k=0}^n\|I_{s,k,R,\delta}^N(t)-I_{s,k,R}^N(t)\|_{L^\infty_{X_s}}\leq \delta C_{d,s,\beta_0,\mu_0,T,M}^n\|\phi_s\|_{L^\infty_{V_s}}\|F_{N,0}\|_{N,\beta_0,\mu_0},
\end{equation*}
\begin{equation*}\sum_{k=0}^n\|I_{s,k,R,\delta}^\infty(t)-I_{s,k,R}^\infty(t)\|_{L^\infty_{X_s}}\leq \delta C_{d,s,\beta_0,\mu_0,T,M}^n \|\phi_s\|_{L^\infty_{V_s}}\|F_{0}\|_{\infty,\beta_0,\mu_0}.
\end{equation*}
\end{lemma}
\begin{proof}
The proof follows the same idea as Lemma 8.7 in \cite{AmpatzoglouThesis} with the difference that the constant term $C_{d,s,\beta_0,\mu_0,T,M}^n$ additionally depends on the maximum collision order $M$.

\end{proof}
Combining Lemma \ref{term by term}, Lemma  \ref{energy truncation} and Lemma \ref{time sep}, we obtain
\begin{proposition}\label{reduction}
For any $s,n\in\mathbb{N}$, $R>1$, $\delta>0$ and $t\in[0,T]$, the following estimates hold:
\begin{equation*}
\begin{aligned}
\|I_s^N(t)-\sum_{k=1}^nI_{s,k,R,\delta}^N(t)\|_{L^\infty_{X_s}}\leq C_{s,\beta_0,\mu_0,T}\|\phi_s\|_{L^\infty_{V_s}}\left(2^{-(n+1)}+e^{-\frac{\beta_0}{3}R^2}+\delta C_{d,s,\beta_0,\mu_0,T,M}^n\right)\|F_{N,0}\|_{N,\beta_0,\mu_0},
\end{aligned}
\end{equation*}
\begin{equation*}
\begin{aligned}
\|I_s^\infty(t)-\sum_{k=1}^nI_{s,k,R,\delta}^\infty(t)\|_{L^\infty_{X_s}}\leq C_{s,\beta_0,\mu_0,T}\|\phi_s\|_{L^\infty_{V_s}}\left(2^{-(n+1)}+e^{-\frac{\beta_0}{3}R^2}+\delta C_{d,s,\beta_0,\mu_0,T,M}^n\right)\|F_{0}\|_{\infty,\beta_0,\mu_0}.
\end{aligned}
\end{equation*}
\end{proposition}
Therefore, the convergence proof reduces to controlling the differences of the observables
$I_{s,k,R,\delta}^N(t)-I_{s,k,R,\delta}^\infty(t),$
 given by \eqref{bbgky truncated time}-\eqref{boltzmann truncated time}. In order for the backwards $(\epsilon_2, \cdots, \epsilon_{M+1})$-flow and the backwards free flow to be comparable, we must eliminate a small measure set of initial data and show that this set is negligible in the limit.

\section{Geometric Estimates}\label{sec::geometric estimates}
 In this section we formulate important general geometric results which will be used extensively for estimating the pathological sets that arise in the following section (see Proposition \ref{bad set triple} in Section \ref{sec:stability}).

\subsection{Ellipsoidal estimates.}  We now move our attention to proving estimates on general ellipsoids. Due to the prevalence of ellipsoidal estimates in this paper as compared to previous work done in \cite{AmpatzoglouPavlovic2020} and \cite{ternary}, we start by proving a few general results which will be used for ellipsoidal estimates of several geometric objects. 
\subsubsection{General ellipsoid results}
We begin by providing the definition of a $(d-1)$-dimensional ellipsoid.
\begin{definition}\label{ellipsoid defn}
    For $d\geq 2$, a set  $\mathbb{E}^{d-1} \subset \R^d$ of the form,
    \begin{equation}\label{A rep of ellipsoid}
        \mathbb{E}^{d-1} = \{x\in \R^d \, : \, \langle x, Ax \rangle = c \},
    \end{equation}
    for some $c > 0$ and positive definite matrix $A \in \R^{d\times d}$, is called a $(d-1)$-dimensional ellipsoid.
\end{definition}
In particular, we will need to study a specific class of ellipsoids which take the following form. 
\begin{definition}
    For $d \geq 2$ and $\ell \in \N$ we say that an $(\ell d -1)$-dimensional ellipsoid $\mathbb{E}^{\ell d -1}$ is an $\ell$-block ellipsoid if its corresponding positive definite matrix $A\in \R^{\ell d \times \ell d}$ given in \eqref{A rep of ellipsoid} is a block matrix of the form
    \begin{equation}
        A = (a_{i,j} I_d)_{1\leq i,j \leq \ell}.
    \end{equation}
\end{definition}

\begin{lemma}\label{decomp}
    Let $\mathbb{E}^{d-1}$ be a $(d-1)$-dimensional ellipsoid. For some invertible lower triangular matrix $L: \R^d \rightarrow \R^d$ and $c>0$, $\mathbb{E}^{d-1}$ can be written in the form,
    \begin{equation}
        \mathbb{E}^{d-1} = \{x\in \R^d \, : \, \langle Lx, Lx \rangle = c \}.
    \end{equation}
\end{lemma}
\begin{proof}
    By Definition \ref{ellipsoid defn}, there exists some positive definite matrix $A \in \R^{d\times d}$ such that,
    \begin{equation}
        \mathbb{E}^{d-1} = \{x\in \R^d \, : \, \langle x, Ax \rangle = c \}.
    \end{equation}
    Since $A$ is positive definite, it has the decomposition
    \begin{equation}
        A = L^T L,
    \end{equation}
    for some invertible lower triangular matrix $L\in \R^{d\times d}$. Therefore, by the properties of the inner product, for all $x \in \R$ we have
    \begin{align}
        \langle x, A x\rangle &= \langle x, L^T L x \rangle, \\
        &= \langle L x, L x \rangle,
    \end{align}
    completing the proof.
\end{proof}
\begin{corollary}
     Let $\mathbb{E}^{\ell d-1}$ be a $(\ell d-1)$-dimensional $\ell$-block ellipsoid. For some invertible lower triangular block matrix $L: \R^{\ell d} \rightarrow \R^{\ell d}$ and $c>0$, $\mathbb{E}^{\ell d-1}$ can be written in the form,
    \begin{equation}
        \mathbb{E}^{\ell d-1} = \{\bm{x} \in \R^{\ell d} \, : \, \langle L\bm{x}, L\bm{x} \rangle = c \}
    \end{equation}
    where $L$ is of the form
    \begin{equation}
        L = 
        \begin{pmatrix}
            b_{1,1}I_d &  &  0 \\
            \vdots &  \ddots& & \\
            b_{\ell,1} I_d &  \cdots & b_{\ell,\ell}I_d
        \end{pmatrix}.
    \end{equation}
\end{corollary}
We now will present a lemma which will allow us to perform a change of variables from the ellipsoid to the sphere leaving any particular coordinate invariant up to a rescaling. Moreover, we show that the ellipsoidal structure remains invariant under invertible linear transformations.
\begin{lemma}\label{ellipsoid to sphere and trans}
    Let $\mathbb{E}^{d-1}\subset \R^d$ be a $(d-1)$-dimensional ellipsoid. The following holds,
    \begin{enumerate}
        \item For any set $A \subset \R^d$ and any $i \in \{ 1, \cdots, d\}$, there exists a bijective linear map $T_i: \mathbb{E}^{d-1} \rightarrow \mathbb{S}_1^{d-1}$ that leaves the $i^{\text{th}}$ component invariant up to a rescaling such that
        \begin{equation}\label{ellipsoid to sphere}
            \int_{\mathbb{E}^{d-1}} \ind_{A} ({x}) d{x} = |\det T_i|^{-1} \int_{\mathbb{S}^{d-1}_1} \ind_{T_i(A)}({y})  d{y}.
        \end{equation}
        That is, if ${x} = (x_1,\cdots, x_d) \in \mathbb{E}^{d-1}$ and $T_i(x) = {y}= (y_1,\cdots,y_d) \in \mathbb{S}_1^{d-1} $, then $x_i = k y_i$ for some $k\in \R$.
        \item For any invertible linear transformation $S: \R^d \rightarrow \R^d$, $S(\mathbb{E}^{d-1})$ is also a $(d-1)$-dimensional ellipsoid.
    \end{enumerate}
\end{lemma}
\begin{proof}
We first will prove $(1)$.
We fix $i\in \{1,\cdots,d\}$ and define the permutation map $P_{1,i}: \R^d \rightarrow \R^d$ where $P_{1,i}({x})$ switches the first and $i^{\text{th}}$ components of ${x} \in \R^d$. We note that $P_{1,i}$ is its own inverse and that $P_{1,i}$ leaves the ellipsoid invariant. That is, since we have that for all $x,y \in \R^d$
\begin{equation}
    \langle P_{1,i} x , P_{1,i} y \rangle = \langle x, y \rangle,
\end{equation}
for any ellipsoid $\tilde{\mathbb{E}}^{d-1}$ we have $P_{1,i}(\tilde{\mathbb{E}}^{d-1}) = \tilde{\mathbb{E}}^{d-1} $. 
\par 
Fix $L:\R^d \rightarrow \R^d$ as in Lemma \ref{decomp}. We note that the map $L^{-1}$ is an invertible lower triangular matrix which sends $ \mathbb{E}^{d - 1} \rightarrow \mathbb{S}^{d -1 }_{\sqrt{c}}$. By rescaling by $\sqrt{c}$, we have
\begin{equation}
    \sqrt{c} L^{-1}: \mathbb{E}^{d-1} \rightarrow \mathbb{S}^{d-1}_1.
\end{equation}
 We define the invertible map $T_i = P_{1,i} \circ (\sqrt{c} L^{-1}) \circ P_{1,i}: \mathbb{E}^{d-1} \rightarrow \mathbb{S}^{d-1}_1$. Note that since $L^{-1}$ leaves the first component invariant up to a rescaling, $T_i$ and $T^{-1}_i$ leaves the $i^{\text{th}}$ component invariant up to a rescaling. Applying the change of variables $y = T_i^{-1}x$, for any measurable $A\subset \R^d$ we obtain,
\begin{align}
    \int_{\mathbb{E}^{d-1}} \ind_A({x}) d{x}
    &= |\det T_i|^{-1} \int_{\mathbb{S}^{d-1}_1} \ind_{T_i(A)}({y})  d{y}.
\end{align}
\par
We now move on to the proof of $(2)$.  Recall the representation of the $\mathbb{E}^{d-1}$ given by,
    \begin{equation}
       \mathbb{E}^{d-1} =  \{ x\in \R^d \, : \, \langle x, Ax \rangle = c \}
    \end{equation}
    for $c > 0 $. Let $S$ be an invertible linear transformation. We wish to show that
    \begin{equation}
        \{ Sx \, : \, x \in \mathbb{E}^{d-1} \}
    \end{equation}
    is an ellipsoid. 
\par
Let $B = (S^{-1})^T A S^{-1}$. It is easy to show that B is a positive definite matrix. Furthermore, we have 
\begin{align}
    \langle Sx, B Sx \rangle = \langle Sx, (S^{-1})^T A x \rangle 
    = \langle x, Ax \rangle
    = c,
\end{align}
completing the proof. 
\end{proof}
We now present the $\ell$-block corollary which will be particularly useful in our setting.
\begin{corollary}\label{ellipsoid to sphere and trans block}
    Let $\mathbb{E}^{\ell d -1}$ be an $\ell$-block ellipsoid. Then the following holds,
    \begin{enumerate}
        \item For any set $A \subset \R^{\ell d}$ and any $i \in \{1,\cdots, \ell \}$, there exists a bijective linear map $T_i: \mathbb{E}^{\ell d -1} \rightarrow \mathbb{S}^{\ell d -1}_{1}$ that leaves the $i^{\textit{th}}$ block invaraint up to a rescaling such that, 
        \begin{equation} \label{ellipsoid to sphere block}
            \int_{\mathbb{E}^{d-1}} \ind_{A} (\bm{x}) d\bm{x} = |\det T_i|^{-1} \int_{\mathbb{S}^{d-1}_1} \ind_{T_i(A)}(\bm{y})  d\bm{y}.
        \end{equation}
        That is, if for $1\leq i \leq \ell$  we let $x_i, y_i\in \R^d$ and define $\bm{x} = (x_1,\cdots, x_\ell) \in \mathbb{E}^{\ell d-1}$  and $T_i(\bm{x}) = \bm{y}= (y_1,\cdots,y_\ell) \in \mathbb{S}_1^{\ell d-1} $, then $x_i = k y_i$ for some $k\in \R$.
        \item For any invertible $\ell$-block matrix $S: \R^{\ell d} \rightarrow \R^{\ell d}$, $S(\mathbb{E}^{\ell d-1})$ is also a $(\ell d-1)$-dimensional $\ell$-block ellipsoid.
    \end{enumerate}
\end{corollary}
\begin{proof}
    The proof of this corollary follows the same process from the proof of Corollary \ref{ellipsoid to sphere and trans block} by a simple replacement of matrices with their $\ell$-block counterparts.
\end{proof}

\subsubsection{Intersection of the ellipsoid with the spherical cap}
We now will present the ellipsoidal estimates that take place on geometric objects of interest. 
\par
We first adapt two conic estimates given in \cite{ternary}. The estimates given in \cite{ternary} hold over the sphere, but for our setting we must show that similar estimates hold over the general $(\ell d - 1)$-dimensional  $\ell$-block ellipsoid $\mathbb{E}^{\ell d -1}$. The two spherical conic estimates which are proved in \cite{ternary} are given below.
\begin{lemma} \label{conic est sphere}
    Let $0 \leq \alpha \leq 1$ and $\nu \in \R^d\setminus \{ 0\}$. Define,
    \begin{equation}
        S(\alpha, \nu) = \{ \omega \in \R^d  \,:\,  |\langle \omega, \nu \rangle| \geq \alpha |\omega||\nu|  \}.
    \end{equation}
    For any $r>0$, the following estimate holds,
    \begin{equation}
        \int_{\mathbb{S}_r^{d-1}} \mathds{1}_{S(\alpha, \nu)}(\omega) d\omega = r^{d-1} |\mathbb{S}_1^{d-2}| \int_0^{\arccos{\alpha}} \sin^{d-2}{\theta} d\theta \lesssim r^{d-1} \arccos \alpha.
    \end{equation}
    
\end{lemma}

\begin{lemma} \label{conic est N sphere}
    Consider $0\leq \alpha \leq 1$ and $\nu \in \R^d \setminus \{0\}$. Define,
    \begin{equation}
        N(\alpha, \nu) = \{ (\omega_1,\omega_2) \in \R^{2d} \, :\, \langle \omega_1 - \omega_2, \nu\rangle \geq \alpha |\omega_1- \omega_2||\nu|  \}.
    \end{equation}
    For any $r>0$, the following estimate holds,
    \begin{equation}
        \int_{\mathbb{S}_r^{2d-1}} \ind_{N(\alpha,\nu)}(\omega_1,\omega_2) d\omega_1 d\omega_2 \lesssim \arccos{\alpha}.
    \end{equation}
    
\end{lemma}
We now state the corresponding ellipsoidal estimates in the following two Lemmas.
\begin{lemma} \label{conic est}
    Let $0 \leq \alpha \leq 1$, $\nu \in \R^d\setminus \{ 0\}$, and $\ell \in \N$. Define,
    \begin{equation}
        S(\alpha, \nu) = \{ \omega \in \R^d  \,:\,  |\langle \omega, \nu \rangle| \geq \alpha |\omega||\nu|  \},
    \end{equation}
    and let $\mathbb{E}^{\ell d -1}$ denote an $(\ell d - 1)$-dimensional $\ell$-block ellipsoid. Then, the following estimate holds for all $1\leq i \leq \ell $,
    \begin{equation}
        \int_{\mathbb{E}^{\ell d -1}} \mathds{1}_{S(\alpha, \nu)}(\omega_i) d\bm{\omega}\lesssim \arccos \alpha.
    \end{equation}
\end{lemma}
\begin{proof}
    By applying \eqref{ellipsoid to sphere block}, there exists an invertible linear transformation $T_i: \mathbb{E}^{\ell d -1} \rightarrow \mathbb{S}^{\ell d -1}_1$ that keeps $\omega_i$ invariant up to a rescaling. We note that $S(\alpha, \nu)(k \omega_i) = S(\alpha,\nu)(\omega_i)$. Then, by \eqref{ellipsoid to sphere block} we have
    \begin{align}
        \int_{\mathbb{E}^{\ell d -1}} \ind_{S(\alpha,\nu)}(\omega_i) d\bm{\omega} &\lesssim \int_{\mathbb{S}^{\ell d -1}_1} \ind_{T_i(S(\alpha,\nu))}(\eta_i) d\bm{\eta} \\
         &= \int_{\mathbb{S}^{\ell d -1}_1} \ind_{S(\alpha,\nu)}(\eta_i) d\bm{\eta}\\
        &\lesssim \int_{B^{(\ell - 1)d}_1} \int_{\mathbb{S}^{d-1}_{\rho_i(\bm{\eta})}} \ind_{S(\alpha,\nu)}(\eta_i) d\eta_i d\eta_1 \cdots d\eta_{i-1} d\eta_{i+1} \cdots d\eta_\ell \\
        &\lesssim \arccos{\alpha}. \label{ellipsoid S est res}
    \end{align}
    where
    \begin{align}
        \rho_i(\bm{\eta}) = \sqrt{1- |\eta_1|^2 - \cdots -|\eta_{i-1}|^2 - |\eta_{i+1}|^2 - \cdots |\eta_{\ell}|^2 },
    \end{align}
    and \eqref{ellipsoid S est res} follows from a direct application of Lemma \ref{conic est sphere}.
    
\end{proof}

\begin{lemma}\label{conic est N}
    Consider $0\leq \alpha \leq 1$, $\nu \in \R^d \setminus \{0\}$, and $\ell \in \N$. Define,
    \begin{equation}
        N(\alpha, \nu) = \{ (\omega_1,\omega_2) \in \R^{2d} \, :\, \langle \omega_1 - \omega_2, \nu\rangle \geq \alpha |\omega_1- \omega_2||\nu|  \},
    \end{equation}
    and let $\mathbb{E}^{\ell d -1}$ denote an $(\ell d - 1)$-dimensional ellipsoid. Then, we have the estimate:
    \begin{equation}
        \int_{\mathbb{E}^{\ell d -1} } \ind_{N(\alpha,\nu)}(\omega_1,\omega_2) d\bm{\omega} \lesssim \arccos{\alpha}.
    \end{equation}
\end{lemma}
\begin{proof}
    Following the same steps as the proof of Lemma \ref{conic est N sphere} given in \cite{AmpatzoglouPavlovic2020} we first note that $N(\alpha, \nu)$ has the representation 
    \begin{equation}
        N(\alpha,\nu) = \{ (\omega_1,\omega_2)\in \R^{2d} \, : \, \omega_1 - \omega_2 \in S(\alpha, \nu))\}.
    \end{equation}
    We define the linear map $T:\R^{\ell d} \rightarrow \R^{\ell d}$ by
    \begin{equation}
        \bm{u} = T(\bm{\omega}) := (\omega_1 + \omega_2, \omega_1 - \omega_2, \omega_3, \cdots, \omega_\ell).
    \end{equation}
    $T$ is a linear bijection and an $\ell$-block matrix and thus by Corollary \ref{ellipsoid to sphere and trans block} sends the $\ell$-block ellipsoid $\mathbb{E}^{\ell d -1}$ to another $\ell$-block ellipsoid which we denote $\mathbb{\tilde{E}}^{\ell d -1}$. Therefore, by changing variables under $T$ we obtain
    \begin{align}
        \int_{\mathbb{E}^{\ell d -1}} \ind_{N(\alpha,\nu)}(\omega_1, \omega_2) d\bm{\omega} &= \int_{\mathbb{{E}}^{\ell d -1}} \ind_{S(\alpha,\nu)}(\omega_1 - \omega_2) d\bm{\omega} \\
        &\lesssim \int_{\mathbb{\tilde{E}}^{\ell d -1}} \ind_{S(\alpha,\nu)}(u_2) d\bm{u}\\
        &\lesssim \arccos{\alpha}, \label{conic N proof res}
    \end{align}
    where \eqref{conic N proof res} follows from a direct application of Lemma \ref{conic est}.
\end{proof}
\subsubsection{Intersection of the ellipsoid with the ball and cylinder}
We now move our attention to ellipsoidal estimates over the ball and cylinder.
For some $\rho > 0$, we define $K^d_\rho$ to be a $d$-dimensional cylinder centered at the origin.
We first recall a result given in \cite{Denlinger}. 
\begin{lemma}
    Let $\rho, r > 0 $ then the following estimates hold, 
    \begin{align}
        |{\mathbb{S}_r^{d-1}} \cap K_\rho^d|_{\mathbb{S}_r^{d-1}} \lesssim r^{d-1} \min \bigg\{1, \big(\frac{\rho}{r}\big)^{\frac{d-1}{2}} \bigg\}
    \end{align}
\end{lemma}
We adapt this result to prove a similar estimate on the $(\ell d -1)$-dimensional sphere for $\ell \geq 2$.
\begin{lemma}\label{sphere cylinder est}
    Let $\rho,r > 0$ and $\ell \geq 2$, the following estimates hold for the $(\ell d -1)$-spherical measure,
    \begin{align}
        |{\mathbb{S}_1^{\ell d-1}} \cap \big( K_\rho^d \times \R^{(\ell-1)d} \big) |_{\mathbb{S}_1^{\ell d-1}} &\lesssim  \min \big\{1, \rho^{\frac{d-1}{2}} \big\}. \\
    \end{align}
\end{lemma}
\begin{proof}
    We start by using the following representation of $\mathbb{S}_1^{\ell d -1}$,
    \begin{equation}
        \mathbb{S}_1^{\ell d -1} = \big\{ (\omega_1,\cdots, \omega_\ell) \in \R^d \times B_1^{(\ell-1)d} \, : \, \omega_1 \in \mathbb{S}^{d-1}_{\sqrt{1-|\omega_2|^2 - \cdots - |\omega_\ell|^2}} \big\}.
    \end{equation}
    
\begin{align}
    |{\mathbb{S}_1^{\ell d-1}} &\cap \big( K_\rho^d \times \R^{(\ell-1)d} \big) |_{\mathbb{S}_1^{d-1}} = \int_{B^{(\ell -1)d}_1} \bigg|\mathbb{S}^{d-1}_{\sqrt{1-|\omega_2|^2 - \cdots - |\omega_\ell|^2}}  \cap K_\rho^d\bigg|_{\mathbb{S}^{d-1}_{\sqrt{1-|\omega_2|^2 - \cdots - |\omega_\ell|^2}}} d\omega_2 \cdots d\omega_\ell \\
    &\lesssim \int_{B^{(\ell-1)d}} (1-|\omega_2|^2 - \cdots - |\omega_\ell|^2)^{\frac{d-1}{2}} \min \bigg \{1, \big( \frac{\rho}{\sqrt{1- |\omega^2|^2 - \cdots |\omega_\ell|^2}}\big)^{\frac{d-1}{2}}   \bigg \} d\omega_2 \cdots d\omega_\ell \\
    &\lesssim \int_0^1 s^{\ell d -d -1} (1-s^2 )^\frac{d-1}{2}  \min \bigg \{1, \big( \frac{\rho}{\sqrt{1- s^2}}\big)^{\frac{d-1}{2}}   \bigg \} d s.
\end{align}
We denote $I(\rho) = \int_0^1 s^{\ell d -d -1} (1-s^2 )^\frac{d-1}{2}  \min \bigg \{1, \big( \frac{\rho}{\sqrt{1- s^2}}\big)^{\frac{d-1}{2}}   \bigg \} d s$. For $\rho \geq 1$ we have, 
\begin{equation}
    I(\rho) \leq \int_0^1 s^{\ell d-d -1} (1-s^2 )^\frac{d-1}{2} \leq 1, 
\end{equation}
since $\ell,d \geq 2$.
For $0<\rho<1$, we have,
\begin{align}
    I(\rho) &= \int_0^{\sqrt{1-\rho^2}} s^{\ell d -d-1}(1-s^2 )^\frac{d-1}{2} \bigg( \frac{\rho}{\sqrt{1-s^2}} \bigg)^\frac{d-1}{2} ds + \int_{\sqrt{1-\rho^2}}^1 s^{\ell d-d -1} (1-s^2 )^\frac{d-1}{2} ds \\
    &\lesssim \rho^{\frac{d-1}{2}} \int_0^1 s^{\ell d -d-1} (1-s^2 )^{\frac{d-1}{4}} ds + \int_0^{\rho^2} (1-u)^{\frac{\ell d -d -1}{2}} u ^{\frac{d-1}{2}} du \\
    &\lesssim \rho^{\frac{d-1}{2}} + \int_0^{\rho^2} u^{\frac{d-1}{2}} du \\
    &= \rho^{\frac{d-1}{2}} + \rho^{d+1} \\
    &\lesssim \rho^{\frac{d-1}{2}},
\end{align}
where for the second integral above we used the substitution $u = 1-s^2$, the third line holds because $\ell,d \geq 2$,  and the last inequality holds since $ d\geq 2$.

\end{proof}
\begin{lemma}[Ellipsoidal estimates for the ball and cylinder] \label{ellipsoidal est}
Let $\mathbb{E}^{\ell d -1} \subset \R^{\ell d}$ be an $\ell$-block ellipsoid, $K^d_\rho \subset \R^d$ be a cylinder centered at the origin with radius $\rho >0$, and $B^d_\rho \subset \R^d$ be the ball centered at the origin with radius $\rho>0$. Then the following estimates hold, 
    \begin{enumerate}
    \item $|\mathbb{E}^{\ell d -1} \cap \big( K_\rho^d \times \R^{(\ell-1)d}\big) | \lesssim \min \big\{1, {\rho}^{\frac{d-1}{2}} \big\}$
    \item $|\mathbb{E}^{\ell d -1} \cap \big( B_\rho^d \times \R^{(\ell-1)d}\big) | \lesssim \min \big\{1, {\rho}^{\frac{d-1}{2}} \big\}$.
\end{enumerate}
\end{lemma}
\begin{proof}
Fix $T_1: \mathbb{E}^{\ell d -1} \rightarrow \mathbb{S}_1^{\ell d -1}$ as in \eqref{ellipsoid to sphere block} such that it keeps the first component of $\bm{\omega}$ invariant up to a rescaling. Thus, we have $T_1(K_\frac{\rho}{r}^d \times \R^{(\ell-1)d}) = (K_\frac{c \rho}{r}^d \times \R^{(\ell -1) d})$ for some constant $c>0$. Thus, by \eqref{ellipsoid to sphere block} and Lemma \ref{sphere cylinder est} we have
\begin{align}
    |\mathbb{E}^{\ell d -1} \cap \big( K_\rho^d \times \R^{(\ell-1)d}\big) |  &= \int_{\mathbb{E}^{\ell d -1}} \mathds{1}_{K_\frac{\rho}{r}^d \times \R^{(\ell-1)d}} (\bm{\omega}) d\bm{\omega} \\
    &\lesssim \int_{\mathbb{S}_1^{\ell d -1}} \mathds{1}_{K_\frac{c \rho}{r}^d \times \R^{(\ell-1)d}} (\bm{\theta}) d\bm{\theta} \\
    &\lesssim \min \big\{1, \big(\frac{\rho}{r}\big)^{\frac{d-1}{2}} \big\}.
\end{align}
An identical calculation can be used to derive the result for the $d$-dimensional ball by noting that $B^d_\rho \subset K_\rho^d$. Notice that by a simple permutation we can choose to leave any variable invariant. Thus, this result holds the same for $K^d_\rho$ occurring in any of the $\ell$ coordinates. 
\end{proof}

For $\mu, \lambda \neq 0$ we define the $2d$-dimensional strip $W^{2d}_{\rho,\mu,\lambda}$ as,
\begin{equation}\label{W strip}
    W^{2d}_{\rho,\mu,\lambda} = \{ (\omega_1, \omega_2) \in \R^{2d} \, : \, |\mu \omega_1 - \lambda \omega_2 | \leq \rho \}.
\end{equation}
We can derive a similar estimate for $W^{2d}_{\rho,\mu,\lambda}$.
\begin{lemma}\label{W strip est}
    For $r,\rho >0$ and an $(\ell d - 1)$-dimensional $\ell$-block ellipsoid $\mathbb{E}^{\ell d -1}$, we have the estimate
    \begin{equation}
        |\mathbb{E}^{\ell d -1} \cap (W^{2d}_{\rho,\mu,\lambda} \times \R^{(\ell-2)d}) | \lesssim \min \bigg\{1, \bigg( \frac{\rho}{|\mu|}\bigg)^{\frac{d-1}{2}}, \bigg( \frac{\rho}{|\lambda|}\bigg)^{\frac{d-1}{2}}  \bigg\}.
    \end{equation}
\end{lemma}
\begin{proof}
    Notice that we have,
    \begin{align}
        W^{2d}_{\rho,\mu,\lambda} &= \{(\omega_1,\omega_2) \in \R^{2d} \, : \, \omega_1 \in B^d_{\frac{\rho}{|\mu|}}(\lambda \mu^{-1} \omega_2)  \} \subset \{(\omega_1,\omega_2) \in \R^{2d} \, : \, \omega_1 \in K^d_{\frac{\rho}{|\mu|}}(\lambda \mu^{-1} \omega_2)  \} \\
        W^{2d}_{\rho,\mu,\lambda} &= \{(\omega_1,\omega_2) \in \R^{2d} \, : \, \omega_1 \in B^d_{\frac{\rho}{|\lambda|}}(\mu \lambda^{-1} \omega_2)  \} \subset \{(\omega_1,\omega_2) \in \R^{2d} \, : \, \omega_1 \in K^d_{\frac{\rho}{|\lambda|}}(\mu \lambda^{-1} \omega_2)  \},
    \end{align}
where we can apply Lemma \ref{ellipsoidal est} to finish the proof. 
    
\end{proof}
\subsubsection{Intersection of the ellipsoid with the annulus}
We now will present two annular estimates which are critical for estimating the measures of pathological sets defined in Section \ref{sec:stability}).

\begin{lemma}\label{annuli est}
    Let $\mathbb{E}^{\ell d-1}$ be an $\ell$-block ellipsoid for $\ell \geq 1$ and $d \geq 2$,
    \begin{equation}
        \mathbb{E}^{\ell d-1} = \{ \bm{x} = (x_1,\cdots, x_\ell) \in \R^{\ell d} \, : \, \langle \bm{x}, B \bm{x} \rangle = c \},
    \end{equation}
    where $c > 0$ and $B \in \R^{\ell d \times \ell d}$ is a symmetric, positive definite $\ell$-block matrix. We define the $d$-dimensional annulus for a general quadratic form as
\begin{equation}
\mathbb{A}^{\ell d}_{\beta,\rho} = \{ \bm{x} = (x_1,\cdots, x_\ell) \in \R^{\ell d} \, : \, \rho - \beta \leq \langle \bm{x}, A \bm{x} \rangle \leq \rho + \beta \},
\end{equation}
where $A: \R^{\ell d \times \ell d}$ is a symmetric non-zero $\ell$-block matrix,  $0 < \beta < 1$, and $\rho > \beta$. Furthermore, we assume that $A$ is not a scalar multiple of $B$, i.e. $A \neq \alpha B$ for any $\alpha \in \R$. The following estimate holds,
\begin{equation} 
    \int_{\mathbb{E}^{\ell d-1}} \ind_{\mathbb{A}^{\ell d}_{\beta,\rho}}(\bm{x}) d\bm{x} \lesssim \beta.
\end{equation}
\end{lemma}
\begin{proof}
    We begin by performing a change of variables to transform the ellipsoid $\mathbb{E}^{\ell d - 1}$ to the sphere $\mathbb{S}_1^{\ell d-1}$.
    $B$ positive definite implies that all of its eigenvalues are real and positive. Therefore, there exists an orthogonal $\ell$-block matrix $\mathcal{O} \in \R^{\ell d\times \ell d}$ such that $\mathcal{O}^T B \mathcal{O} = D$, where $D \in \R^{\ell d\times \ell d}$ is a positive definite diagonal matrix. Making the substitution
    \begin{align}
        \bm{x} = S \bm{y}, \quad S= \sqrt{c} D \mathcal{O},
    \end{align}
    we obtain, 
    \begin{equation}
        S^{-1}(\mathbb{E}^{\ell d-1}) = \mathbb{S}^{\ell d-1}_1,
    \end{equation}
    and our estimate becomes, 
    \begin{equation}
         \int_{\mathbb{E}^{\ell d-1}} \ind_{\mathbb{A}^{\ell d}_{\beta,\rho}}(\bm{x}) d\bm{x} \lesssim \int_{\mathbb{S}_1^{\ell d-1}} \ind_{\tilde{\mathbb{A}}^{\ell d}_{\beta,\rho}}(\bm{y}) d\bm{y},
    \end{equation}
    where
    \begin{align}
        \tilde{\mathbb{A}}^{\ell d}_{\beta,\rho} = \{ \bm{x} \in \R^{\ell d} \, : \, \rho - \beta \leq \langle \bm{x}, \tilde{A} \bm{x} \rangle \leq \rho + \beta \}, \quad \tilde{A} = (S^{-1})^T A S^{-1}.
    \end{align}
$\tilde{A}$ is symmetric and nonzero $\ell$-block matrix, so $\tilde{\mathbb{A}}^d_{\beta,\rho}$ remains the $\beta$-annulus of a non-trivial $\ell$-block quadratic form. We also know that since $A \neq \alpha B$ for any $\alpha \in \R$,  $\tilde{\mathbb{A}}^d_{\beta,\rho}$ is not transformed to the sphere. Furthermore, there exists an $\ell$-block orthogonal matrix $\tilde{\mathcal{O}} \in \R^{\ell d \times \ell d}$ such that $\tilde{\mathcal{O}}^T \tilde{A} \tilde{\mathcal{O}} = \tilde{D} $ where $\tilde{D}$ is a nonzero diagonal matrix. Performing another change of variables, noting that $\tilde{\mathcal{O}}$ leaves the sphere invariant, we obtain,
\begin{equation}
    \int_{\mathbb{S}_1^{\ell d-1}} \ind_{\tilde{\mathbb{A}}^{\ell d}_{\beta,\rho}}(\bm{y}) d\bm{y} = \int_{\mathbb{S}_1^{\ell d-1}} \ind_{\tilde{\mathbb{D}}^{\ell d}_{\beta,\rho}}(\bm{z}) d\bm{z}, \quad \tilde{\mathbb{D}}^{\ell d}_{\beta,\rho} = \{ \bm{x} = (x_1,\cdots, x_\ell) \in \R^{\ell d} \, : \, \rho - \beta \leq \langle \bm{x},  \tilde{D}  \bm{x} \rangle \leq \rho + \beta \},
\end{equation}
hence, it suffices to prove,
\begin{equation}\label{D est}
    \int_{\mathbb{S}_1^{\ell d-1}} \ind_{\tilde{\mathbb{D}}^{\ell d}_{\beta,\rho}}(\bm{x}) d\bm{x} \lesssim \beta.
\end{equation}
Let $\lambda_1,\cdots, \lambda_\ell$ be the diagonal elements of the $\ell$-block diagonal matrix $\tilde{D}$. By a permutation of variables we reorder the $(\lambda_i)_{1\leq i \leq \ell}$ into descending order so that,
\begin{equation}
    \lambda_1 \geq \cdots \geq \lambda_\ell.
\end{equation}
Furthermore, we denote the number of times the maximum element $\lambda = \lambda_1 $ appears in $\{\lambda_i \}_{i=1}^\ell$ as $m$, so that,
\begin{equation}\label{lambda order}
    \lambda = \lambda_1 = \cdots = \lambda_m > \lambda_{m+1} \geq \cdots \geq \lambda_\ell, \quad m  \in \{1,\cdots, \ell-1 \}.
\end{equation}
Note, that since $\tilde{\mathbb{D}}^{\ell d}_{\beta,\rho}$ is not a sphere we have that $m \leq \ell-1$. 
Therefore, we have 
\begin{equation}
    \tilde{\mathbb{D}}^{\ell d}_{\beta,\rho} = \{ \bm{x}\in \R^{\ell d} \, : \, \rho - \beta \leq \lambda |x_1|^2 + \cdots + \lambda |x_m|^2 + \lambda_{m+1} |x_{m+1}|^2  + \cdots  +\lambda_k |x_k|^2 \leq \rho + \beta \}.
\end{equation} 
By the the co-area formula we obtain the following estimate,
\begin{align}
    \int_{\mathbb{S}^{\ell d-1}_1} \ind_{\tilde{\mathbb{D}}^{\ell d}_{\beta,\rho}}(\bm{x}) d\bm{x} &= \int_{B_1^{(\ell -1)d}} \int_{\mathbb{S}^{d-1}_{\sqrt{1-|x_1|^2 -\cdots |x_{\ell-1}|^2}}} \ind_{\tilde{\mathbb{D}}^{(\ell-1) d}_{\beta,\rho}}(x_1,\cdots,x_{\ell-1}) dx_\ell \cdots dx_1, \\
    &\lesssim \int_{B_1^{(\ell -1)d}} \ind_{\tilde{\mathbb{D}}^{(\ell-1) d}_{\beta,\rho}}(x_1,\cdots,x_{\ell-1}) dx_{\ell-1} \cdots dx_1, \\
    &= \int_{0 \leq r_1 \leq  1} \int_{S_{r_1}^{(\ell-1)d - 1}} \ind_{\tilde{\mathbb{D}}^{(\ell-1) d}_{\beta,\rho}}(x_1,\cdots,x_{\ell-1}) dx_{\ell-1} \cdots dx_1 dr_1,
\end{align}
where
\begin{equation}
    \tilde{\mathbb{D}}^{(\ell-1) d}_{\beta,\rho} = \{ (x_1,\cdots, x_{\ell-1}) \in \R^{(\ell -1)d} \, : \, \rho - \lambda_\ell - \beta \leq (\lambda_1 - \lambda_\ell) |x_1|^2 + \cdots + (\lambda_{\ell-1} - \lambda_\ell)|x_{\ell-1}|^2 \leq \rho - \lambda_\ell + \beta \}.
\end{equation}
By applying this procedure $\ell - m$ times and recalling \eqref{lambda order} we obtain,
\begin{align}\label{iterated int}
    \int_{\mathbb{S}^{\ell d-1}_1} \ind_{\tilde{\mathbb{D}}^{\ell d}_{\beta,\rho}}(\bm{x}) d\bm{x} &\lesssim \int_{0 \leq r_1 \leq 1}\int_{0\leq r_2 \leq r_1} \cdots \int_{0\leq r_{\ell - m-1}\leq r_{\ell - m - 2}} \int_{B^{md}_{r_{\ell-m-1}}} \ind_{\tilde{\mathbb{D}}^{m d}_{\beta,\rho}}(x_1,\cdots,x_{m}) dx_{m} \cdots dx_1 dr_{\ell- m-1} \cdots dr_1,
\end{align}
where
\begin{equation}
    \tilde{\mathbb{D}}^{m d}_{\beta,\rho} = \{ (x_1,\cdots, x_{m}) \in \R^{m d} \, : \, \rho(\bm{r}, \bm{\lambda}) - \beta \leq (\lambda - \lambda_{m+1}) |x_1|^2 + \cdots + (\lambda - \lambda_{m+1})|x_{m}|^2 \leq \rho(\bm{r}, \bm{\lambda}) + \beta \}.
\end{equation}
\begin{equation}
    \rho(\bm{r}, \bm{\lambda}) = \rho - \lambda_\ell (1+ r_1^2) - \cdots - \lambda_{m+2} (r_{\ell - m - 2}^2 + r_{\ell - m - 1}^2) - \lambda_{m+1} r_{\ell-m-1}^2,
\end{equation}
where we denote
\begin{equation}
    \bm{r} = (r_1, \cdots r_{\ell -m-1}), \quad \bm{\lambda} = (\lambda_{m+1}, \cdots, \lambda_{\ell}).
\end{equation}
Since $\lambda_{m+1},\cdots, \lambda_{\ell}$ may take on negative values we can bound $\rho(\bm{r}, \bm{\lambda}) \leq \tilde{\rho}$ where
\begin{equation}
    \tilde{\rho} := \rho + 2 |\lambda_\ell| + \cdots + 2|\lambda_{m+2}| + |\lambda_{m+1}| > 0.
\end{equation}
Estimating the inner integral of \eqref{iterated int}, we have for $\tilde{\rho} > \beta$
\begin{align}
    \int_{B^{md}_{r_{\ell-m-1}}} \ind_{\tilde{\mathbb{D}}^{m d}_{\beta,\rho}}(x_1,\cdots,x_{m}) dx_{1} \cdots dx_m &\leq \int_{\tilde{\mathbb{D}}^{m d}_{\beta,\rho}} dx_1 \cdots dx_m, \\
    &\leq \int_{ \frac{\tilde{\rho} - \beta}{\lambda - \lambda_{m+1}} \leq |x_1|^2 + \cdots + |x_{m}|^2 \leq  \frac{\tilde{\rho} + \beta}{\lambda - \lambda_{m+1}}} dx_1 \cdots dx_m,\\
    &\lesssim (\tilde{\rho} + \beta)^{md/2} - (\tilde{\rho} - \beta)^{md/2} \\
    &\lesssim \beta,
\end{align}
and since $d\geq 2$ and $m\geq 1$, for $\tilde{\rho} \leq \beta$,
\begin{align}
    \int_{B^{md}_{r_{\ell-m-1}}} \ind_{\tilde{\mathbb{D}}^{m d}_{\beta,\rho}}(x_1,\cdots,x_{m}) dx_{1} \cdots dx_m &\leq \int_{\tilde{\mathbb{D}}^{m d}_{\beta,\rho}} dx_1 \cdots dx_m, \\
    &\leq \int_{  |x_1|^2 + \cdots + |x_{m}|^2 \leq  \frac{2 \beta}{\lambda - \lambda_{m+1}}} dx_1 \cdots dx_m,\\
    &\leq \bigg( \frac{2\beta}{\lambda - \lambda_{m+1}} \bigg)^{md/2}, \\
    &\lesssim \beta,
\end{align}
where both estimates are uniform in $r_1, \cdots, r_{\ell - m - 1}$. Therefore, by estimating $r_1, \cdots, r_{\ell - m -1 } \leq 1$ in \eqref{iterated int}, we obtain the result \eqref{D est}.
\end{proof}

\subsection{Transition map}
Now we will define the transition map which will be used for controlling post-collisional configurations. For $v_1, \cdots, v_{\ell+1} \in \R^d$, we define the sets
\begin{equation}
    \Omega = \{ (\omega_1, \cdots, \omega_\ell) \in \R^{\ell d} \, : \, \sum_{i = 1}^\ell |\omega_i|^2 + \sum_{1\leq i < j \leq \ell} |\omega_i - \omega_j|^2 < \frac{3}{2} \text{ and } b_\ell(\bm{\omega}, v_2 - v_1, \cdots, v_{\ell+1} - v_1) > 0   \},
\end{equation}
and
\begin{equation}\label{mathcal E}
    \mathcal{E}^+_{v_1,\cdots, v_{\ell+1}} = \E \cap \Omega.
\end{equation}
We also define the smooth map $\Psi(\nu_1,\cdots, \nu_\ell) = \sum_{i=1}^\ell |\nu_i|^2 + \sum_{1\leq i < j \leq \ell}|\nu_i - \nu_j|^2$ so that we have
\begin{equation}
    \E = [\Psi = 1].
\end{equation}

\begin{lemma}\label{impact dir change of variables}
    Fix $v_1, \cdots, v_{\ell+1} \in \R^d $ and $r>0$ such that,
    \begin{equation}
        r^2 = \sum_{1\leq i < j \leq \ell+1} |v_i - v_j|^2.
    \end{equation}
We denote 
\begin{equation}
    \bm{v}' = 
    \begin{pmatrix}
        v_1 - v_2 \\
        \vdots \\
        v_1- v_{\ell +1}
    \end{pmatrix}
\end{equation}
and define the transition map  $\mathcal{T}_{v_1,\cdots v_{\ell+1}}: \Omega \rightarrow \R^{\ell d}\setminus \{ r^{-1} \bm{v'} \}$ by
\begin{equation}
    \bm{\nu} = \mathcal{T}_{v_1,\cdots v_{\ell+1}}(\bm{\omega}) = \frac{1}{r} 
    \begin{pmatrix}
        v_1^* - v_2^* \\
        \vdots \\
        v_1^* - v_{\ell+1}^*
    \end{pmatrix}
\end{equation}.
    The transition map satisfies the following properties,
\begin{enumerate}
    \item $\mathcal{T}_{v_1,\cdots, v_{\ell+1}}$ is smooth in $\Omega$ with bounded derivative uniformly in $r$,
    \begin{equation}\label{T derivative est}
        ||D \mathcal{T}_{v_1,\cdots, v_{\ell+1}}||_{\infty} \leq C_{d,\ell}, \quad \forall \bm{\omega} \in \Omega.
    \end{equation}
    \item The Jacobian of $\mathcal{T}_{v_1,\cdots v_{\ell+1}}$ is given by 
    \begin{equation}\label{T Jac}
        \jac(\mathcal{T}_{v_1,\cdots v_{\ell+1}}) \simeq r^{-\ell d} c^{\ell d}(\bm{\omega}, v_1, \cdots, v_{\ell+1}), \quad \forall \bm{\omega} \in \Omega.
    \end{equation}

    \item  $ \mathcal{T}_{v_1,\cdots v_{\ell+1}}: \mathcal{E}^+_{v_1,\cdots, v_{\ell+1}} \rightarrow \E\setminus \{r^{-1} \bm{v}' \}$ is bijective and we have
    \begin{equation}
        \mathcal{E}^+_{v_1,\cdots, v_{\ell+1}} = [\Psi \circ \mathcal{T}_{v_1,\cdots, v_{\ell+1}} = 1].
    \end{equation}
    \item For any measureable $g: \R^{\ell d } \rightarrow [0,+\infty]$, there holds
    \begin{equation} \label{T g est}
        \int_{ \mathcal{E}^+_{v_1,\cdots, v_{\ell+1}}} (g \circ\mathcal{T}_{v_1,\cdots, v_{\ell+1}} )(\bm{\omega}) |\jac{\mathcal{T}_{v_1,\cdots, v_{\ell+1}}}| d\bm{\omega} \lesssim \int_{\E} g(\bm{\nu}) d\bm{\nu}.
    \end{equation}
    
\end{enumerate}

\end{lemma}

\begin{proof}
We can write the transformation $\mathcal{T}_{v_1,\cdots, v_{\ell+1}}$ as
\begin{equation}\label{T eqn}
    \mathcal{T}_{v_1,\cdots, v_{\ell+1}}(\bm{\omega}) = r^{-1} \big(\bm{v}' + (\ell + 1)c(\bm{\omega}, \bm{v}') \bm{\omega}\big).
\end{equation}
\par
\textit{(1)}:
We now will use the form given in \eqref{T eqn} to calculate the derivative of $\mathcal{T}_{v_1,\cdots, v_{\ell+1}}$
\begin{equation}\label{T derivative}
    D\mathcal{T}_{v_1,\cdots, v_{\ell + 1}} (\bm{\omega}) = r^{-1} (\ell+1)\big(  c(\bm{\omega}, \bm{v}') I_{\ell d} +  \bm{\omega} \nabla^T_{\bm{\omega}} c(\bm{\omega}, \bm{v}')   \big). 
\end{equation}
By a simple calculation we find
\begin{equation} \label{nabla omega}
    \nabla_{\bm{\omega}} c(\bm{{\omega}}, \bm{v}) =
    \begin{pmatrix}
        (v_2 - v_1) + \cdots + (v_2 - v_{\ell+1}) \\
        \vdots \\
        (v_{\ell +1} - v_1) +  \cdots + (v_{\ell +1} - v_{\ell}).
    \end{pmatrix}
\end{equation}
Therefore, we have
\begin{align}\label{nabla r est}
    |\nabla_{\bm{\omega}} c(\bm{{\omega}}, \bm{v})| & \leq C_{d,\ell} r.
\end{align}
Similarly, we can derive the estimate
\begin{align}\label{c r est}
    |c(\bm{{\omega}}, \bm{v})| &= | \sum_{i=1}^\ell \langle \omega_i,v_{i+1} - v_1 \rangle + \sum_{1\leq i < j \leq \ell} \langle \omega_i - \omega_j, v_{i + 1} - v_{j+1} \rangle  | \\
    & \leq C_{d,\ell} \sum_{1\leq i < j \leq \ell} |v_{i} - v_{j}| \\
    &\leq C_{d,\ell} r.
\end{align}
Combining \eqref{nabla r est} and \eqref{c r est} into \eqref{T derivative} we obtain the desired estimate given \eqref{T derivative est}.
\par
\textit{(2)}:
By applying \eqref{T derivative} to the identity, 
\begin{equation}
    \det (\lambda I_n + w u^T) = \lambda^n (1+ \lambda^{-1} \langle w, u \rangle ),
\end{equation}
given in Lemma \ref{app det} of the appendix,
we can calculate the Jacobian, 
\begin{equation}
    \jac(\mathcal{T}_{v_1,\cdots, v_{\ell+1}}) = (\ell + 1)^{\ell d} r^{- \ell d} c^{\ell d}(\bm{{\omega}}, \bm{v}')  \big( 1+ c^{-1}(\bm{{\omega}}, \bm{v}') \langle \bm{\omega}, \nabla_{\bm{\omega}}c(\bm{{\omega}}, \bm{v}') \rangle \big).
\end{equation}
By recalling \eqref{nabla omega} we can easily show that 
\begin{equation}
   \langle \bm{\omega}, \nabla_{\bm{\omega}} c(\bm{{\omega}}, \bm{v}) \rangle = c(\bm{{\omega}}, \bm{v}),
 \end{equation}
 thus \ref{T Jac} follows.
\par
\textit{(3)}:
 We first will show that  $\mathcal{T}_{v_1,\cdots, v_{\ell + 1}}: \mathcal{E}^+_{v_1,\cdots, v_{\ell+1}} \rightarrow \E\setminus \{r^{-1} \bm{v}' \}$ is injective. Fix $\bm{\omega}\in \mathcal{E}^+_{v_1,\cdots, v_{\ell+1}}$ and denote $\bm{\nu} = \mathcal{T}_{v_1,\cdots, v_{\ell + 1}}(\bm{\omega})$
\begin{align}
    \nu_1^2 + \cdots + \nu_\ell^2 + \sum_{1\leq i < j \leq \ell } |\nu_i - \nu_j|^2 &= r^{-2} \bigg( \sum_{i = 1}^\ell |v^* - v^*_{i} |^2  + \sum_{1\leq i < j \leq \ell} |v^*_{i} - v^*_{j}|^2\bigg) \\
    &= r^{-2} \bigg( \sum_{i = 1}^\ell |v - v_{i} |^2  + \sum_{1\leq i < j \leq \ell} |v_{i} - v_{j}|^2\bigg) \\
    &= 1.
\end{align}
To prove injectivity, let $\bm{\omega},\bm{\omega}' \in \mathcal{E}^+_{v_1,\cdots, v_{\ell+1}}$ such that $\mathcal{T}_{v_1,\cdots v_{\ell+1}}(\bm{\omega}) = \mathcal{T}_{v_1,\cdots v_{\ell+1}}(\bm{\omega}')$. Therefore, $c(\bm{\omega}, \bm{v}) \bm{\omega} =c(\bm{\omega}', \bm{v}) \bm{\omega}'$ which implies that there exists a $k \in \R$ such that $\bm{\omega} = k \bm{\omega}'$. But, $\bm{\omega}\in \E$ implies that $k = 1$.
\par
To prove surjectivity, fix $\bm{\nu} \in \E\setminus \{r^{-1} \bm{v}' \}$. By the representation given by \eqref{T eqn}, we want to find an $\bm{\omega} \in \mathcal{E}^+_{v_1,\cdots, v_{\ell+1}} $ that satisfies the equation
\begin{equation}\label{omega surj}
    c(\bm{\omega}, \bm{v})\bm{\omega} = (\ell + 1)^{-1} (r \bm{\nu} - \bm{v}').
\end{equation}
Since $c(\bm{\omega}, \bm{v}') > 0$, we want to find an $\bm{\omega}$ that satisfies 
\begin{equation} \label{omega lambda}
    \bm{\omega}  = \lambda (r\bm{\nu} - \bm{v}'),
\end{equation}
for some $\lambda > 0$. Substituting \eqref{omega lambda} into \eqref{omega surj} we obtain
\begin{align}
    \lambda^2 c(r\bm{\nu} - \bm{v}', \bm{v}')(r\bm{\nu} - \bm{v}') = (\ell + 1)^{-1}(r\bm{\nu} - \bm{v}').
\end{align}
Thus, we have
\begin{equation}
    \lambda = \big((\ell+1)c(r\bm{\nu} - \bm{v}', \bm{v}')\big)^{-1/2} >0.
\end{equation}
\par
\textit{(4)}: 
By a simple calculation, for all $\bm{\nu}\in \R^{\ell d} $ we have the estimates
\begin{equation}
    4 \Psi(\bm{\nu}) \leq | \nabla \Psi (\bm{\nu})|^2 \leq 24 \Psi(\bm{\nu}).
\end{equation}
Therefore, $\nabla \Psi (\bm{\nu}) \neq 0$ for all $\bm{\nu} \in [\frac{1}{2} < \Psi < \frac{3}{2}]$. By Lemma \ref{indicatrix lemma} we have
\begin{align}
    &\int_{\mathcal{E}^+_{v_1,\cdots, v_{\ell+1}}} (g \circ \mathcal{T}_{v_1,\cdots, v_{\ell +1}})(\bm{\omega}) |\jac{\mathcal{T}_{v_1,\cdots, v_{\ell +1}}}| \frac{|\nabla \Psi(\mathcal{T}_{v_1,\cdots, v_{\ell +1}}(\bm{\omega}))|}{|\nabla (\Psi \circ \mathcal{T}_{v_1,\cdots, v_{\ell +1}})(\bm{\omega})|} d\bm{\omega} \\
    &= \int_{[\Psi \circ \mathcal{T}_{v_1,\cdots, v_{\ell +1}}=1]} (g \circ \mathcal{T}_{v_1,\cdots, v_{\ell +1}})(\bm{\omega}) |\jac{\mathcal{T}_{v_1,\cdots, v_{\ell +1}}}| \frac{|\nabla \Psi(\mathcal{T}_{v_1,\cdots, v_{\ell +1}}(\bm{\omega}))|}{|\nabla (\Psi \circ \mathcal{T}_{v_1,\cdots, v_{\ell +1}})(\bm{\omega})|} d\bm{\omega}\\
    &=\int_{[\Psi = 1]} g(\bm{\nu}) \mathcal{N}_{\mathcal{T}_{v_1,\cdots, v_{\ell +1}}}(\bm{\nu}, [\Psi \circ \mathcal{T}_{v_1,\cdots, v_{\ell +1}}=1]) d\bm{\nu} \\ 
    &= \int_{\E} g(\bm{\nu}) \mathcal{N}_{\mathcal{T}_{v_1,\cdots, v_{\ell +1}}}(\bm{\nu}, \mathcal{E}^+_{v_1,\cdots, v_{\ell+1}}) d\bm{\nu} \\
    &= \int_{\E} g(\bm{\nu)} d\bm{\nu}, \label{indicatrix calc last line}
\end{align}
where \eqref{indicatrix calc last line} follows from the fact that $\mathcal{T}_{v_1,\cdots v_{\ell+1}}: \mathcal{E}^+_{v_1,\cdots, v_{\ell+1}} \rightarrow \E\setminus \{r^{-1} \bm{v}' \}$ is a bijection. By the chain rule and \eqref{T derivative est} we have
\begin{equation}\label{chain rule psi T}
    \frac{|\nabla (\Psi \circ \mathcal{T}_{v_1,\cdots, v_{\ell +1}})(\bm{\omega})|}{|\nabla\Psi(\mathcal{T}_{v_1,\cdots, v_{\ell +1}}(\bm{\omega}))|} = \frac{|D^T \mathcal{T}_{v_1,\cdots, v_{\ell +1}}(\bm{\omega})\nabla \Psi (\mathcal{T}_{v_1,\cdots, v_{\ell +1}}(\bm{\omega}))|}{|\nabla\Psi(\mathcal{T}_{v_1,\cdots, v_{\ell +1}}(\bm{\omega}))|} = C_{d,\ell} ||D\mathcal{T}_{v_1,\cdots, v_{\ell +1}}(\bm{\omega})||_{\infty} \leq C_{d,\ell}.
\end{equation}
Therefore, by \eqref{chain rule psi T} and since $g\geq 0$, we obtain the desired result
\begin{equation}
    \int_{ \mathcal{E}^+_{v_1,\cdots, v_{\ell+1}}} (g \circ\mathcal{T}_{v_1,\cdots, v_{\ell+1}} )(\bm{\omega}) |\jac{\mathcal{T}_{v_1,\cdots, v_{\ell+1}}}| d\bm{\omega} \lesssim \int_{\E} g(\bm{\nu}) d\bm{\nu}.
\end{equation}
\end{proof}

\subsubsection{Geometric estimates on the transition map}
We now present estimates on the transition map when applied to the cylinder $K^d_\rho$. For the following lemma we will make use of the notation $\mathbb{K}_{\rho,i}^{\ell d}$ for $\ell \in \N$ defined by
\begin{equation}\label{K bb defn}
    \mathbb{K}_{\rho,i}^{\ell d} = \R^d \times \cdots \times K^d_\rho \times \cdots \times \R^d \subset \R^{\ell d} \quad 1\leq i \leq \ell ,
\end{equation}
where the cylinder $K^d_\rho \subset \R^d$ occurs in the $i^{\textit{th}}$ place.
\begin{lemma} \label{transformation of E}
    For $\bm{\nu} = (\nu_1, \cdots, \nu_\ell) = \mathcal{T}_{v_1,\cdots v_{\ell+1}}(\bm{\omega})$ the following holds, 
    \begin{align}
        v_1^* &\in K^d_\rho
        \iff 
        S_1 \bm{\nu} \in \mathbb{K}^{\ell d}_{\frac{(\ell+1)\rho}{r},1} \\
        v_{i+1}^* &\in K^d_\rho
        \iff 
        S_{i+1} \bm{\nu}
        \in \mathbb{K}^{\ell d}_{\frac{(\ell+1)\rho}{r},i} \quad 1 \leq i \leq \ell
    \end{align}
    where $S_1, \cdots, S_{\ell+1} \subset \R^{\ell d \times \ell d}$ are $\ell$-block invertible matrices given by,
    \begin{align}
        S_1(\bm{\nu}) = 
        \begin{pmatrix}
            \nu_1 + \cdots + \nu_{\ell} \\
            \nu_2 \\
            \vdots \\
            \nu_\ell
        \end{pmatrix},
        \quad
        S_{i+1}(\bm{\nu}) = 
        \begin{pmatrix}
            \nu_1 \\
            \vdots \\
            \tilde{\nu}_i \\
            \vdots \\
            \nu_\ell
        \end{pmatrix}
        \quad 1\leq i \leq \ell,
    \end{align}
where for all $1\leq i \leq \ell$, $\tilde{\nu}_i$ is given by
\begin{equation}
    \tilde{\nu}_i = -\ell \nu_i + \sum_{\substack{1\leq j \leq \ell \\ j \neq i}} \nu_{j}.
\end{equation}
\end{lemma}

\begin{proof}
    We start by substituting $\nu_i$ for $c(\bm{\omega},\bm{v}) \omega_i$ for $i = 1 ,\cdots \ell$, 
    \begin{align}
        \nu_i &= r^{-1}(v_1^* - v_{i+1}^* ) \\
        & = r^{-1}( v_1 - v_{i+1} + (\ell + 1) c(\bm{\omega},\bm{v}) \omega_i).
    \end{align}
    Therefore, we have
    \begin{equation}\label{c omega}
        c(\bm{\omega},\bm{v}) \omega_i = \frac{r\nu_i - v_1 + v_{i+1}}{\ell+1} \quad \text{ for } i = 1, \cdots \ell.
    \end{equation}
    Thus, substituting \eqref{c omega} back into the post-collisional velocities, we find
    \begin{align}
        v_1^* &= v_1 + c(\bm{\omega},\bm{v}) \sum_{i=1}^\ell \omega_i \\
        &= (\ell+1)^{-1} \sum_{i=1}^{\ell +1} v_{i} +  \frac{r}{\ell +1} \sum_{i=1}^\ell \nu_i.
    \end{align}
For $i = 1, \cdots, \ell$, in a similar way we obtain
\begin{align}
    v_{i+1}^* &= v_{i+1} + c(\bm{\omega}, \bm{v})(\ell \omega_i + \sum_{j \neq i} \omega_j)  \\
    &= (\ell+1)^{-1} \sum_{j=1}^{\ell+1} v_{j} + \frac{r}{\ell+1}( -\ell \nu_i + \sum_{j\neq i} \nu_j  ).
\end{align}
After translation and dilation we obtain the desired result. Furthermore, by elementary row operations it is easy to verify $S_i$ is invertible for all $1\leq i \leq \ell + 1$.

\end{proof}

\begin{lemma}\label{S est}
    For $\rho, r >0$ and $S_i$ defined as above, the following estimate holds for all $1\leq i \leq \ell + 1$
\begin{align}
    \int_{S_i(\E)} \mathds{1}_{\mathbb{K}^{\ell d}_{\frac{(\ell+1)\rho}{r},i}}(\theta_1,\cdots, \theta_\ell) d\bm{\theta} &\lesssim \min \big\{1, \big(\frac{\rho}{r}\big)^{\frac{d-1}{2}} \big\}.
\end{align}    
\end{lemma}
\begin{proof}
    This estimate is a direct consequence of Corollary \ref{ellipsoid to sphere and trans block}, Lemma \ref{ellipsoidal est}, and Lemma \ref{transformation of E}.
\end{proof}

 \section{Good configurations and stability}\label{sec:stability}

In this section, we investigate stability of good configurations under adjunctions of collisional particles. 
We define the set of well-separated configurations 
\begin{equation}\label{well sep conf} 
\Delta_m(\theta) =\{\widetilde{Z}_m=(\widetilde{X}_m,\widetilde{V}_m)\in\mathbb{R}^{2dm}: |\widetilde{x}_i-\widetilde{x}_j|>\theta,\quad\forall 1\leq i<j\leq m\},\quad m\geq 2,\quad \Delta_1(\theta)=\mathbb{R}^{2d}.
\end{equation}
A good configuration is a configuration which remains well-separated under backwards time evolution. More precisely, given $\theta>0$, $t_0>0$, we define the set of good configurations as:
\begin{equation}\label{good conf def}
G_m(\theta,t_0)=\left\{Z_m=(X_m,V_m)\in\mathbb{R}^{2dm}:Z_m(t)\in\Delta_m(\theta),\quad\forall t > t_0\right\},
\end{equation}
where $Z_m(t)$ denotes the backwards in time free flow of $Z_m=(X_m,V_m)$, given by:
\begin{equation}\label{back-wards flow}
Z_m(t)=\left((X_m\left(t\right),V_m\left(t\right)\right):=(X_m-tV_m,V_m),\quad t\geq 0.
\end{equation}
Since $Z_m$ is the initial point of the trajectory $Z_m(t)$, for $m\geq 2$, we can rewrite $G_m(\theta,t_0)$ as,
\begin{equation}\label{good conf def m>=2}
\begin{aligned}
G_m(\theta,t_0)=\left\{Z_m=(X_m,V_m)\in\mathbb{R}^{2dm}:|x_i(t)-x_j(t)|>\theta,\quad\forall t > t_0,\quad\forall i<j\in \left\{1,...,m\right\}\right\}.\end{aligned}
\end{equation}

From  now on, we consider parameters 
$R>>1$ and $0< \delta,\eta,\epsilon_0,\alpha<<1$ satisfying:
\begin{equation}\label{choice of parameters}
 \alpha<<\epsilon_0<<\eta\delta,\quad R\alpha<<\eta\epsilon_0.
\end{equation}

\subsection{Stability under $\ell$-nary adjunction}\label{subsec:ternary}
Given $v\in\mathbb{R}^d$, we denote
\begin{equation}
\left(\E \times B_R^{\ell d}\right)^+(v)=\big\{(\omega_1, \cdots, \omega_\ell,v_{1},\cdots, v_{\ell})\in\E\times B_R^{\ell d}:b_{\ell+1}(\omega_1,\cdots,\omega_\ell,v_{1}-v,\cdots,v_{\ell}-v)>0\big\},
\end{equation}
where we recall $b_{\ell+1}$ from \ref{impact param}.

\par
In the construction of pathological sets we will make use of the following lemma from \cite{Gallagher} presented in the same way as in Section 9 of \cite{ternary}.
 
 \begin{lemma}\label{cylinder lemma}
 Let $\epsilon << \alpha$, $\bar{y}_1,\bar{y}_2 \in \R^d$ such that $|\bar{y}_1 - \bar{y}_2| > \epsilon_0$, and $v_1 \in B_R^d$. There exists a $d$-cylinder $K_\eta^d \subseteq \R^d$ such that for all $Z_2 = (y_1,y_2,v_1,v_2)$ where $y_1\in B_\alpha^d(\bar{y}_1)$, $y_2 \in B_\alpha^d(\bar{y}_2)$, $v_2 \in B_R^d \setminus K_\eta^d$ we have $Z_2 \in G_2(\epsilon,0) \cap G_2(\epsilon_0, \delta)$.
 \end{lemma}

We are now ready to state the announced stability result.
 
 \begin{proposition}\label{bad set triple} Let $1\leq \ell \leq M$ and consider parameters $\alpha,\epsilon_0,R,\eta,\delta$ as in \eqref{choice of parameters} and fix interaction zone vector $\bm{\epsilon_{\ell}} = (\epsilon_2, \cdots, \epsilon_{\ell+1})$ such that $0 < \epsilon_2 < \cdots < \epsilon_{\ell+1} << 1$ and $\epsilon_{\ell } << \eta^2 \epsilon_{\ell+1} << \alpha$. Let $m\in\mathbb{N}$, $\bar{Z}_m(t)=(\bar{X}_m,\bar{V}_m)\in G_m(\epsilon_0,0)$, $\bar{Z}_m(0) \in \Delta_m(\epsilon_0)$, and $X_m\in B_{\alpha/2}^{dm}(\bar{X}_m)$. Then, for $k \in\{1,\cdots,m\}$ there is a subset $\mathcal{B}_{k}^\ell(\bar{Z}_m)\subseteq (\E\times B_R^{\ell d})^+(\bar{v}_k)$ such that:
\begin{enumerate}
\item For any $(\omega_1, \cdots, \omega_\ell,v_{m+1}, \cdots, v_{m+\ell})\in (\E\times B_R^{\ell d})^+(\bar{v}_{k})\setminus\mathcal{B}_{k}^\ell(\bar{Z}_m)$, one has:
\begin{align}
Z_{m+\ell}(t)&\in \mathring{\mathcal{D}}_{m+\ell,\bm{\epsilon_\ell}},\quad\forall t> 0,\label{in phase pre}\\
Z_{m+\ell}&\in G_{m+\ell}(\epsilon_0/2,\delta)\label{epsilon/2 pre}\\
\bar{Z}_{m+\ell}&\in G_{m+\ell}(\epsilon_0,\delta),\label{epsilon pre}
\end{align}
where 
\begin{equation}\label{pre-collisional notation ternary}
\begin{aligned}
&Z_{m+\ell}=(x_1,\cdots,x_k,\cdots,x_m,x_{m+1},\cdots,x_{m+\ell},\bar{v}_1,\cdots,\bar{v}_k,\cdots,\bar{v}_m,v_{m+1},\cdots,v_{m+\ell}),\\
&x_{m+i}=x_{k}-\epsilon_{\ell+1}\omega_{i},\quad\forall i\in\{1,\cdots,\ell\},\\
&\bar{Z}_{m+\ell}=(\bar{x}_1,\cdots,\bar{x}_k,\cdots,\bar{x}_m,\bar{x}_{m+1},\cdots,\bar{x}_{m+\ell},\bar{v}_1,\cdots,\bar{v}_k,\cdots,\bar{v}_m,v_{m+1},\cdots,v_{m+\ell}),\\
&\bar{x}_{m+i}=\bar{x}_{k}-\epsilon_{\ell+1}\omega_{i},\quad\forall i\in\{1,\cdots,\ell\}.
\end{aligned}
\end{equation}
\item For any $(\omega_1,\cdots,\omega_\ell,v_{m+1},\cdots,v_{m+\ell})\in (\E\times B_R^{\ell d})^+(\bar{v}_k)\setminus\mathcal{B}_{k}^\ell(\bar{Z}_m)$, one has:
\begin{align}
Z_{m+\ell}^*(t)&\in \mathring{\mathcal{D}}_{m+\ell,\bm{\epsilon_\ell}},\quad\forall t> 0,\label{in phase post}\\
Z_{m+\ell}^*&\in G_{m+\ell}(\epsilon_0/2,\delta),\label{epsilon/2 post}\\
\bar{Z}_{m+\ell}^*&\in G_{m+\ell}(\epsilon_0,\delta),\label{epsilon post}
\end{align}
where 
\begin{equation}\label{post-collisional notation ternary}
\begin{aligned}
&Z_{m+\ell}^*=(x_1,\cdots,x_k,\cdots,x_m,x_{m+1},\cdots,x_{m+\ell},\bar{v}_1,\cdots,\bar{v}_k^*,\cdots,\bar{v}_m,v_{m+1}^*,\cdots,v_{m+\ell}^*),\\
&x_{m+i}=x_{k}+\epsilon_{\ell+1}\omega_{i},\quad\forall i\in\{1,\cdots,\ell\},\\
&\bar{Z}_{m+\ell}^*=(\bar{x}_1,\cdots,\bar{x}_k,\cdots,\bar{x}_m, \bar{x}_{m+1},\cdots,\bar{x}_{m+\ell},\bar{v}_1,\cdots,\bar{v}_k^*,\cdots,\bar{v}_m,v_{m+1}^*,\cdots,v_{m+\ell}^*),\\
&\bar{x}_{m+i}=\bar{x}_{k}+\epsilon_{\ell+1}\omega_{i},\quad\forall i\in\{1,\cdots,\ell\},\\
&(\bar{v}_{k}^*,v_{m+1}^*,\cdots,v_{m+\ell}^*)=T^{\ell+1}_{\omega_1,\cdots,\omega_\ell}(\bar{v}_{k},v_{m+1},\cdots,v_{m+\ell}).
\end{aligned}
\end{equation}
\end{enumerate}
Furthermore, the following measure estimate holds:
\begin{equation}\label{B measure est}
    |\mathcal{B}^\ell_k(\bar{Z}_m)| \lesssim m R^{\ell d} \eta^{\frac{d-1}{2 \ell d + 2}}. 
\end{equation}

\end{proposition}

We first will prove two Lemmas which will help in the proof of Proposition \ref{bad set triple}.

\begin{lemma}
    Assume the hypotheses of Proposition \ref{bad set triple} and fix $1 \leq i < j \leq m$. Then, we have the estimate,
    \begin{equation} \label{x alpha triangle e02}
        |x_i(t) - x_j(t)| > \frac{\epsilon_0}{2}, \quad \forall t \geq 0.
    \end{equation}
    Furthermore, since $\epsilon_0 >> \epsilon_{\ell+1}$ for all $\ell \in \{1, \cdots, M \}$, this also implies the estimate,
    \begin{equation}\label{x alpha triangle}
        |x_i(t) - x_j(t)| > \epsilon_{\ell+1}, \quad \forall t \geq 0.
    \end{equation}
\end{lemma}
\begin{proof}
    Recalling that $X_m \in B^{dm}_{\alpha/2}(\bar{X}_m)$, $\bar{Z}_m(0) \in \Delta_m(\epsilon_0)$, and $\bar{Z}_m(t) \in G_m(\epsilon_0,0)$, we can compute via the triangle inequality that for all $t\geq 0$,
    \begin{align}
        |x_i(t) - x_j(t)| &= |\bar{x_i} - \bar{x}_j - t(\bar{v}_i - \bar{v}_j) + (x_i - x_j) - (\bar{x}_i - \bar{x}_j)| \\
        &\geq \epsilon_0 - \alpha \\
        &>\frac{\epsilon_0}{2} \\
        &> \epsilon_{\ell+1}.
    \end{align}
\end{proof}

Without loss of generality, we will prove Proposition \ref{bad set triple} for $k = m$. Therefore, we present the following Lemma for $k=m$ only.

\begin{lemma}\label{int phase space setup lemma}
    Assume the hypotheses of Proposition \ref{bad set triple} with $k = m$. Further, assume that,
    \begin{align}
        |x_i(t)-x_j(t)| &> \epsilon_{\ell+1} \quad \text{ for all } 1\leq i< m \text{ and } m \leq j \leq m+\ell, \label{eps ell bound} \\
        |x_i(t)-x_j(t)| &> \epsilon_{\ell} \quad \text{ for all } m \leq i < j \leq m+\ell, \label{eps ell minus one bound}
\end{align} 
holds for all $t>0$.
Then, \eqref{in phase pre} holds.

\end{lemma}
\begin{proof}
    Fix $n \in \{ 1, \cdots, \ell \}$. In order to show \eqref{in phase pre} it is equivalent to prove that for every $n+1$-tuple of particles $1\leq i_1 < \cdots<  i_{n+1} \leq m+\ell$, we have 
    \begin{equation}
        d_{n+1}(x_{i_1}(t), \cdots, x_{i_{n+1}}(t)) >\epsilon_{n+1} \quad \forall t > 0.
    \end{equation}

    We proceed by splitting into three cases.
    \par
    \textit{(1): $n=\ell$ and $(i_1, \cdots, i_{n+1}) = (m,\cdots, m+\ell)$.} We can easily calculate,
    \begin{align}
        d^2_{\ell+1}(x_m(t),&\cdots,x_{m+\ell}(t)) = \sum_{i=1}^\ell \big| \epsilon_{\ell+1} \omega_i + t(v_{m+i} - \bar{v}_m) \big|^2 + \sum_{1\leq i < j \leq \ell} \big| \epsilon_{\ell+1} (\omega_i - \omega_j) + t(v_{m+i} - v_{m+j}) \big|^2 \\
        &\geq \epsilon_{\ell+1}^2 \big( \sum_{i=1}^\ell |\omega_i|^2 + \sum_{1\leq i < j \leq \ell} |\omega_i - \omega_j|^2  \big) + 2\epsilon_{\ell+1} t b_{\ell+1}(\omega_1,\cdots, \omega_\ell, v_{m+1} - \bar{v}_m, \cdots, v_{m+\ell} - \bar{v}_m)\\
        &>\epsilon^2_{\ell+1} \quad \forall t > 0,
    \end{align}
    where we utilize the fact that 
    \begin{equation}
        (\omega_1, \cdots, \omega_\ell,v_{m+1}, \cdots, v_{m+\ell})\in (\E\times B_R^{\ell d})^+(v_m) \implies b_{\ell+1}(\omega_1,\cdots, \omega_\ell, v_{m+1} - \bar{v}_m, \cdots, v_{m+\ell} - \bar{v}_m) > 0.
    \end{equation}
    \par
    \textit{(2): $n=\ell$ and  $i_1 < m$.}
    We calculate,
    \begin{align}
        d_{\ell+1}^2(x_{i_1}(t) ,\cdots, x_{i_{\ell+1}}(t)) &\geq |x_{i_1}(t) - x_{i_2}(t)|^2 \\
        &> \epsilon_{\ell+1}^2, \label{phase space pre prop}
    \end{align}
    where \eqref{phase space pre prop} holds for $i_2 > m $ by \eqref{eps ell bound}, and holds for $i_2 \leq m$ by \eqref{x alpha triangle}.
    \par
    \textit{(3): $n < \ell$.}
    If we let $(i_1,\cdots, i_{n+1})$ indicate the $n+1$-tuple, then by the pigeon hole principle we apply \eqref{x alpha triangle} if $i_1 < i_2 \leq m$,   \eqref{eps ell bound} if $i_1 < m \leq i_2$, or \eqref{eps ell minus one bound} if $m \leq i_1 < i_2$ to obtain in all cases
    \begin{align}
        d_{n+1}(x_{i_1}(t), \cdots, x_{i_{n+1}}(t)) > \epsilon_{n+1}.
    \end{align}
    Therefore, \eqref{in phase pre} follows, i.e. $Z_{m+\ell}(t) \in \mathring{\mathcal{D}}_{m+\ell,\bm{\epsilon_\ell}}$ for all $t>0$.
    
\end{proof}

 \begin{proof}[Proof of Proposition \ref{bad set triple}]

Without loss of generality, we set $k=m$. We will start with the pre-collisional part.
Lemma \ref{int phase space setup lemma} states that \eqref{eps ell bound} and \eqref{eps ell minus one bound} imply \eqref{in phase pre}. Our aim thus is to prove \eqref{eps ell bound}, \eqref{eps ell minus one bound}, \eqref{epsilon/2 pre}, and \eqref{epsilon pre}.  We proceed by splitting into cases on $1\leq i < j \leq m+\ell$.

\par
\textit{Case 1: Fix $i,j$ such that $1 \leq i < j \leq m$.}
\par
We first note that in this case \eqref{eps ell minus one bound} does not apply and \eqref{eps ell bound} is only partially covered for $1\leq i < m$ and $j=m$.
Recalling the calculation to obtain \eqref{x alpha triangle e02}, we have
\begin{equation}
    |x_i(t) - x_j(t)| > \frac{\epsilon_0}{2} > \epsilon_{\ell+1}, \quad \text{ for all }t > 0,
\end{equation}
satisfying \eqref{epsilon/2 pre}. Furthermore, \eqref{eps ell bound} is satisfied for the case of $1\leq i < m $ and $j = m$. $\bar{Z}_m(0) \in \Delta_m(\epsilon_0)$ and $\bar{Z}_m \in G_m(\epsilon_0,0)$ ensures that \eqref{epsilon pre} is satisfied. 
\par
\textit{Case 2: Fix $i,j$ such that $1\leq i < m, m+1 \leq j \leq m+\ell$.}
\par 
Once again, in this case \eqref{eps ell minus one bound} does not apply but we complete the proof for \eqref{eps ell bound}. We denote $n_j \in \{ 1, \cdots, \ell \}$ such that $j = m +n_j$. 
From $\bar{Z}_m(0) \in \Delta_m(\epsilon_0)$ we have,
\begin{equation}
    |\bar{x}_i - \bar{x}_m| > \epsilon_0,
\end{equation}
and from $X_m \in B^{dm}_{\alpha / 2}(\bar{X}_m)$ we have,
\begin{align}
        |x_{j} - \bar{x}_m| &= |x_{m+n_j} - \bar{x}_m|\\
        &\leq |x_m - \bar{x}_m| + |x_{m+n_j} -x_m|, \\
        &\leq \frac{\alpha}{2} + \epsilon_{\ell+1} |\omega_{n_j}| \\
        &< \alpha.
\end{align}
 By applying Lemma \ref{cylinder lemma} for $\bar{y}_1 = \bar{x}_i $, $\bar{y}_2 = \bar{x}_m$, $y_1 = x_i$, $y_2 = x_{m+n_j}$, and $\epsilon = \epsilon_{\ell +1}$, there exists a cylinder $K_\eta^{d,i}$ such that for all $v_j \in B_R^d \setminus K_\eta^{d,i}$ 
\begin{equation}
    |x_i(t) - x_j(t)| > \epsilon_{\ell+1} \quad \forall t > 0, \label{case 2 eps ell bound}
\end{equation}
and
\begin{equation}
    |x_i(t) - x_j(t)| \geq \epsilon_0 \quad \forall t > \delta. \label{case 2 eps 0 bound}
\end{equation}
\eqref{case 2 eps ell bound} implies that \eqref{eps ell bound} is satisfied and \eqref{case 2 eps 0 bound} implies that \eqref{epsilon/2 pre} and \eqref{epsilon pre} are both satisfied. 
We will denote,
\begin{equation}
    U_{m+n_j}^i = \mathbb{E}^{\ell d-1}_1 \times B_R^d \times \cdots \times B_R^d \times K^{d,i}_\eta \times B^d_R \times \cdots \times B_R^d,
\end{equation}
where the cylinder $K_\eta^{d,i} $ occurs in the $n_j$th spot. 

\par 
\textit{Case 3: Fix $i,j$ such that $m\leq i < j \leq m+\ell$.}
\par
 \eqref{eps ell bound} is handled by the prior two cases and in this case we prove the entirety of \eqref{eps ell minus one bound}. We denote $n_i, n_j \in \N $, $1 \leq n_i< n_j\leq \ell$ such that $i = m+ n_i$ and $j = m + n_j$. For ease of notation, we present the proof for when  $m<i$ and note that the proof for $i=m $ follows in an identical way with $n_i$ set to $0$ and $v_{m+n_i}$ set to $\bar{v}_m$. Proceeding, we denote 
\begin{align}
    \gamma &= \frac{\epsilon_{\ell}}{\epsilon_{\ell+1}} << \eta^2, \label{gamma def} \\
    \gamma'&= (1- \gamma)^\frac{1}{2},
\end{align}
and note that $\gamma << 1$ and $1-\gamma'<< 1$.
Inspired by the procedure introduced in section 9 of \cite{AmpatzoglouPavlovic2020}, consider the polynomial
\begin{align} \label{polynomial P}
    P(t) = \big|x_i(t) - x_j(t) \big|^2 - \gamma \epsilon_{\ell+1}^2 |\omega_{n_i} - \omega_{n_j}|^2.
\end{align}
We expand $|x_i(t) - x_j(t)|^2$,
\begin{align}
        |x_i(t) - x_j(t)|^2 &= |x_{m+n_i}(t) - x_{m+n_j}(t)|^2 \\ 
        &= |\epsilon_{\ell+1} (\omega_{n_i} - \omega_{n_j}) + (v_{m+n_i} - v_{m+n_j})t |^2 \\ 
        &= \epsilon_{\ell+1}^2 |\omega_{n_i} - \omega_{n_j}|^2 + 2\epsilon_{\ell+1} t \langle \omega_{n_i} - \omega_{n_j}, v_{m+n_i}- v_{m+n_j}\rangle + t^2 |v_{m+n_i}- v_{m+n_j}|^2,\label{xi minus xj}
\end{align}
which allows us to write \eqref{polynomial P} as,
\begin{equation}
    P(t) = (1- \gamma )\epsilon_{\ell+1}^2 |\omega_{n_i} - \omega_{n_j}|^2 + 2\epsilon_{\ell+1} t \langle \omega_{n_i} - \omega_{n_j}, v_{m+n_i}- v_{m+n_j}\rangle + t^2 |v_{m+n_j}- v_{m+n_i}|^2.
\end{equation}
We define the sets,
\begin{align}
    \Omega_{n_i,n_j} &= \{(\omega_1,\cdots, \omega_\ell, v_{m+1}, \cdots, v_{m+\ell}) \in \mathbb{E}^{\ell d -1}_1 \times B_R^{\ell d} \, : \, |\omega_{n_i}-\omega_{n_j}| \leq \sqrt{\gamma}\} \\
    A_{n_i,n_j} &= \{(\omega_1,\cdots, \omega_\ell, v_{m+1}, \cdots, v_{m+\ell}) \in \mathbb{E}^{\ell d -1}_1 \times B_R^{\ell d} \, :\, \\
    &|\langle \omega_{n_i}- \omega_{n_j}, v_{m+n_i} - v_{m+ n_j} \rangle| \geq \gamma' |\omega_{n_i} - \omega_{n_j}||v_{m+ n_i} - v_{m+n_j}| \}.
\end{align}
The discriminant for $P$ is given by,
\begin{equation}
    \begin{split}
        \Delta &= 4 \epsilon_{\ell+1}^2 |\langle \omega_{n_i}- \omega_{n_j}, v_{m+n_i}- v_{m+n_j}\rangle |^2 - 4(1- \gamma) \epsilon_{\ell+1}^2 |\omega_{n_i}-\omega_{n_j}|^2 |v_{m+n_i}- v_{m+n_j}|^2 \\
        &= 4\epsilon_{\ell+1}^2 \bigg( |\langle \omega_{n_i}- \omega_{n_j}, v_{m+n_i}- v_{m+n_j}\rangle |^2 - \gamma'^2 |\omega_{n_i}-\omega_{n_j}|^2 |v_{m+n_i}- v_{m+n_j}|^2 \bigg).
    \end{split}
\end{equation}
We note that for $(\omega_1,\cdots, \omega_\ell, v_{m+1}, \cdots, v_{m+\ell}  ) \notin A_{n_i,n_j}$ we have $\Delta < 0$. We additionally notice that since $\gamma << 1$, $P(0) > 0$ which implies $P(t) > 0$ for all $t \in \R$. Hence, if we additionally have  $(\omega_1,\cdots, \omega_\ell, v_{m+1}, \cdots, v_{m+\ell}) \notin  \Omega_{n_i,n_j}$ we obtain, 
\begin{equation}
    |x_i(t) - x_j(t)|^2 > \gamma \epsilon_{\ell+1}^2 |\omega_{n_i} - \omega_{n_j}|^2 > \epsilon_{\ell }^2 
\end{equation}
thus satisfying \eqref{eps ell minus one bound}. 
\par
We now move our attention to proving \eqref{epsilon/2 pre} and \eqref{epsilon pre}. We define the set
\begin{equation}
    V_{n_i,n_j} = \{(\omega_1,\cdots, \omega_\ell, v_{m+1}, \cdots, v_{m+\ell}) \in \mathbb{E}^{\ell d -1}_1 \times B_R^{\ell d} \, : \, |v_{m+n_i} - v_{m+n_j}| < \eta  \}.
\end{equation}
and note that for $(\omega_1,\cdots, \omega_\ell, v_{m+1}, \cdots, v_{m+\ell}) \notin  V_{n_i,n_j}$, $t\geq \delta$, and recalling that $\eta \delta >> \epsilon_{\ell+1}$, the following calculation holds:
\begin{equation}
\begin{split}
    |x_i(t) - x_j(t)| &= |\epsilon_{\ell+1} (\omega_{n_i}- \omega_{n_j}) - t(v_{m+n_i} - v_{m+n_j})| \\
    &\geq |v_{m+n_i} - v_{m+n_j}|\delta - 2\epsilon_{\ell+1} \\
    &>\eta \delta - 2\epsilon_{\ell+1} \\
    &>\epsilon_0,
    \end{split}
\end{equation}
thus satisfying \eqref{epsilon/2 pre}. By noticing that we also have $|\bar{x}_i - \bar{x}_j| = |\epsilon_{\ell+1} (\omega_i - \omega_j)|$, we can repeat an identical calculation to prove \eqref{epsilon pre}.
These same calculation can be repeated for the case in which $i=m$ obtaining analogous sets.
\par
Therefore, by employing the notation $\bm{\omega} = (\omega_1, \cdots, \omega_\ell)$ and $\bm{\tilde{v}_m} = (v_{m+1}, \cdots, v_{m+\ell})$ we can succinctly write all pre-collisional sets as,
\begin{align}
    \Omega_{i} &= \{(\bm{\omega}, \bm{\tilde{v}_m}) \in \mathbb{E}^{\ell d -1}_1 \times B_R^{\ell d} \, : \, |\omega_{i}| \leq \sqrt{\gamma}\}, \quad 1\leq i \leq \ell\\
    \Omega_{i,j} &= \{(\bm{\omega}, \bm{\tilde{v}_m}) \in \mathbb{E}^{\ell d -1}_1 \times B_R^{\ell d} \, : \, |\omega_{i}-\omega_{j}| \leq \sqrt{\gamma}\}, \quad 1\leq i < j\leq \ell \\
    U_{m+i}^j &= \{(\bm{\omega}, \bm{\tilde{v}_m}) \in \mathbb{E}^{\ell d -1}_1 \times B_R^{\ell d} \, : \, v_{m+i} \in K_\eta^{d,j} \}, \quad 1\leq i \leq \ell, \quad 1\leq j < m \\
    V_{i} &= \{(\bm{\omega}, \bm{\tilde{v}_m}) \in \mathbb{E}^{\ell d -1}_1 \times B_R^{\ell d} \, : \, |v_{m+i} - \bar{v}_{m}| \in B^d_\eta  \}, \quad 1\leq i \leq \ell \\
    V_{i,j} &= \{(\bm{\omega}, \bm{\tilde{v}_m}) \in \mathbb{E}^{\ell d -1}_1 \times B_R^{\ell d} \, : \, |v_{m+i} - v_{m+j}| \in B^d_\eta  \}, \quad 1\leq i < j \leq \ell \\
    A_{i} &= \{(\bm{\omega}, \bm{\tilde{v}_m}) \in \mathbb{E}^{\ell d -1}_1 \times B_R^{\ell d} \, :\, |\langle \omega_{i}, v_{m+i} - \bar{v}_{m} \rangle| \geq \gamma' |\omega_{i} ||v_{m+ i} - \bar{v}_{m}| \}, \quad 1\leq i \leq \ell \\
    A_{i,j} &= \{(\bm{\omega}, \bm{\tilde{v}_m}) \in \mathbb{E}^{\ell d -1}_1 \times B_R^{\ell d} \, :\, |\langle \omega_{i}- \omega_{j}, v_{m+i} - v_{m+ j} \rangle| \geq \gamma' |\omega_{i} - \omega_{j}||v_{m+ i} - v_{m+j}| \}, \quad  1\leq i < j \leq \ell .
\end{align}
The proof follows the identical steps for the post-collisional case in which the velocities $\bar{v}_m, v_{m+1}, \cdots, v_{m+\ell}$ are replaced with $\bar{v}^*_m, v^*_{m+1}, \cdots, v^*_{m+\ell}$. Due to $\bar{v}^*_m$ depending on the impact parameters, $\bm{\omega}$, we must now include
\begin{equation}
    U_m^{j,*} = \{(\bm{\omega}, \bm{\tilde{v}_m}) \in \mathbb{E}^{\ell d -1}_1 \times B_R^{\ell d} \, : \, \bar{v}_m^* \in K_\eta^{d,i} \}, \quad \forall 1\leq j < m,
\end{equation}
in our collection of post-collisional pathological sets.  We list the post-collsional pathological sets below, omitting the sets $\{\Omega_i\}_{i=1}^\ell$ and $\{\Omega_{i,j}\}_{1\leq i< j\leq \ell}$ which are included in the post-collisional collection but remain unchanged. 
\begin{align}
    U_m^{j,*} &= \{(\bm{\omega}, \bm{\tilde{v}_m}) \in \mathbb{E}^{\ell d -1}_1 \times B_R^{\ell d} \, : \, \bar{v}_m^* \in K_\eta^{d,j} \}, \quad 1\leq j < m \\
    U_{m+i}^{j,*} &= \{(\bm{\omega}, \bm{\tilde{v}_m}) \in \mathbb{E}^{\ell d -1}_1 \times B_R^{\ell d} \, : \, v_{m+i}^* \in K_\eta^{d,j} \}, \quad 1\leq i \leq \ell, \quad 1\leq j < m \\
    V^*_{i} &= \{(\bm{\omega}, \bm{\tilde{v}_m}) \in \mathbb{E}^{\ell d -1}_1 \times B_R^{\ell d} \, : \, |v^*_{m+i} - \bar{v}^*_{m}| \in B^d_\eta  \}, \quad 1\leq i \leq \ell \\
    V^*_{i,j} &= \{(\bm{\omega}, \bm{\tilde{v}_m}) \in \mathbb{E}^{\ell d -1}_1 \times B_R^{\ell d} \, : \, |v^*_{m+i} - v^*_{m+j}| \in B^d_\eta  \}, \quad 1\leq i < j \leq \ell \\
    A_{i}^* &= \{(\bm{\omega}, \bm{\tilde{v}_m}) \in \mathbb{E}^{\ell d -1}_1 \times B_R^{\ell d} \, :\, |\langle \omega_{i}, v^*_{m+i} - \bar{v}^*_{m} \rangle| \geq \gamma' |\omega_{i}||v^*_{m+ i} - \bar{v}^*_{m}| \}, \quad  1\leq i \leq \ell \\
    A_{i,j}^* &= \{(\bm{\omega}, \bm{\tilde{v}_m}) \in \mathbb{E}^{\ell d -1}_1 \times B_R^{\ell d} \, :\, |\langle \omega_{i}- \omega_{j}, v^*_{m+i} - v^*_{m+ j} \rangle| \geq \gamma' |\omega_{i} - \omega_{j}||v^*_{m+ i} - v^*_{m+j}| \}, \quad  1\leq i < j \leq \ell .
\end{align}
Furthermore, we will denote
\begin{align}
    {\Omega} &= \bigg(\bigcup_{i=1}^\ell \Omega_i\bigg) \cup \bigg(\bigcup_{1\leq i<j\leq \ell} \Omega_{i,j}\bigg)\\
    \mathcal{U^{*}} &= \bigg( \bigcup_{\substack{1\leq i \leq \ell \\ 1\leq j < m}} U^{j,*}_{m+i} \bigg) \cup \bigg( \bigcup_{1\leq j < m} U^{j,*}_{m} \bigg), \quad 
    \mathcal{U} = \bigg( \bigcup_{\substack{1\leq i \leq \ell \\ 1\leq j < m}} U^j_{m+i} \bigg)
    \\
    \mathcal{V^*} &= \bigg( \bigcup_{i=1}^\ell V^*_i \bigg) \cup \bigg( \bigcup_{1\leq i<j\leq \ell}V^*_{i,j} \bigg), \hspace{11mm} \mathcal{V} = \bigg( \bigcup_{i=1}^\ell V_i \bigg) \cup \bigg( \bigcup_{1\leq i<j\leq \ell}V_{i,j} \bigg)  \\
    \mathcal{A^*} &= \bigg( \bigcup_{i= 1}^\ell A^*_{i} \bigg) \cup \bigg( \bigcup_{1\leq i < j \leq \ell} A^*_{i,j}\bigg), \hspace{11mm} \mathcal{A} = \bigg( \bigcup_{i= 1}^\ell A_{i} \bigg) \cup \bigg( \bigcup_{1\leq i < j \leq \ell} A_{i,j}\bigg)
\end{align}
Therefore, we can write the pre- and post-collisional pathological sets as,
\begin{align}
    \mathcal{B}_m^{\ell,+}(Z_m) &= (\E \times B^{\ell d}_R)^+(\bar{v}_m) \cap \big({\Omega} \cup \mathcal{U^{*}} \cup \mathcal{V^{*}} \cup \mathcal{A^{*}}\big) \\
    \mathcal{B}_m^{\ell,-}(Z_m)  &= (\E \times B^{\ell d}_R)^+(\bar{v}_m) \cap \big( {\Omega} \cup \mathcal{U} \cup \mathcal{V} \cup \mathcal{A}\big)
\end{align}
where,
\begin{equation}
    \mathcal{B}^\ell_m(Z_m) = \mathcal{B}_m^{\ell,+} (Z_m) \cup \mathcal{B}_m^{\ell,-}(Z_m)
\end{equation}
We will prove the desired measure estimate of $\mathcal{B}^\ell_m$ given in \eqref{B measure est} in the following two sections. 
\end{proof}

\subsection{Measure Estimates of the Pre-Collisional Pathological Sets}
Since the collisional law does not come into play, the pre-collisional pathological set estimates follows the calculation done in \cite{AmpatzoglouPavlovic2020} and \cite{ternary}. The main difference that occurs is that the estimates in this paper take place over the ellipsoid $\E$ instead of the sphere. We first will establish some useful notation for the velocities. We denote
\begin{align}
    {\bm{{v}_m}} &= (\bar{v}_m, v_{m+1}, \cdots, v_{m+\ell}) \in \R^{(\ell +1)d}, \label{v_m} \\ 
    \bm{v_m}' &=  (v_{m+1} -\bar{v}_m, \cdots, v_{m+\ell}- \bar{v}_m) \in \R^{\ell d},\label{v_m'}  \\ 
    \bm{\tilde{v}_m} &= ( v_{m+1}, \cdots, v_{m+\ell}) \in \R^{\ell d}. \label{v_m tilde}
\end{align}
\subsubsection{Estimates of $\Omega$}
We first note that in the case of $\ell = 1$ our ellipsoid $\E = \mathbb{S}^{d-1}_1$. In this case, we can use the estimate for $\Omega_1$ given in \cite{AmpatzoglouPavlovic2020} to find,
\begin{equation}
    |\Omega_1| \lesssim R^{ d} \gamma^{{d/2}}, \quad \ell =1.
\end{equation}
For the case of $\ell \geq 2$, we can obtain a similar estimate by first applying \eqref{ellipsoid to sphere block} to change variables via an invertible transformation $T_i:\E \rightarrow \mathbb{S}^{\ell d -1}_1$ which leaves the the $i^{\text{th}}$ component invariant up to a rescaling. More precisely, we have
\begin{equation}
    T_i (\Omega_i^{\omega}) = \{\bm{\omega} \in \mathbb{S}^{\ell d -1}_1  \, : \, |\omega_{i}| \leq \frac{1}{\lambda} \sqrt{\gamma}\},
\end{equation}
where $\Omega_i^{\omega}$ denotes the projection of $\Omega_i$ onto $\E$.

Applying this change of variables we can then use the estimate already done for a rescaled version of $\Omega_i$ in \cite{AmpatzoglouPavlovic2020} to obtain 
\begin{align}
    |\Omega_{i}| &\lesssim R^{\ell d}\int_{\E} \mathds{1}_{\Omega_i^{\omega}}(\bm{\omega}) d\bm{\omega} \\
    &\lesssim R^{\ell d}\int_{\mathbb{S}_1^{(\ell d - 1)}}\mathds{1}_{T_i (\Omega_i^{\omega})}(\bm{\nu}) d\bm{\nu} \\
    &\lesssim R^{\ell d} \gamma^{d/2}.
\end{align}
Furthermore, by applying Lemma \ref{W strip est} we obtain
\begin{align}
    |\Omega_{i,j}| \lesssim R^{\ell d} \gamma^{\frac{d-1}{4}} \quad  1\leq i < j \leq \ell.
\end{align}
By recalling \eqref{gamma def}, we have $\gamma << \eta^2$. Therefore, the following inequalities hold
\begin{align}
    \gamma^{d/2} & \lesssim \eta \\
    \gamma^{\frac{d-1}{4}} &\lesssim \eta^{\frac{d-1}{2}},
\end{align}
and we obtain the desired estimate,
\begin{equation}
    |\Omega| \lesssim R^{\ell d} (\eta + \eta^{\frac{d-1}{2}}) \lesssim m R^{\ell d} \eta^{\frac{d-1}{2 \ell d + 2}}.
\end{equation}
\subsubsection{Estimates of $\mathcal{U}$ and $\mathcal{V}$.}
These collections of sets depend only on velocities and thus are unaffected by the ellipsoid $\E$. Thus, we can simply use the same estimates as done in \cite{AmpatzoglouPavlovic2020} and \cite{ternary} which are,
\begin{align}
    |U^j_{m+i}| &\lesssim R^{\ell d} \eta^{\frac{d-1}{2}}, \quad  1\leq i \leq \ell, \quad 1 \leq j < m,\\
    |V_i| &\lesssim R^{\ell d} \eta^d, \quad 1\leq i \leq \ell,\\
    |V_{i,j}| &\lesssim R^{\ell d} \eta^d, \quad 1\leq i < j\leq \ell.
\end{align}
Therefore, the desired estimate holds,
\begin{equation}
    |\mathcal{U} \cup \mathcal{V}| \lesssim m R^{\ell d} \eta^{\frac{d-1}{2\ell d  + 2}}.
\end{equation}
\subsubsection{Estimates of $\mathcal{A}$}
By Fubini and Lemma \ref{conic est}, for all $1\leq i \leq \ell$ we have
\begin{align}
    |A_i| &= \int_{\E \times B^{\ell d}_R} \ind_{A_{i}}(\bm{\omega}, \bm{\tilde{v}_m}) d\bm{\omega} d\bm{\tilde{v}_m}\\
    &= \int_{B^{\ell d}_R} \int_{\E} \ind_{S(\gamma', v_{m+i} - \bar{v}_m)} (\omega_i) d\bm{\omega} d\bm{\tilde{v}_m}\\
    &\lesssim R^{\ell d} \arccos{\sqrt{1-\frac{\gamma}{2}  }}.
\end{align}
Similarly, for the collection $\{A_{i,j}\}_{1\leq i < j \leq \ell}$, by Fubini and Lemma \ref{conic est N} we have
\begin{align}
    |A_{i,j}| &= \int_{\E \times B^{\ell d}_R} \ind_{A_{i,j}}(\bm{\omega}, \bm{\tilde{v}_m}) d\bm{\omega} d\bm{\tilde{v}_m}\\
    &= \int_{B^{\ell d}_R} \int_{\E} \ind_{N(\gamma', v_{m+i} - v_{m+j})} (\omega_i,\omega_j) d\bm{\omega} d\bm{\tilde{v}_m}\\
    &\lesssim R^{\ell d} \arccos{\sqrt{1-\frac{\gamma}{2}  }}.
\end{align}
Since $\eta << 1$ we have the inequalities
\begin{equation}
    \frac{\eta}{\sqrt{2}} \leq \sin \eta \leq \eta.
\end{equation}
Recalling $\gamma << \eta^2$ from \eqref{gamma def}, we have
\begin{equation}\label{arccos gamma eta}
    \gamma << 2 \sin^2 \eta \implies \arccos \sqrt{1- \frac{\gamma}{2}} < \eta.
\end{equation}
Therefore, we obtain the estimates,
\begin{equation}
    |\mathcal{A}| \lesssim R^{\ell d} \eta \lesssim m R^{\ell d} \eta^{\frac{d-1}{2\ell d + 2}}.
\end{equation}

\subsection{Measure Estimates of the Post-Collisional Pathological Sets}
Here, we provide measure estimates for the post-collisional pathological sets.
\subsubsection{Estimates of $\mathcal{U^*}$ and $\mathcal{V^*}$ }
We adapt to our context the program presented in \cite{ternary} by applying the co-area formula with respect to the transformation $\Phi_{\bar{v}_m}: \R^{\ell d} \rightarrow \R$ defined by, 
\begin{equation}\label{phi derivative bounds}
    \Phi_{\bar{v}_m}(v_{m+1}, \cdots, v_{m+\ell}) = \sum_{i=1}^\ell |v_{m+i}-\bar{v}_m|^2 + \sum_{1\leq i < j\leq \ell} |v_{m+i} - v_{m+j}|^2,
\end{equation}
which remains invariant under the properties of the collision according to Proposition \ref{coll prop}.
By a simple calculation we see that given $r > 0$ and for all $(v_{m+1}, \cdots, v_{m+\ell})\in  \Phi_{\bar{v}_m}^{-1}(\{ r^2 \})$ we have,
\begin{equation}
    2r \leq |\nabla \Phi_{\bar{v}_m}(v_{m+1}, \cdots, v_{m+\ell})| \leq 4r.
\end{equation}
We also define the set $G^{\ell d}_R(\bar{v}_m):= [0 \leq \Phi_{\bar{v}_m} \leq 16 R^2]$. By the triangle inequality and $\bar{v}_m \in B^d_R$ we have
\begin{equation}
    B^{\ell d}_R \subset G^{\ell d}_R(\bar{v}_m).
\end{equation}
Recall $\mathcal{E}^+_{\bm{{v}_m}}$ from \eqref{mathcal E} and use the co-area formula and Fubini to derive the estimate
\begin{align}
    |\mathcal{U^*} \cup \mathcal{V^*}| &= \int_{(\mathbb{E}^{\ell d -1} \times B^{\ell d}_R)^+} \mathds{1}_{\mathcal{U}^*\cup \mathcal{V^*}} d \bm{\omega} d\bm{\tilde{v}_m} \\
    &\leq \int_{G^{\ell d}_R(\bar{v}_m)} \int_{\mathcal{E}^+_{\bm{{v}_m} }} \ind_{\mathcal{U}^*\cup \mathcal{V^*}} d\bm{\omega} d\bm{\tilde{v}_m}\\
    &= \int_0^{16 R^2} \int_{\Phi^{-1}_{\bar{v}_m}(\{s \})} |\nabla \Phi_{\bar{v}_m}(\bm{v})|^{-1} \int_{\mathcal{E}^+_{\bm{{v}_m}}} \ind_{\mathcal{U}^*\cup \mathcal{V^*}} d\bm{\omega} d\sigma(\bm{\tilde{v}_m}) ds\\
    &= \int_0^{4R} 2r \int_{\Phi^{-1}(\{ r^2\})} |\nabla \Phi_{\bar{v}_m}(\bm{v})|^{-1} \int_{\mathcal{E}^+_{\bm{{v}_m}}} \ind_{\mathcal{U}^*\cup \mathcal{V^*}} d\bm{\omega} d\sigma(\bm{\tilde{v}_m}) dr\\
    &\lesssim \int_0^{4R} \int_{\Phi^{-1}_{\bar{v}_m}(\{ r^2\})} \int_{\mathcal{E}^+_{\bm{\tilde{v}_m}}} \mathds{1}_{\mathcal{U}^*\cup \mathcal{V^*}}(\bm{\omega}) d \bm{\omega} d\sigma(\bm{\tilde{v}_m}) dr, \label{last line U star}
\end{align}
where \eqref{last line U star} follows from the lower bound of \eqref{phi derivative bounds}. We thus aim to estimate the integral 
\begin{equation}\label{U star general est}
    \int_{\mathcal{E}^+_{\bm{v_m}}} \mathds{1}_{\mathcal{U}^*\cup \mathcal{V^*}}(\bm{\omega}) d \bm{\omega}
\end{equation}
for fixed $0<r < 4R$ and $\bm{\tilde{v}_m} \in \Phi^{-1}_{\bar{v}_m}(\{r^2 \})$. We introduce a parameter $0< \beta < 1$, which will be chosen later in terms of $\eta$.
\par
 Our aim is to change variables via the transformation $\mathcal{T}_{\bm{{v}_m}}$ defined in Lemma \eqref{impact dir change of variables}. To maintain a bounded Jacobian, however, this change of variables can not be done on the whole space. We define the two subspaces, 
\begin{align}
    \mathbb{E}^1 &= \{ \bm{\omega} \in \mathcal{E}^+_{\bm{{v}_m}} \, : \, b_{\ell +1}(\bm{\omega} , \bm{{v_m}'}) > \beta |\bm{{v}_m}'|   \}   \\
    \mathbb{E}^2 &= \{ \bm{\omega} \in \mathcal{E}^+_{\bm{{v}_m}} \, : \, b_{\ell +1}(\bm{\omega} , \bm{{v_m}'}) \leq \beta |\bm{v_m}'|  \}, \label{E2}
\end{align}
so that we have, $\mathcal{E}^+_{\bm{{v}_m}} = \mathbb{E}^1 \cup \mathbb{E}^2$.
\par
\subsubsection*{Estimate on $\mathbb{E}^2$ }
First, we will show that $|\mathbb{E}^2| \lesssim \beta$. We can rewrite $b_{\ell +1}(\bm{\omega} , \bm{{v_m}'})$ as
\begin{align}
    b_{\ell +1}(\bm{\omega} , \bm{{v_m}'}) &= \sum_{i = 1}^\ell \langle \omega_i, v_{m+i}-\bar{v}_m\rangle + \sum_{1\leq i < j \leq \ell} \langle \omega_i - \omega_j, v_{m+i} - v_{m+j} \rangle \\
    &= \sum_{i=1 }^\ell \langle \omega_i, \ell v_{m+i} - \bar{v}_m - \sum_{j\neq i} v_{m+j}\rangle \\
    &= \sum_{i=1}^\ell \langle \omega_i, \ell(v_{m+i} - \bar{v}_m) - \sum_{j\neq i}(v_{m+j} - \bar{v}_m) \rangle\\
    &= \langle \bm{\omega}, A \bm{v_m}' \rangle,
\end{align}
where $\bm{v_m}'$ is given in \eqref{v_m'} and $A \in \R^{\ell \times \ell}$, given by
\begin{equation}
    A = 
    \begin{pmatrix}
        \ell& -1 &\cdots & -1 \\
        -1 & \ddots &  & \vdots \\
        \vdots& & \ddots & -1 \\
        -1 & \cdots & -1 & \ell
    \end{pmatrix},
\end{equation}
is a strictly diagonally dominant matrix for $\ell \geq 1$ and thus invertible. We can now write \eqref{E2} as
\begin{align}
    \mathbb{E}^2 &= \{ \bm{\omega} \in \mathcal{E}^+_{\bm{v_m}} \, : \, 0< \langle \bm{\omega}, A \bm{v_m}' \rangle \leq \beta |\bm{v_m}'|  \}.
\end{align}
Fix a bijective linear map $T:\E \rightarrow \mathbb{S}^{\ell d - 1}_1$ as in \eqref{ellipsoid to sphere block}. Note that
\begin{align}
    T(\mathbb{E}^2) &= \{ \bm{\theta} \in T(\mathcal{E}^+_{\bm{v_m}}) \, : \,  0< \langle T^{-1} \bm{\theta}, A \bm{v_m}' \rangle \leq \beta |\bm{v_m}'|     \} \\
     &= \{ \bm{\theta} \in \mathbb{S}^{\ell d -1}_1 \, : \,  0< \langle  \bm{\theta},( T^{-1})^T A \bm{v_m}' \rangle \leq \beta |\bm{v_m}'|  \}\\
     &= \{ \bm{\theta} \in \mathbb{S}^{\ell d -1}_1 \, : \,  0< \langle  \bm{\theta},B \bm{v_m}' \rangle \leq \beta \frac{|\bm{v_m}'|}{|B \bm{v_m}'|} |B \bm{v_m}'|  \} \\
     &\subset \{ \bm{\theta} \in \mathbb{S}^{\ell d -1}_1 \, : \,  0< \langle  \bm{\theta},B \bm{v_m}' \rangle \leq \beta |B|_{op}^{-1} |B \bm{v_m}'|  \},
\end{align}
where $B := ( T^{-1})^T A$ is an invertible linear transformation. Notice that $T(\mathbb{E}^2)$ is the union of two unit $(\ell d -1)$-spherical caps of angle
\begin{align}
    \pi/2 - \arccos{\big(\beta |B|_{op}^{-1}\big)} &= \arcsin{\big(\beta |B|_{op}^{-1}\big)} \\
    &\leq \beta |B|_{op}^{-1} \\
    &\lesssim \beta.
\end{align}
Therefore, by applying a change of variables with respect to $T$, we have
\begin{align}
        \int_{\mathcal{E}^+_{\bm{v_m}}} \mathds{1}_{\mathbb{E}^2}(\bm{\omega}) d\bm{\omega} &\lesssim \int_{\mathbb{S}^{\ell d - 1}_1} \ind_{T(\mathbb{E}^2)} (\bm{\theta}) d \bm{\theta} \\
        &\lesssim \beta. \label{E2 est}
\end{align}
\par 
\subsubsection*{Estimate on $\mathbb{E}^1$ }
We now move our attention to $\mathbb{E}^1$. For all $1\leq i < j \leq \ell$ we apply the fact
\begin{equation}
    |v_{m+i} - v_{m+j}|^2 \leq |v_{m+i} - \bar{v}_m|^2 + |v_{m+j} - \bar{v}_m|^2 
\end{equation}
to derive the estimate, 
\begin{equation}\label{r v est}
    r^2 \leq \ell |\bm{v_m}'|^2,
\end{equation}
where we recall that $r>0$ is fixed such that $(v_{m+1}, \cdots, v_{m+\ell}) \in\Phi^{-1}(\{r^2 \})$.Therefore, recalling our calculation for the Jacobian of $\mathcal{T}_{\bm{{v}_m}}(\bm{{\omega}})$ given 
in \eqref{T Jac} along with \eqref{r v est}, for all $\bm{\omega}\in \mathbb{E}^1$ we have
\begin{align}
    \jac^{-1}(\mathcal{T}_{\bm{{v}_m}}(\bm{\omega})) &\lesssim r^{\ell d}  c^{-\ell d}(\bm{\omega}, \bm{{v}_m}) \\
    &\simeq r^{\ell d}  b_{\ell+1}^{-\ell d}(\bm{\omega}, \bm{{v_m}'}) \\
    &\leq r^{\ell d} \beta^{-\ell d} |\bm{v_m}'|^{-\ell d} \\
    &\lesssim \beta^{-\ell d}.
\end{align}
We now can move on to estimate the pathological sets on $\mathbb{E}^1$, starting with $U^{j,*}_m$
\begin{equation}
    U_m^{j,*} = \{(\omega_1,\cdots, \omega_\ell, v_{m+1}, \cdots, v_{m+\ell}) \in  \E \times B_R^{\ell d} \, : \, \bar{v}_m^* \in K_\eta^{d,j}\}, \quad 1\leq j < m.
\end{equation}
Therefore, by Fubini, \eqref{T g est} and recalling $S_1$ from Lemma \ref{transformation of E} we obtain
\begin{align}
    \int_{\mathbb{E}^1} \mathds{1}_{U_m^{j,*}} (\bm{\omega}) d\bm{\omega} &= \int_{\mathbb{E}^1} \mathds{1}_{K^{d,j}_{\frac{(\ell+1)\eta}{r}}\times \R^{(\ell-1) d}} \circ S_1 \circ \mathcal{T}_{\bm{{v}_m}} (\bm{\omega}) d\bm{\omega} \\
    &\lesssim \beta^{-\ell d} \int_{\E} \mathds{1}_{K^{d,j}_{\frac{(\ell+1)\eta}{r}}\times \R^{ (\ell-1) d}} \circ S_1 (\bm{\nu}) d\bm{\nu} \\
    &\lesssim \beta^{-\ell d} \int_{S_1(\E)} \mathds{1}_{K^{d,j}_{\frac{(\ell+1)\eta}{r}}\times \R^{(\ell-1) d}}(\bm{\theta}) d\bm{\theta}\\
    &\lesssim \beta^{-\ell d} \min \big\{1, \big(\frac{\eta}{r}\big)^{\frac{d-1}{2}} \big\}, \label{Um* est}
\end{align}
where \eqref{Um* est} follows from Lemma \ref{S est}.
By a permutation of variables this same calculation can be repeated for $U^{j,*}_{m+i}$ using the transformations $S_{i+1}$ for $1\leq i \leq \ell$ given in Lemma \ref{transformation of E}. Therefore, we have
\begin{equation}\label{Um+i* est}
    \int_{\mathbb{E}^1} \mathds{1}_{U_{m+i}^{j,*}} (\bm{\omega}) d\bm{\omega} \lesssim \beta^{-\ell d} \min \big\{1, \big(\frac{\eta}{r}\big)^{\frac{d-1}{2}} \big\}  , \quad 1\leq i \leq \ell, \quad 1\leq j < m.
\end{equation}
\par
We now move on to the collection of sets $\{V^*_i \}_{i=1}^\ell$. We first will deal with $V_1^*$. Recall, $V^*_1$ is given by
\begin{equation}
    V^*_{1} = \{(\omega_1,\cdots, \omega_\ell, v_{m+1}, \cdots, v_{m+\ell}) \in \mathbb{E}^{\ell d -1}_1 \times B_R^{\ell d} \, : \, |\bar{v}^*_{m} - v^*_{m+1}| \in B^d_\eta  \}.
\end{equation}
Let $\bm{\nu} = \mathcal{T}_{\bm{\bar{v}_m}}(\bm{\omega})$ and notice that
\begin{align}
    \bar{v}_m^* - v_{m+1}^* \in B^d_{\eta} \iff \nu_1 \in B^d_{\frac{\eta}{r}}.
\end{align}

By applying \eqref{T g est} and Lemma \ref{ellipsoidal est} we obtain the estimate 
\begin{align}
     \int_{\mathbb{E}^1} \mathds{1}_{V_1^*} (\bm{\omega}) d\bm{\omega} &= \int_{\mathbb{E}^1} \mathds{1}_{ B^d_{\frac{\eta}{r}} \times \R^{(\ell-1) d} } \circ \mathcal{T}_{\bm{{v}_m} } (\bm{\omega}) d\bm{\omega} \\
     &\lesssim \beta^{-\ell d} \int_{\E} \mathds{1}_{ B^d_{\frac{\eta}{r}} \times \R^{(\ell-1) d} }(\bm{\nu}) d\bm{\nu} \\
      &\lesssim \beta^{-\ell d} \min \big\{1, \big(\frac{\eta}{r}\big)^{\frac{d-1}{2}} \big\}.
\end{align}
By a simple permutation of coordinates we can derive an identical estimate for all $\{ V_i^*\}_{i=1}^\ell$
\begin{equation}\label{Vi* est}
    \int_{\mathbb{E}^1} \mathds{1}_{V_{i}^*} (\bm{\omega}) d\bm{\omega} \lesssim \beta^{-\ell d} \min \big\{1, \big(\frac{\eta}{r}\big)^{\frac{d-1}{2}} \big\}  , \quad i = 1,\cdots, \ell.
\end{equation}
\par
Finally, we consider the collection of sets $\{V_{i,j} \}_{1\leq i < j \leq \ell}$. We first will estimate $V^*_{1,2}$, given by 
\begin{equation}
    V^*_{1,2} = \{(\omega_1,\cdots, \omega_\ell, v_{m+1}, \cdots, v_{m+\ell}) \in \mathbb{E}^{\ell d -1}_1 \times B_R^{\ell d} \, : \, |{v}^*_{m+1} - v^*_{m+2}| \in B^d_\eta  \}.
\end{equation}
By calculating $v_{m+1}^* - v_{m+2}^*  $ we can show that
\begin{equation}
    v_{m+1}^* - v^*_{m+2} \in B^d_\eta \iff -\nu_1 + \nu_2 \in B^d_{\frac{\eta}{ r}},
\end{equation}
and thus $(\nu_1,\nu_2) \in W^{2d}_{\frac{\eta}{ r},1,1}$ which is defined in \eqref{W strip}.
Estimating \eqref{U star general est} for $V_{1,2}^*$ by applying \eqref{T g est} and Lemma \ref{W strip est} we obtain the estimate
\begin{align}
     \int_{\mathbb{E}^1} \mathds{1}_{V_{1,2}^*} (\bm{\omega}) d\bm{\omega} &= \int_{\mathbb{E}^1} \mathds{1}_{ W^{2d}_{\frac{\eta}{ r},1,1} \times \R^{(\ell -2)d}   } \circ \mathcal{T}_{\bm{{v}_m} } (\bm{\omega}) d\bm{\omega} \\
     &\lesssim \beta^{-\ell d}  \int_{\E} \mathds{1}_{ W^{2d}_{\frac{\eta}{ r},1,1} \times \R^{(\ell -2)d} }(\bm{\nu}) d\bm{\nu} \\
      &\lesssim \beta^{-\ell d} \min \big\{1, \big(\frac{\eta}{r}\big)^{\frac{d-1}{2}} \big\}.
\end{align}
Again, by a simple permutation of variables we obtain an identical estimate for all $\{ V^* _{i,j}\}_{1\leq i < j \leq \ell}$
\begin{equation}\label{Vij* est}
    \int_{\mathbb{E}^1} \mathds{1}_{V_{i,j}^*} (\bm{\omega}) d\bm{\omega} \lesssim \beta^{-\ell d} \min \big\{1, \big(\frac{\eta}{r}\big)^{\frac{d-1}{2}} \big\}  , \quad i,j = 1,\cdots, \ell \quad i<j.
\end{equation}
Therefore, combining \eqref{E2 est}, \eqref{Um* est}, \eqref{Um+i* est}, \eqref{Vi* est}, and \eqref{Vij* est} we obtain
\begin{align}
    |\mathcal{U^*} \cup \mathcal{V}^*| &\lesssim \int_0^{4R} \int_{\Phi^{-1}_{\bar{v}_m}(\{r^2 \})}  \big(  \beta  + m \beta^{-\ell d} \min \big\{1, \big(\frac{\eta}{r}\big)^{\frac{d-1}{2}} \big\}    \big) d\sigma({\bm{\tilde{v}_m}}) dr \\
    &\lesssim \int_0^{4R} r^{\ell d -1} \big(  \beta  + m \beta^{-\ell d} \min \big\{1, \big(\frac{\eta}{r}\big)^{\frac{d-1}{2}} \big\}    \big)  dr \\
    &\lesssim m R^{\ell d} \big(  \beta  +  \beta^{-\ell d} \eta^{\frac{d-1}{2}}   \big).
\end{align}
Choosing $\beta = \eta^{\frac{d-1}{2\ell d + 2 }} << 1$, for $d\geq 2$ we obtain
\begin{equation}
    |\mathcal{U^*} \cup \mathcal{V}^*| \lesssim m R^{\ell d} \eta^{\frac{d-1}{2\ell d + 2 }}.
\end{equation}

\par
\subsubsection{Estimate of $\mathcal{A^*}$}
As opposed to the previous estimates, we estimate the sets of the form $\{A^*_{i}\}_{i=1}^\ell$ and $\{ A^*_{i,j}\}_{1\leq i < j \leq \ell}$ by a change of variables on velocities instead of impact parameters.
\par
\textit{Estimate of $\{A^*_{i}\}_{i=1}^\ell$.} We proceed by proving an estimate on $A_1^*$ and argue that a similar proof can be done for all $\{A^*_{i}\}_{i=1}^\ell$. Recall that,
\begin{equation}
A_{1}^* = \{(\omega_1,\cdots, \omega_\ell, v_{m+1}, \cdots, v_{m+\ell}) \in \mathbb{E}^{\ell d -1}_1 \times B_R^{\ell d} \, :\, |\langle \omega_{1}, v^*_{m+1} - \bar{v}^*_{m} \rangle| \geq \gamma' |\omega_{1}||v^*_{m+ 1} - \bar{v}^*_{m}| \}.
\end{equation}
We define the smooth map $F^1$ which we use to change variables, 
\begin{equation}
    F^1(v_{m+1}) = v_{m+1}^* - \bar{v}_m^* = v_{m+1} - \bar{v}_m - c(\bm{\omega}, \bm{v_m}) (\ell +1 ) \omega_1. 
\end{equation}
We calculate, 
\begin{align}
    \frac{\partial F^1}{\partial v_{m+1}} = I_d - (\ell+1) \omega_1 \nabla^T_{v_{m+1}}c(\bm{\omega}, \bm{v_m}).
\end{align}
By recalling the definition of $c(\bm{\omega}, \bm{v_m})$ given in \eqref{impact param}, we have,
\begin{equation}\label{nabla c}
    \nabla^T_{v_{m+1}}c(\bm{\omega}, \bm{v_m}) = \frac{2}{\ell+1} (\ell \omega_1^T - \omega_2^T - \cdots - \omega_\ell^T).
\end{equation}
Therefore,
\begin{equation}
     \frac{\partial F^1}{\partial v_{m+1}} =I_d - 2 \omega_1(\ell \omega_1^T - \omega_2^T - \cdots - \omega_\ell^T).
\end{equation}
By applying \eqref{app det} we can derive the determinant,
\begin{equation} \label{F jacobian}
    \jac(F^1(v_{m+1})) = 1 - 2 \big( \ell|\omega_1|^2 - \langle \omega_1, \omega_2 \rangle - \cdots \langle \omega_1, \omega_\ell \rangle \big).
\end{equation}
Now, we will show that $F^1$ is injective. Consider $v_{m+1}, \xi_{m+1} \in B^d_R$ such that,
\begin{equation}
    F^1(v_{m+1}) = F^1(\xi_{m+1}).
\end{equation}
Then, we have,
\begin{equation}
    v_{m+1} - \xi_{m+1} = (\ell+1) \big(c(\bm{\omega}, \bm{v_m}) -c(\bm{\omega}, \bm{\xi})\big) \omega_1,
\end{equation}
where,
\begin{align}
    c(\bm{\omega}, \bm{v_m}) -c(\bm{\omega}, \bm{\xi}) = \frac{2}{\ell+1} \big( \langle v_{m+1} - \xi_{m+1}, \ell \omega_1 - \omega_2 - \cdots - \omega_\ell \rangle   \big).
\end{align}
Therefore,
\begin{equation}\label{v minus xi}
    v_{m+1} - \xi_{m+1} = 2  \big( \langle v_{m+1} - \xi_{m+1}, \ell \omega_1 - \omega_2 - \cdots - \omega_\ell \rangle   \big) \omega_1.
\end{equation}
Note, that for $\omega_1 = 0$ injectivity is trivial. Assuming that $\omega_1 \neq 0$, we use the substitution $v_{m+1} - \xi_{m+1} = \lambda \omega_1$ into \eqref{v minus xi} to calculate,
\begin{equation}
    \lambda (1- 2 \langle \omega_1, \ell \omega_1 - \omega_2 - \cdots - \omega_\ell \rangle ) = 0.
\end{equation}
For $|(1- 2 \langle \omega_1, \ell \omega_1 - \omega_2 - \cdots - \omega_\ell \rangle )| \neq 0$ then we must have that $\lambda = 0 $ and $F^1$ is injective. 
\par
Additionally, since we have that $|\omega_i| \leq 1$ and $|v_{m+i}| \leq R$ for  all $i=1,\cdots, \ell$, we have
\begin{align}
    |F^1(v_{m+1})| &= |v_{m+1}| + |\bar{v}_m| + 2|\omega_1| \bigg(  \sum_{i=1}^\ell |\omega_i| (|v_{m+i}| + |\bar{v}_m|) + \sum_{1\leq i < j \leq \ell} (|\omega_i| + |\omega_j|) (|v_{m+i}| + |v_{m+j}|)  \bigg) \\
    &\leq \bigg(2 + 4\ell  + 8 {\ell \choose 2}\bigg)R \\
    &= (4\ell^2 +2)R.
\end{align}
Therefore, 
\begin{equation}
    F^1(B^d_R) \subset B^d_{(4\ell^2 +2)R}.
\end{equation}
By Fubini, we have the estimate,
\begin{equation}
    |A^*_{1}| \leq \int_{\E \times B^{(\ell -1)d}_R} \int_{B^d_R} \mathds{1}_{\tilde{A}_1} (v_{m+1}) dv_{m+1} \cdots dv_{m+\ell} d\bm{\omega},
\end{equation}
where $\tilde{A}_1$ is defined,
\begin{equation}
    \tilde{A}_1 = \{ v_{m+1} \in B^d_R \, : \, (\bm{\omega}, v_{m+1}, \cdots, v_{m+\ell}) \in A^*_{1} \}.
\end{equation}
Define the sets,
\begin{align}
    J_1 &= \{ (\omega_1,\cdots, \omega_\ell) \in \E \, : \,  \big| \jac(F^1(v_{m+1})) \big| \leq  \ell \beta \} \\
    &= \{ (\omega_1,\cdots, \omega_\ell) \in \E \, : \,  \big| 1 - 2 \big( \ell|\omega_1|^2 - \langle \omega_1, \omega_2 \rangle - \cdots - \langle \omega_1, \omega_\ell \rangle \big) \big| \leq  \ell \beta \}.
\end{align}
and,
\begin{align}
    I_1 &=  \int_{(\E \cap J_1)\times B^{(\ell -1)d}_R} \int_{B^d_R} \mathds{1}_{\tilde{A}_1} (v_{m+1}) dv_{m+1} \cdots dv_{m+\ell} d\bm{\omega} \\
     I'_1 &=\int_{(\E \setminus J_1 ) \times B^{(\ell -1)d}_R} \int_{B^d_R} \mathds{1}_{\tilde{A}_1} (v_{m+1}) dv_{m+1} \cdots dv_{m+\ell} d\bm{\omega}.
\end{align}
so that we can write, 
\begin{equation}
    |A^*_{1}| \leq I_1 + I'_1
\end{equation}
By Lemma \ref{annuli est} we have that,
\begin{equation}\label{I1 est}
    I_1 \lesssim R^{\ell d} \beta.
\end{equation}

Now, for $I'_1$, by recalling \eqref{F jacobian} we have that,
\begin{equation}
    |\jac \big(F^1 (v_{m+1})\big)|^{-1} \lesssim \beta^{-1}.
\end{equation}

Additionally, note that
\begin{equation}
    v_{m+1} \in \tilde{A}_1 \iff F^1(v_{m+1}) \in U_{\omega_1}
\end{equation}
where $U_{\omega_1}$ is given by,
\begin{equation} \label{U def}
    U_{\omega_1} = \{ \nu \in \R^d   \, : \, \langle \omega_1, \nu \rangle \geq \gamma' |\omega_1||\nu|  \}.
\end{equation}
Therefore, by making this substitution we calculate,
\begin{equation}
    \int_{B^d_R} \mathds{1}_{\tilde{A}_1}(v_{m+1}) dv_{m+1} = \int_{B^d_R} \mathds{1}_{U_{\omega_1}} (F^1(v_{m+1})) dv_{m+1} \lesssim \beta^{-1} \int_{ B^d_{(4\ell^2 +2)R}} \mathds{1}_{U_{\omega_1}}(\nu) d\nu.
\end{equation}
Recalling Lemma \ref{conic est} we have,
\begin{equation}
    \mathds{1}_{U_{\omega_1}}(\nu) = \ind_{S(\gamma', \nu)}(\omega_1), \quad \forall \omega_1 \in B^d_1, \quad \forall \nu \in B^d_{(4\ell^2 +2)R}.
\end{equation}
Therefore, by Fubini and Lemma \ref{conic est} we obtain,
\begin{align}
    I_1' & \leq \beta^{-1} \int_{(\E \setminus J_1) \times B_R^{(\ell-1)d}} \int_{B^d_{(4\ell^2 +2)R}} \ind_{U_{\omega_1}}(\nu) d\nu d\omega_1 \cdots d\omega_\ell dv_{m+2} \cdots dv_{m+\ell} \\
    &\leq \beta^{-1} \int_{B^d_{(4\ell^2 +2)R} \times B^{(\ell-1)d}_R} \int_{\E} \ind_{S(\gamma', \nu)}(\omega_1) d\omega_1 \cdots d\omega_\ell d\nu dv_{m+2} \cdots dv_{m+\ell} \\ 
    &\lesssim R^{\ell d} \beta^{-1} \arccos{\gamma'} \\
    &= R^{\ell d} \beta^{-1} \arccos{\sqrt{1-\frac{\gamma}{2}}}.
\end{align}
By a permutation of coordinates we can obtain an identical calculation for all $\{ A_i^*\}_{i=1}^\ell$. By recalling \eqref{I1 est}, we obtain the estimate
\begin{equation}\label{Ai* est}
    |A_i^*| \lesssim R^{\ell d} \big(\beta + \beta^{-1} \arccos{\sqrt{1- \frac{\gamma}{2}}}\big), \quad 1\leq i \leq \ell.
\end{equation}
\par
\textit{Estimate of $\{A^*_{i,j}\}_{1\leq i < j \leq \ell }$.}
Like in the previous proof, we will present the estimate only for $A^*_{1,2}$ and argue that the proof for the general case follows in the same way. Recall,
\begin{equation}
A_{1,2}^* = \{(\omega_1,\cdots, \omega_\ell, v_{m+1}, \cdots, v_{m+\ell}) \in \mathbb{E}^{\ell d -1}_1 \times B_R^{\ell d} \, :\, |\langle \omega_{2} - \omega_1, v^*_{m+1} - v^*_{m+2} \rangle| \geq \gamma' |\omega_{2} - \omega_{1}||v^*_{m+ 1} - v^*_{m+2}| \}.
\end{equation}
We define the function,
\begin{equation}
    F^{1,2} (v_{m+2}) = v_{m+1}^* - v_{m+2}^* = v_{m+2} - v_{m+1} - (\ell+1) c(\bm{\omega},\bm{v_m}) (\omega_2 - \omega_1). 
\end{equation}
By recalling \eqref{nabla c} we calculate,
\begin{align}
    \frac{\partial F^{1,2}}{\partial v_{m+2}} (v_{m+2}) &= I_d - (\ell + 1) (\omega_2 - \omega_1) \nabla_{v_{m+2}}^T c(\bm{\omega},\bm{v_m}) \\
    &= I_d - 2 (\omega_2 - \omega_1 ) ( - \omega_1 + \ell \omega_2 - \omega_3 - \cdots - \omega_\ell).
\end{align}

By applying Lemma \ref{app det} we calculated the Jacobian,
\begin{align} \label{F12 Jac}
    \jac(F^{1,2}(v_{m+2})) = 1- 2 \big( \omega_1^2 + \ell \omega_2^2 - \langle \omega_2, \omega_1\rangle - \langle \omega_2, \omega_3 \rangle - \cdots \langle \omega_2, \omega_\ell \rangle + \langle \omega_1, \omega_3\rangle + \cdots + \langle \omega_1, \omega_\ell \rangle \big).
\end{align}
Now, we will show that $F^{1,2}$ is injective. Consider $v_{m+2}, \xi_{m+2} \in B^d_R$ such that,
\begin{equation}
    F^{1,2}(v_{m+2}) = F^{1,2}(\xi_{m+2}).
\end{equation}
Then, we have,
\begin{equation}
    v_{m+2} - \xi_{m+2} = (\ell+1) \big(c(\bm{\omega},\bm{v_m}) -c(\bm{\omega},\bm{\xi})\big) (\omega_2 - \omega_1),
\end{equation}
where,
\begin{align}
    c(\bm{\omega},\bm{v_m}) -c(\bm{\omega},\bm{\xi}) = \frac{2}{\ell+1} \big( \langle v_{m+2} - \xi_{m+2}, - \omega_1 + \ell \omega_2 - \omega_3 -\cdots - \omega_\ell \rangle   \big).
\end{align}
Therefore,
\begin{equation}\label{v minus xi 2}
    v_{m+2} - \xi_{m+2} = 2  \big( \langle v_{m+2} - \xi_{m+2}, - \omega_1 +\ell  \omega_2 - \omega_3 - \cdots - \omega_\ell \rangle   \big) (\omega_2 - \omega_1).
\end{equation}
Note, that for $\omega_2 - \omega_1 = 0$ injectivity is trivial. Assuming that $\omega_2 - \omega_1 \neq 0$,
we use the substitution $v_{m+2} - \xi_{m+2} = \lambda ( \omega_2 - \omega_1)$ into \eqref{v minus xi 2} to arrive at,
\begin{equation}
    \lambda (  1- 2 \big( \omega_1^2 + \ell \omega_2^2 - \langle \omega_2, \omega_1\rangle - \langle \omega_2, \omega_3 \rangle - \cdots \langle \omega_2, \omega_\ell \rangle + \langle \omega_1, \omega_3\rangle + \cdots + \langle \omega_1, \omega_\ell \rangle \big) ) = \lambda\big(\jac ( F^{1,2}(v_{m+2}))\big) = 0.
\end{equation}
For $|\jac (F^{1,2}(v_{m+2}))| \neq 0$ then we must have that $\lambda = 0 $ and $F^{1,2}$ is injective. 
\par 

Additionally, since we have that $|\omega_i| \leq 1$ and $|v_{m+i}| \leq R$ for  all $i=1,\cdots, \ell$, we have
\begin{align}
    |F^{1,2}(v_{m+2})| &= |v_{m+2}| + |v_{m+1}| + 2|\omega_2 - \omega_1| \bigg(  \sum_{i=1}^\ell |\omega_i|(|v_{m+i}| + |\bar{v}_m|) + \sum_{1\leq i < j \leq \ell} (|\omega_i| + |\omega_j|) (|v_{m+i}| + |v_{m+j}|)  \bigg) \\
    &\leq \bigg(2 + 8\ell  + 16 {\ell \choose 2}\bigg)R \\
    &= (8\ell^2 +2)R.
\end{align}
Therefore, 
\begin{equation}
    F^{1,2}(B^d_R) \subset B^d_{(8\ell^2 +2)R}.
\end{equation}
By Fubini, we have the estimate,
\begin{equation}
    |A^*_{1, 2}| \leq \int_{\E \times B^{(\ell -1)d}_R} \int_{B^d_R} \mathds{1}_{\tilde{A}_{1,2}} (v_{m+2}) dv_{m+2} dv_{m+1} dv_{m+3} \cdots dv_{m+\ell} d\bm{\omega},
\end{equation}
where $\tilde{A}_{1,2}$ is defined,
\begin{equation}
     \tilde{A}_{1,2} = \{ v_{m+2} \in B^d_R \, : \, (\bm{\omega}, v_{m+1}, \cdots, v_{m+\ell}) \in A^*_{1,2} \}.
\end{equation}
We define the $J_{1,2}$ by,
\begin{equation}
    J_{1,2} = \{ (\omega_1,\cdots, \omega_\ell) \in \E \, : \,  \big| Jac(F^{1,2}) \big| \leq  \ell \beta \}.
\end{equation}
where $Jac(F^{1,2})$ is given in \eqref{F12 Jac}.
We also define the sets,
\begin{align}
    I_{1,2} &=  \int_{(\E \cap J_{1,2})\times B^{(\ell -1)d}_R} \int_{B^d_R} \mathds{1}_{\tilde{A}_{1,2} } (v_{m+1}) dv_{m+1} \cdots dv_{m+\ell} d\bm{\omega} \\
     I_{1,2}' &=\int_{(\E \setminus J_{1,2} ) \times B^{(\ell -1)d}_R} \int_{B^d_R} \mathds{1}_{\tilde{A}_{1,2} } (v_{m+1}) dv_{m+1} \cdots dv_{m+\ell} d\bm{\omega},
\end{align}
so that we can write,
\begin{equation}
    |A^*_{1,2}| \leq I_{1,2} + I_{1,2}'.
\end{equation}

By Lemma \ref{annuli est} we have that,
\begin{equation}\label{I12 est}
    I_{1,2} \lesssim R^{{\ell d}} \beta.
\end{equation}

Now, for $I_{1,2}'$, by recalling \eqref{F12 Jac} we have that,
\begin{equation}
    |\jac \big(F^{1,2} (v_{m+2})\big)|^{-1} \lesssim \beta^{-1}.
\end{equation}
Additionally, note that
\begin{equation}
    v_{m+2} \in \tilde{A}_{1,2} \iff F^{1,2}(v_{m+2}) \in U_{\omega_1,\omega_2}
\end{equation}
where $U_{\omega_1,\omega_2}$ is given by
\begin{equation}
    U_{\omega_1,\omega_2} = \{ \nu\in\R^d \, : \, \langle \omega_1 - \omega_2, \nu \rangle \geq \gamma' |\omega_1 - \omega_2||\nu| \}.
\end{equation}
Therefore, by making this substitution we calculate,
\begin{equation}
    \int_{B^d_R} \mathds{1}_{\tilde{A}_{1,2}}(v_{m+2}) dv_{m+2} = \int_{B^d_R} \mathds{1}_{U_{\omega_1,\omega_2}} (F^{1,2}(v_{m+2})) dv_{m+2} \lesssim \beta^{-1} \int_{ B^d_{(8\ell^2 +2)R}} \mathds{1}_{U_{\omega_1,\omega_2}}(\nu) d\nu.
\end{equation}
Recalling Lemma \ref{conic est N} we have
\begin{equation}
    \ind_{U_{\omega_1,\omega_2}}(\nu) = \ind_{N(\gamma', \nu)}(\omega_1,\omega_2), \quad \forall\omega_1, \omega_2 \in B_1^d,\quad \forall \nu \in B^d_{(8\ell^2 + 2)R}.
\end{equation}
Therefore, by Fubini and Lemma \ref{conic est N} we obtain,
\begin{align}
    I_{1,2}' & \leq \beta^{-1} \int_{(\E \setminus J_1) \times B_R^{(\ell-1)d}} \int_{B^d_{(8\ell^2 +2)R}} \ind_{U_{\omega_1,\omega_2}}(\nu) d\nu d\omega_1 \cdots d\omega_\ell dv_{m+1} dv_{m+3}\cdots dv_{m+\ell} \\
    &\leq \beta^{-1} \int_{B^d_{(8\ell^2 +2)R} \times B^{(\ell-1)d}_R} \int_{\E} \ind_{N(\gamma', \nu)}(\omega_1,\omega_2) d\omega_1 \cdots d\omega_\ell d\nu dv_{m+1} dv_{m+3} \cdots dv_{m+\ell} \\ 
    &\lesssim R^{\ell d} \beta^{-1} \arccos{\gamma'} \\
    &= R^{\ell d} \beta^{-1} \arccos{\sqrt{1-\frac{\gamma}{2}}}.
\end{align}
By a permutation of coordinates we can obtain identical estimates for all $\{A^*_{i,j} \}_{1\leq i < j \leq \ell}$. Therefore, recalling \eqref{I12 est}, we obtain the estimate
\begin{equation}\label{Aij* est}
    |A_{i,j}^*| \lesssim R^{\ell d} \big(\beta + \beta^{-1} \arccos{\sqrt{1- \frac{\gamma}{2}}}\big), \quad 1\leq i < j \leq \ell.
\end{equation}
By combining  \eqref{Ai* est} and \eqref{Aij* est} we obtain
\begin{align}
    |\mathcal{A^*}| \lesssim R^{\ell d} \bigg(\beta  + \beta^{-1} \arccos{\sqrt{1- \frac{\gamma}{2}}}\bigg).
\end{align}
Recall from \eqref{arccos gamma eta} that we have
\begin{equation}
    \arccos \sqrt{1- \frac{\gamma}{2}} < \eta.
\end{equation}
We select
\begin{equation}
    \beta = \eta^{1/2} << 1,
\end{equation}
to obtain the estimate
\begin{align}
    |\mathcal{A^*}| \lesssim R^{\ell d} \eta^{1/2} \lesssim m R^{\ell d} \eta^{\frac{d-1}{2\ell d +2}},
\end{align}
for $d\geq 2$. 
\par
Therefore, we obtain the estimate
\begin{equation}
    |\mathcal{B}^\ell_k(\bar{Z}_m)| \lesssim m R^{\ell d} \eta^{\frac{d-1}{2 \ell d + 2}},
\end{equation}
thus completing the proof of Proposition \ref{bad set triple}.

\section{Elimination of recollisions}\label{sec::reco}

In this section, inspired by the work done in \cite{Gallagher}, by the implementation of pseudo-trajectories as carried out in the binary-ternary case given in \cite{AmpatzoglouPavlovic2020}, we reduce the convergence proof to comparing truncated elementary observables. The convergence proof will then be completed in Section \ref{sec::convergence}.
\subsection{Restriction to good configurations}
For $M \in \N$ and $m\in \N$ define the set
\begin{equation}\label{good set intersection}
    G_m(\epsilon_{M+1}, \epsilon_0, \delta) := G_m(\epsilon_{M + 1}, 0) \cap G_m({\epsilon_0, \delta}),
\end{equation}
where $G_m(\theta, t_0)$ is defined in \eqref{good conf def}.
\begin{lemma}\label{velocity subset lemma}
    Let $s \in \N$ and $\alpha, \epsilon_0, R, \eta, \delta$ be parameters as in \eqref{choice of parameters} and $\epsilon_2 << \cdots << \epsilon_{M+1} << \alpha$. Then for any $X_s \in \Delta_s^X(\epsilon_0)$, there is a subset of velocities $\mathcal{M}_s(X_s) \subset B^{ds}_R$ of measure
    \begin{equation}
        |\mathcal{M}_s(X_s)|_{ds} \leq C_{d,s} R^{ds} \eta^{\frac{d-1}{2}},
    \end{equation}
    such that
    \begin{equation}
        Z_s \in G_s(\epsilon_{M+1}, \epsilon_0, \delta), \quad \forall V_s \in B_R^{ds} \setminus \mathcal{M}_s(X_s).
    \end{equation}
\end{lemma}
\begin{proof} 
    We apply Proposition 11.2 from \cite{AmpatzoglouThesis} for $\epsilon = \epsilon_{M+ 1}$
\end{proof}
We denote $\mathcal{M}_s^c(X_s) = B_R^{ds} \setminus \mathcal{M}_s(X_s)$. For $1 \leq k \leq n$, recall the observables $I^N_{s,k,R, \delta}, I^{\infty}_{s,k,R,\delta}$ defined in \eqref{bbgky truncated time} - \eqref{boltzmann truncated time}. By restricting the domain of integration to velocities giving good configurations we define
\begin{align}\label{bbgky obs}
    \widetilde{I}^N_{s,k,R, \delta}(t,X_s) &:=
\int_{\mathcal{M}_s^c(X_s)}\phi_s(V_s)f_{N,R,\delta}^{(s,k)}(t,X_s, V_s)\,dV_s,  \\
    \widetilde{I}^{\infty}_{s,k,R,\delta}(t, X_s) &:=
\int_{\mathcal{M}_s^c(X_s)}\phi_s(V_s)f_{R,\delta}^{(s,k)}(t,X_s, V_s)\,dV_s \label{boltz obs}
\end{align}

We recall the notation given in \eqref{S_k} and \eqref{sigma tilde}, for $k \in \N$,
\begin{equation*}
S_k:=\left\{\sigma=(\sigma_1,...,\sigma_k):\sigma_i\in\left\{1,\cdots, M\right\},\quad\forall i=1,...,k\right\}.
\end{equation*}
and given $\sigma\in S_k$, for any $1\leq j \leq k$ we write
\begin{equation}
\widetilde{\sigma}_j=\sum_{i=1}^j \sigma_i.
\end{equation}
We also recall the well separated configurations $\Delta_s(\theta)$ given in \eqref{well sep conf}.
\begin{lemma} \label{I M diff}
    Let $s,n \in \N$, $\ell \in \{1, \cdots, M \}$, and $\alpha, \epsilon_0, R, \eta, \delta$ be parameters as in \eqref{choice of parameters} and $\epsilon_2 << \cdots << \epsilon_{M+1} << \alpha$ satisfying the scaling given in \eqref{scaling}, and $t \in [0,T]$. Then, the following estimates hold:
    \begin{align}
        \sum_{k=1}^n ||I^N_{s,k,R,\delta}(t) - \widetilde{I}^N_{s,k,R,\delta} (t)||_{L^\infty(\Delta_s^X(\epsilon_0))} &\leq C_{d,s, \mu_0,T} R^{ds} \eta^{\frac{d-1}{2}} ||F_{N,0}||_{N,\beta_0,\mu_0}, \\
        \sum_{k=1}^n ||I^\infty_{s,k,R,\delta}(t) - \widetilde{I}^\infty_{s,k,R,\delta} (t)||_{L^\infty(\Delta_s^X(\epsilon_0))} &\leq C_{d,s, \mu_0,T} R^{ds} \eta^{\frac{d-1}{2}} ||F_{0}||_{\infty,\beta_0,\mu_0}.
    \end{align}
\end{lemma}

\begin{proof}
    The result is obtained by a repeated application of the estimates in Theorem \ref{well posedness BBGKY}, Proposition \ref{remark for initial}, and Lemma \ref{velocity subset lemma}.
\end{proof}

\begin{remark}\label{obs k zero}
For $s \in \N$ and $X_s \in \Delta_s^X(\epsilon_0)$, by the construction of $\mathcal{M}_s(X_s)$ we have that
\begin{equation}
    \widetilde{I}^N_{s,0,R, \delta}(t,X_s) = \widetilde{I}^\infty_{s,0,R, \delta}(t,X_s).
\end{equation}
Therefore, by applying Lemma \ref{I M diff} to achieve convergence we must only bound the $\widetilde{I}^N_{s,k,R, \delta}(t,X_s) - \widetilde{I}^\infty_{s,k,R, \delta}(t,X_s)$ terms for $k = 1, \cdots, n$. 
\end{remark}

\subsection{Reduction to elementary observables}

In this subsection, we express the observables $\widetilde{I}^N_{s,k,R, \delta}(t)$, and $\widetilde{I}^\infty_{s,k,R, \delta}(t)$ in terms of its elementary parts. 
\par
For $s \in \N$ and $\ell = 1, \cdots, M$, we recall the truncated BBGKY collisional operators $\mathcal{C}^{N,R}_{s, s+\ell}$ given in \eqref{velocity truncation of operators} and we write them in terms of their gain and loss terms:

\begin{equation}
    \mathcal{C}^{N,R}_{s, s+\ell} = \sum_{i = 1 }^s \mathcal{C}^{N,R,+, i}_{s, s+\ell} - \sum_{i=1}^s \mathcal{C}^{N,R,-,i}_{s, s+\ell},
\end{equation}
where
\begin{align}\label{C + i}
    \mathcal{C}^{N,R,+, i}_{s, s+\ell} g_{s + \ell}(Z_s) &= A^\ell_{N,\epsilon_{\ell+1}, s}  \int_{\mathbb{S}_1^{\ell d-1}\times B_R^{\ell d}}  b_+(\bm{\omega}, v_{s+1}-v_{i}, \cdots,v_{s+\ell}- v_{i})  g_{s+\ell}(Z^{i,*,\ell+1}_{s+\ell,\epsilon_{\ell+1}}) d\bm{\omega} d v_{s+1} \cdots d v_{s+\ell} \\
    \mathcal{C}^{N,R,-, i}_{s, s+\ell}g_{s + \ell}(Z_s) &= A^\ell_{N,\epsilon_{\ell+1}, s}   \int_{\mathbb{S}_1^{\ell d-1}\times B_R^{\ell d}}  b_+(\bm{\omega}, v_{s+1}-v_{i}, \cdots,v_{s+\ell}- v_{i})  g_{s+\ell}(Z^{i}_{s+\ell,\epsilon_{\ell+1}}) d\bm{\omega} d v_{s+1} \cdots d v_{s+\ell},\label{C - i}
\end{align}
where $Z^{i,*,\ell+1}_{s+\ell,\sigma_{\ell+1}}$ and $Z^{i}_{s+\ell,\sigma_{\ell+1}}$ are defined in \eqref{Zi defn}. We similarly define the truncated Boltzmann hierarchy operators, $C^{\infty, R}_{s,s+\ell}$:
\begin{equation}
    \mathcal{C}^{\infty,R}_{s, s+\ell} = \sum_{i = 1 }^s \mathcal{C}^{\infty,R,+, i}_{s, s+\ell} - \sum_{i=1}^s \mathcal{C}^{\infty,R,-,i}_{s, s+\ell},
\end{equation}
where
\begin{align}\label{C inf + i}
    \mathcal{C}^{\infty,R,+, i}_{s, s+\ell}g_{s + \ell}(Z_s) &= \frac{1}{\ell!}   \int_{\mathbb{S}_1^{\ell d-1}\times B_R^{\ell d}}  b_+(\bm{\omega}, v_{s+1}-v_{i}, \cdots,v_{s+\ell}- v_{i})  g_{s + \ell}(Z^{i,*,\ell+1}_{s+\ell}) d\bm{\omega} d v_{s+1} \cdots d v_{s+\ell} \\
    \mathcal{C}^{\infty,R,-, i}_{s, s+\ell}g_{s + \ell}(Z_s) &=  \frac{1}{\ell!}  \int_{\mathbb{S}_1^{\ell d-1}\times B_R^{\ell d}}  b_+(\bm{\omega}, v_{s+1}-v_{i}, \cdots,v_{s+\ell}- v_{i})  g_{s + \ell}(Z^{i}_{s+\ell}) d\bm{\omega} d v_{s+1} \cdots d v_{s+\ell},\label{C inf - i}
\end{align}
where $Z^{i,*,\ell+1}_{s+\ell}$ and $Z^{i}_{s+\ell}$ are defined in \eqref{Z inf hier}.
We introduce the following notation which will aid in keeping track of all possible particle adjunctions:
\begin{align}
    \mathcal{M}_{s,k, \sigma} &= \{ \tilde{M} = (m_1, \cdots, m_k) \in \N^k, \, : \, m_i \in \{1, \cdots, s+\tilde{\sigma}_{i-1}\}, \quad \forall i \in \{ 1,\cdots, k \} \}, \\
    \mathcal{J}_{s,k, \sigma} &= \{ J = (j_1,\cdots, j_k)\in \N^k \, :\, j_i\in \{-1,1\}, \quad \forall i \in \{1,\cdots, k \} \}, \\
    \mathcal{U}_{s, k, \sigma} &= \mathcal{J}_{s,k,\sigma} \times \mathcal{M}_{s,k,\sigma}.
\end{align}
We can express the BBGKY hierarchy observable functional $\tilde{I}^N_{s,k,R,\delta}$ in terms of elementary observables $\tilde{I}^N_{s,k,R,\delta, \sigma}$:
\begin{align}
    \tilde{I}^N_{s,k,R,\delta} &= \sum_{\sigma \in S_{k}} \sum_{(J,\tilde{M})\in \mathcal{U}_{s,k, \sigma}} \bigg( \prod_{i=1}^k j_i \bigg) \tilde{I}^N_{s,k,R,\delta, \sigma}(t,J,\tilde{M})(X_s), \\
    \tilde{I}^N_{s,k,R,\delta, \sigma}(t,J,\tilde{M})(X_s) &= \int_{\mathcal{M}_s^c(X_s)} \phi_s(V_s) \int_{\mathcal{T}_{k,\delta}(t)} T_s^{t - t_1} C^{N,R,j_1,m_1}_{s,s+\tilde{\sigma}_1} \cdots&
    \\
    &\cdots T^{t_{k-1}- t_k}_{s+{\tilde{\sigma}_{k-1}}} C^{N,R,j_k,m_k}_{s+\tilde{\sigma}_{k-1}, s+\tilde{\sigma}_k} T^{t_k}_{s+\tilde{\sigma}_k} f_{N,0}^{(s+\tilde{\sigma}_k)}(Z_s) \,dt_k \cdots dt_1 dV_s,
\end{align}
where $C^{N,R,j_i,m_i}_{s+\tilde{\sigma}_{i-1}}$ are defined in \eqref{C + i}-\eqref{C - i}.
We do the same for the Boltzmann hierarchy observable functional while having in mind the definitions \eqref{C inf + i}-\eqref{C inf - i} 

\begin{align}
    \tilde{I}^\infty_{s,k,R,\delta} &= \sum_{\sigma \in S_{k}} \sum_{(J,\tilde{M})\in \mathcal{U}_{s,k, \sigma}} \bigg( \prod_{i=1}^k j_i \bigg) \tilde{I}^\infty_{s,k,R,\delta, \sigma}(t,J,\tilde{M})(X_s), \\
    \tilde{I}^\infty_{s,k,R,\delta, \sigma}(t,J,\tilde{M})(X_s) &= \int_{\mathcal{M}_s^c(X_s)} \phi_s(V_s) \int_{\mathcal{T}_{k,\delta}(t)} S_s^{t - t_1} C^{\infty,R,j_1,m_1}_{s,s+\tilde{\sigma}_1} \cdots& \\
    &\cdots S^{t_{k-1}- t_k}_{s+{\tilde{\sigma}_{k-1}}} C^{\infty,R,j_k,m_k}_{s+\tilde{\sigma}_{k-1}, s+\tilde{\sigma}_k} S^{t_k}_{s+\tilde{\sigma}_k} f_0^{(s+\tilde{\sigma}_k)}(Z_s) \,dt_k \cdots dt_1 dV_s.
\end{align}

\subsection{Boltzmann hierarchy pseudo-trajectories.}\label{subsec::pseudo}
Fix $s \in \N$, $1 \leq k \leq n$, $\sigma \in S_k$, and $t\in [0,T]$. Fix $(t_1, \cdots, t_k) \in \mathcal{T}_k(t)$, $J = (j_1,\cdots, j_k)$, $\tilde{M} = (m_1, \cdots, m_k)$, and $(J,\tilde{M})\in \mathcal{U}_{s,k,\sigma}$. For all $i = 1, \cdots, k$ we define $(\bm{\omega}_{\sigma_i,i}, \bm{v}_{{\sigma_i,i}}) \in \mathbb{E}^{\sigma_i d - 1}_1 \times B_R^{d \sigma_i}$,
\begin{equation}
    (\bm{\omega}_{\sigma_i,i}, \bm{v}_{{\sigma_i,i}}) = ( \omega_{s+\tilde{\sigma}_{i-1}+1}, \cdots, \omega_{s+\tilde{\sigma}_i}, v_{s+\tilde{\sigma}_{i-1}+1}, \cdots, v_{s+\tilde{\sigma}_i}).
\end{equation}
At time $t = t_0$, assume we are given an initial configuration $Z_s = (X_s,V_s) \in \R^{2ds}$. The configuration will evolve under backwards free flow until time $t_1$, at which point the configuration $(\bm{\omega}_{\sigma_1,1}, \bm{v}_{{\sigma_1,1}})$ is added to the $m_1-$particle with the adjunction being precollisional if $j_1 = -1$ and postcollisional if $j_1 = 1$. The process continues until we obtain an $(s+ \tilde{\sigma}_k)-$configuration at time $t_{k+1} = 0$. The construction of the Boltzmann hierarchy pseudo-trajectory is concretely given by,
\begin{itemize}
    \item \textbf{Base Case:} We define, 
    \begin{equation}
        Z_s^\infty (t^-_0) = (x^\infty_1(t^-_0), \cdots, x^\infty_s(t_0^-), v^\infty_1(t^-_0), \cdots, v^\infty_{s}(t^-_0)) := Z_s.
    \end{equation}
    \item \textbf{Inductive Step:} Consider $t_i$, for $i\in \{1,\cdots, k \}$ and assume we are given
    \begin{equation}
        Z^\infty_{s+\tilde{\sigma}_{i-1}}(t^-_{i-1}) = (x^\infty_1(t^-_{i-1}), \cdots, x^\infty_{s+ \tilde{\sigma}_{i-1}}(t_{i-1}^-), v^\infty_1(t^-_{i-1}), \cdots, v^\infty_{s+ \tilde{\sigma}_{i-1}}(t^-_{i-1})).
    \end{equation}
    We define $Z^\infty_{s+\tilde{\sigma}_{i-1}}(t^+_{i-1}) = (x^\infty_1(t^+_{i-1}), \cdots, x^\infty_{s+ \tilde{\sigma}_{i-1}}(t_{i-1}^+), v^\infty_1(t^+_{i-1}), \cdots, v^\infty_{s+ \tilde{\sigma}_{i-1}}(t^+_{i-1}))$ as
    \begin{equation}
        Z^\infty_{s+\tilde{\sigma}_{i-1}}(t^+_{i-1}) := (X^\infty_{s+\tilde{\sigma}_{i-1}}(t^-_{i-1}) - (t_{i-1} - t_i) V^\infty_{s+\tilde{\sigma}_{i-1}}(t^-_{i-1}), V^\infty_{s+\tilde{\sigma}_{i-1}}(t^-_{i-1})),
    \end{equation}
    and define $Z^\infty_{s+\tilde{\sigma}_{i}}(t^-_{i}) = (x^\infty_1(t^-_{i}), \cdots, x^\infty_{s+\tilde{\sigma}_i}(t_{i}^-), v^\infty_1(t^-_{i}), \cdots, v^\infty_{s+\tilde{\sigma}_i}(t^-_{i}))$ as,
    \begin{equation}
    (x^\infty_j(t^-_{i}), v^\infty_{j}(t_i^-)) := (x^\infty_{j}(t^+_i), v^\infty_j(t^+_i) ), \quad \forall j\in \{1,\cdots, s+ \tilde{\sigma}_{i-1} \} \setminus \{m_i \},
    \end{equation}
    where for $j_i = -1$ we define the remaining particles,
    \begin{align}
        (x^\infty_{m_i}(t^-_{i}), v^\infty_{m_i}(t_i^-)) &:= (x^\infty_{m_i}(t^+_i), v^\infty_{m_i}(t^+_i) ), \\
        (x^\infty_{s+\tilde{\sigma}_{i-1}+\ell}(t^-_{i}), v^\infty_{s+\tilde{\sigma}_{i-1}+\ell} (t_i^-)) &:= (x^\infty_{m_i}(t^+_i), v_{s+\tilde{\sigma}_{i-1}+\ell} ), \quad \forall \ell \in \{1,\cdots, \sigma_i \},
    \end{align}
    and for $j_i = 1$ we define,
    \begin{align}
        (x^\infty_{m_i}(t^-_{i}), v^\infty_{m_i}(t_i^-)) &:= (x^\infty_{m_i}(t^+_i), v^{\infty,*(\sigma_i+1)}_{m_i}(t^+_i) ), \\
        (x^{\infty}_{s+\tilde{\sigma}_{i-1}+\ell}(t^-_{i}), v^\infty_{s+\tilde{\sigma}_{i-1}+\ell} (t_i^-)) &:= (x^\infty_{m_i}(t^+_i), v^{*(\sigma_i+1)}_{s+\tilde{\sigma}_{i-1}+\ell} ), \quad \forall \ell \in \{1,\cdots, \sigma_i \},
    \end{align}
    where,
    \begin{align}
        (v^{\infty,*(\sigma_i+1)}_{m_i}(t^+_i), v^{*(\sigma_i+1)}_{s+\tilde{\sigma}_{i-1}+1}, \cdots, v^{*(\sigma_i+1)}_{s+\tilde{\sigma}_{i}} ) &= T^{\sigma_i + 1}_{\bm{\omega}_{\sigma_i,i}}(v^{\infty}_{m_i}(t^+_i), v_{s+\tilde{\sigma}_{i-1}+1}, \cdots, v_{s+\tilde{\sigma}_{i}} ),
    \end{align}
    where we recall $T^{\sigma_i + 1}_{\bm{\omega}_{\sigma_i,i}}$ from Definition \ref{defn T omega}.
    \item \textbf{Final Step:} For $t_{k+1} = 0$, we obtain
    \begin{equation}
        Z^\infty_{s+\tilde{\sigma}_k}(0^+) = Z^\infty_{s+\tilde{\sigma}_k}(t_{k+1}^+) = (X^\infty_{s+\tilde{\sigma}_k}(t_k^-) - t_k V^\infty_{s+\tilde{\sigma}_k}(t_k^-), V^\infty_{s+\tilde{\sigma}_k}(t_k^-) ).
    \end{equation}

\end{itemize}
The process is illustrated in the following diagram (to be read from right to left):
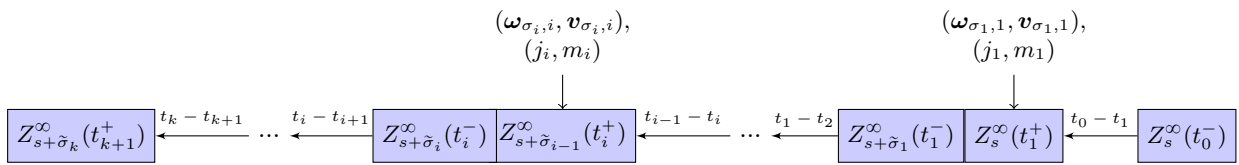
\begin{center}
\begin{tikzpicture}[node distance=2.5cm,auto,>=latex']\label{boltzmann pseudo diagram}
\node[int](0-){\small$ Z_s^\infty(t_0^-)$};
\node[int,pin={[init]above:\small$\begin{matrix}(\bm{\omega}_{\sigma_1,1},\bm{v}_{\sigma_1,1}),\\(j_1,m_1)\end{matrix}$}](1+)[left of=0-,node distance=2.3cm]{\small$Z_s^\infty(t_1^+)$};
\node[int](1-)[left of=1+,node distance=1.5cm]{$Z_{s+\widetilde{\sigma}_1}^\infty(t_1^-)$};
\node[](intermediate1)[left of=1-,node distance=2cm]{...};
\node[int,pin={[init]above:\small$\begin{matrix}(\bm{\omega}_{\sigma_i,i},\bm{v}_{\sigma_i,i}),\\(j_i,m_i)\end{matrix}$}](i+)[left of=intermediate1,node distance=2.5cm]{\small$Z_{s+\widetilde{\sigma}_{i-1}}^\infty(t_i^+)$};
\node[int](i-)[left of=i+,node distance=1.7cm]{\small$Z_{s+\widetilde{\sigma}_i}^\infty(t_i^-)$};
\node[](intermediate2)[left of=i-,node distance=2.2cm]{...};
\node[int](end)[left of=intermediate2,node distance=2.5cm]{\small$Z_{s+\widetilde{\sigma}_k}^\infty(t_{k+1}^+)$};

\path[<-] (1+) edge node {\tiny$t_{0}-t_1$} (0-);
\path[<-] (intermediate1) edge node {\tiny$t_{1}-t_2$} (1-);
\path[<-] (i+) edge node {\tiny$t_{i-1}-t_i$} (intermediate1);
\path[<-] (intermediate2) edge node {\tiny$t_{i}-t_{i+1}$} (i-);
\path[<-] (end) edge node {\tiny$t_{k}-t_{k+1}$} (intermediate2);
\end{tikzpicture}
\end{center}
\begin{definition} \label{boltz pseudo}
    Let $s\in \N$, $Z_s = (X_s,V_s) \in \R^{2ds}$, $(t_1,\cdots, t_k)\in \mathcal{T}_k(t)$, and $J = (j_1,\cdots, j_k)$, $M = (m_1, \cdots, m_k)$, such that $(J,M) \in \mathcal{U}_{s,k}$. For every $i = 1,\cdots, k$ and $\sigma \in S_k$, let $(\bm{\omega}_{\sigma_i,i}, \bm{v}_{\sigma_i,i})\in \mathbb{S}^{\sigma_i d -1}_1 \times B^{\sigma_i d}_R$. We call the sequence constructed above $\{ Z^\infty_{s+ \tilde{\sigma}_{i}}(t_i^+) \}_{i = 0, \cdots, k+1}$ the Boltzmann hierarchy pseudo-trajectory of $Z_s$.
\end{definition}

\subsection{Reduction to truncated elementary observables.}

Fix $m \in \N$, $t \in [0,T]$, $1 \leq k \leq n$, and $\sigma \in S_k$. Recall the set of separated collision times $\mathcal{T}_{k,\delta}(t)$ from \eqref{separated collision times} and the intersection of good sets $G_m(\epsilon_{M+1},\epsilon_0,\delta)$ from \eqref{good set intersection}. Let $X_s \in \Delta_s^X(\epsilon_0)$, $(J,\tilde{M}) \in \mathcal{U}_{s,k,\sigma}$, $(t_1,\cdots, t_k) \in \mathcal{T}_{k,\delta}$, and consider the Boltzmann hierarchy pseudo-trajectory of $Z_s$, denoted as $\{ Z^\infty_{s+ \tilde{\sigma}_{i}}(t_i^+) \}_{i = 0, \cdots, k+1}$. 
\par
Fix $\mathcal{M}_s(X_s)$ as in Lemma \ref{velocity subset lemma}, so that for any $V_s \in \mathcal{M}^c_s(X_s)$, we have $Z_s \in G_s(\epsilon_{M+1}, \epsilon_0, \delta)$. By the fact that $t_0 - t_1 > \delta$, we have $Z_s^\infty(t_1^+) \in G_s(\epsilon_0, 0)$. We now apply Proposition \ref{bad set triple} to obtain a set $\mathcal{B}_{k}^{\sigma_1}(Z_s^\infty(t_1^+)) \subseteq (\mathbb{E}^{\sigma_1 d - 1}_1 \times B_R^{\sigma_1 d} )^+(v_{m_1}^\infty(t_1^+))$ such that the following holds true for the $\sigma_1$-adjunction:
\begin{align}
    Z^\infty_{s+\tilde{\sigma}_1}(t_2^+) \in G_{s+\tilde{\sigma}_1}(\epsilon_0, 0), \quad \forall (\bm{\omega}_{\sigma_1, 1}, \bm{v}_{\sigma_1,1})\in [\mathcal{B}_{m_1}^{\sigma_1}(Z^\infty_s (t_1^+))]^c.
\end{align}
Iterating this process $k$-times, we obtain for $i \in \{1,\cdots, k \}$,
\begin{equation}
    Z^\infty_{s+\tilde{\sigma}_{i-1}}(t_i^+) \in G_{s+\tilde{\sigma}_{i-1}}(\epsilon_0, 0),
\end{equation}
and the sets 
\begin{equation}
    \mathcal{B}_{m_i}^{\sigma_i}(Z_s^\infty(t_i^+)) \subseteq (\mathbb{E}^{\sigma_i d - 1}_1 \times B_R^{\sigma_i d} )^+(v_{m_i}^\infty(t_i^+)),
\end{equation}
such that
\begin{equation}
    Z^\infty_{s+\tilde{\sigma}_i}(t_{i+1}^+) \in G_{s+\tilde{\sigma}_i}(\epsilon_0, 0), \quad \forall (\bm{\omega}_{\sigma_i, 1}, \bm{v}_{\sigma_i,i})\in [\mathcal{B}_{m_i}^{\sigma_i}(Z^\infty_s (t_i^+))]^c,
\end{equation}
with last step of this iteration yielding the result, $Z^\infty_{s+\tilde{\sigma}_{k}}(0^+) \in G_{s+\tilde{\sigma}_k}(\epsilon_0, 0)$.
\par
We now are able to define the truncated elementary observables by truncating the domains of the observables $\widetilde{I}^N_{s,k,R,\delta}, \widetilde{I}^\infty_{s,k,R,\delta}$, given in \eqref{bbgky obs}-\eqref{boltz obs}, by the sets $\{[\mathcal{B}_{m_i}^{\sigma_i}]^c\}_{i=1}^k$. The BBGKY hierarchy truncated observables are defined as:
\begin{align*}
    J^N_{s,k,R,\delta,\sigma}(t,J,M)(X_s) = \int_{\mathcal{M}_s^c(X_s)} \phi_s(V_s) &\int_{\mathcal{T}_{k,\delta}(t)} T_s^{t - t_1} \widetilde{C}^{N,R,j_1,m_1}_{s,s+\tilde{\sigma}_1} T^{t_1 - t_2}_{s+\tilde{\sigma}_1} \cdots \\
    &\cdots \widetilde{C}^{N,R,j_k,m_k}_{s+\tilde{\sigma}_{k-1},s+\tilde{\sigma}_k} T^{t_k} f_{N,0}^{(s+ \tilde{\sigma}_k)}(Z_s) dt_k, \cdots, dt_1 dV_s,
\end{align*}
where for every $i \in \{1, \cdots, k \}$ we define
\begin{equation}
\widetilde{C}^{N,R,j_i,m_i}_{s+\tilde{\sigma}_{i-1},s+\tilde{\sigma}_i} g_{N, s+ \tilde{\sigma}_i} = C^{N,R,j_i,m_i}_{s+\tilde{\sigma}_{i-1},s+\tilde{\sigma}_i} \bigg( g_{N, s+\tilde{\sigma}_i} \ind_{(\bm{\omega}_{\sigma_i,i}, \bm{v}_{\sigma_i,i}) \in [\mathcal{B}_{m_i}^{\sigma_i}(Z^\infty_{s+\tilde{\sigma}_{i-1}}(t_i^+))]^c }  \bigg).
\end{equation}
We similarly define the Boltzmann hierarchy truncated elementary observables:

\begin{align}\label{boltz trunc elem}
    J^\infty_{s,k,R,\delta,\sigma}(t,J,M)(X_s) = \int_{\mathcal{M}_s^c(X_s)} \phi_s(V_s) &\int_{\mathcal{T}_{k,\delta}(t)} S_s^{t - t_1} \widetilde{C}^{\infty,R,j_1,m_1}_{s,s+\tilde{\sigma}_1} S^{t_1 - t_2}_{s+\tilde{\sigma}_1} \cdots \\
    &\cdots \widetilde{C}^{\infty,R,j_k,m_k}_{s+\tilde{\sigma}_{k-1},s+\tilde{\sigma}_k} S^{t_k} f_0^{(s+ \tilde{\sigma}_k)}(Z_s) dt_k, \cdots, dt_1 dV_s,
\end{align}
where for every $i \in \{1, \cdots, k \}$ we define
\begin{equation}
\widetilde{C}^{\infty,R,j_i,m_i}_{s+\tilde{\sigma}_{i-1},s+\tilde{\sigma}_i} g_{ s+ \tilde{\sigma}_i} = C^{\infty,R,j_i,m_i}_{s+\tilde{\sigma}_{i-1},s+\tilde{\sigma}_i} \bigg( g_{ s+\tilde{\sigma}_i} \ind_{(\bm{\omega}_{\sigma_i,i}, \bm{v}_{\sigma_i,i}) \in [\mathcal{B}_{m_i}^{\sigma_i}(Z^\infty_{s+\tilde{\sigma}_{i-1}}(t_i^+))]^c }  \bigg).
\end{equation}
\begin{proposition} \label{boltz trunc obs est}
    Let $s,n \in \N$, $t \in [0,T]$, and fix parameters $\alpha, \epsilon_0, R, \eta, \delta$ as in \eqref{choice of parameters} and $(N, \epsilon_2, \cdots, \epsilon_{M+1})$ according to the scaling \eqref{scaling} with $\epsilon_2 << \cdots << \epsilon_{M+1} << \alpha$. Then, the following estimates hold:
    \begin{align}
        \sum_{k=1}^n \sum_{\sigma \in S_k} \sum_{(J,\tilde{M})\in \mathcal{U}_{s,k,\sigma}} &||\widetilde{I}^N_{s,k,R,\delta, \sigma}(t,J,\tilde{M}) - J^N_{s,k,R,\delta,\sigma}(t,J,\tilde{M}) ||_{L^\infty(\Delta_s^X(\epsilon_0))} \leq \\
        &\leq C_{d,s,\mu_0,T}^n ||\phi_s||_{L^\infty_{V_s}}  R^{d(s+(M+1)n)} \eta^{\frac{d-1}{2Md + 2}}  ||F_{N,0}||_{N,\beta_0,\mu_0},\\
        \sum_{k=1}^n \sum_{\sigma \in S_k} \sum_{(J,\tilde{M})\in \mathcal{U}_{s,k,\sigma}} &||\widetilde{I}^\infty_{s,k,R,\delta, \sigma}(t,J,\tilde{M}) - J^\infty_{s,k,R,\delta,\sigma}(t,J,\tilde{M}) ||_{L^\infty(\Delta_s^X(\epsilon_0))} \leq\\
        &\leq C_{d,s,\mu_0,T}^n ||\phi_s||_{L^\infty_{V_s}} R^{d(s+(M+1)n)} \eta^{\frac{d-1}{2Md + 2}} ||F_{0}||_{\infty,\beta_0,\mu_0}.
    \end{align}
\end{proposition}

\begin{proof}
This proof follows in an analogous way as the proof of Proposition 10.5 in \cite{AmpatzoglouPavlovic2020}.
\end{proof}

\section{Convergence proof}\label{sec::convergence}

\subsection{BBGKY hierarchy pseudo-trajectories and proximity to the Boltzmann hierarchy pseudo-trajectories}
Along the same lines of Subsection \ref{subsec::pseudo}, we will define the BBGKY hierarchy pseudo-trajectory. Let $s,k\in \N$, $(N, \epsilon_2,\cdots, \epsilon_{M+1})$ be in the scaling \eqref{scaling}, $t \in [0,T]$. Recall $\mathcal{T}_k(t)$ from \eqref{collision times} where we use the convention, $t_0 = t$ and $t_{k+1} = 0$. Fix $(t_1,\cdots, t_k) \in \mathcal{T}_k$, $\sigma \in S_k$, $J = (j_1, \cdots, j_k)$, $\tilde{M} = (m_1,\cdots, m_k)$, $(J,\tilde{M}) \in \mathcal{U}_{s,k,\sigma}$, and for each $i = 1,\cdots, k$, $(\bm{\omega}_{\sigma_i,i}, \bm{v}_{\sigma_i,i}) \in \mathbb{E}^{d \sigma_i - 1}_1 \times B^{d \sigma_i}_R$.
\par
We construct the BBGKY hierarchy pseudo-trajectory following the same steps as was done for the Boltzmann case with the one difference that we now take into account the interaction zone $\{\epsilon_{\ell+1} \}_{\ell = 1}^M$ for the adjusted particles. For initial configuration $Z_s = (X_s, V_s) \in \R^{2ds}$, the construction is given as follows:

\begin{itemize}
    \item \textbf{Base Case:} We define, 
    \begin{equation}
        Z_s^N (t^-_0) = (x^N_1(t^-_0), \cdots, x^N_s(t_0^-), v^N_1(t^-_0), \cdots, v^N_{s}(t^-_0)) := Z_s.
    \end{equation}
    \item \textbf{Inductive Step:} Consider $t_i$, for $i\in \{1,\cdots, k \}$ and assume we are given
    \begin{equation}
        Z^N_{s+\tilde{\sigma}_{i-1}}(t^-_{i-1}) = (x^N_1(t^-_{i-1}), \cdots, x^N_{s+ \tilde{\sigma}_{i-1}}(t_{i-1}^-), v^N_1(t^-_{i-1}), \cdots, v^N_{s+ \tilde{\sigma}_{i-1}}(t^-_{i-1})).
    \end{equation}
    We define $Z^N_{s+\tilde{\sigma}_{i-1}}(t^+_{i-1}) = (x^N_1(t^+_{i-1}), \cdots, x^N_{s+ \tilde{\sigma}_{i-1}}(t_{i-1}^+), v^N_1(t^+_{i-1}), \cdots, v^N_{s+ \tilde{\sigma}_{i-1}}(t^+_{i-1}))$ as
    \begin{equation}
        Z^N_{s+\tilde{\sigma}_{i-1}}(t^+_{i-1}) := (X^N_{s+\tilde{\sigma}_{i-1}}(t^-_{i-1}) - (t_{i-1} - t_i) V^N_{s+\tilde{\sigma}_{i-1}}(t^-_{i-1}), V^N_{s+\tilde{\sigma}_{i-1}}(t^-_{i-1})),
    \end{equation}
    and define $Z^N_{s+\tilde{\sigma}_{i}}(t^-_{i}) = (x^N_1(t^-_{i}), \cdots, x^N_{s+\tilde{\sigma}_i}(t_{i}^-), v^N_1(t^-_{i}), \cdots, v^N_{s+\tilde{\sigma}_i}(t^-_{i}))$ as,
    \begin{equation}
    (x^N_j(t^-_{i}), v^N_{j}(t_i^-)) := (x^N_{j}(t^+_i), v^N_j(t^+_i) ), \quad \forall j\in \{1,\cdots, s+ \tilde{\sigma}_{i-1} \} \setminus \{m_i \},
    \end{equation}
    where for $j_i = -1$ we define the remaining particles,
    \begin{align}
        (x^N_{m_i}(t^-_{i}), v^N_{m_i}(t_i^-)) &:= (x^N_{m_i}(t^+_i), v^N_{m_i}(t^+_i) ), \\
        (x^N_{s+\tilde{\sigma}_{i-1}+\ell}(t^-_{i}), v^N_{s+\tilde{\sigma}_{i-1}+\ell} (t_i^-)) &:= (x^N_{m_i}(t^+_i) - \epsilon_{\sigma_i +1} \omega_{s+ \tilde{\sigma}_{i-1}+\ell}, v_{s+\tilde{\sigma}_{i-1}+\ell} ), \quad \forall \ell \in \{1,\cdots, \sigma_i \},
    \end{align}
    and for $j_i = 1$ we define,
    \begin{align}
        (x^N_{m_i}(t^-_{i}), v^N_{m_i}(t_i^-)) &:= (x^N_{m_i}(t^+_i), v^{N,*(\sigma_i+1)}_{m_i}(t^+_i) ), \\
        (x^{N}_{s+\tilde{\sigma}_{i-1}+\ell}(t^-_{i}), v^N_{s+\tilde{\sigma}_{i-1}+\ell} (t_i^-)) &:= (x^N_{m_i}(t^+_i) +\epsilon_{\sigma_i +1} \omega_{s+ \tilde{\sigma}_{i-1}+\ell}, v^{*(\sigma_i+1)}_{s+\tilde{\sigma}_{i-1}+\ell} ), \quad \forall \ell \in \{1,\cdots, \sigma_i \},
    \end{align}
    where,
    \begin{align}
        (v^{N,*(\sigma_i+1)}_{m_i}(t^+_i), v^{*(\sigma_i+1)}_{s+\tilde{\sigma}_{i-1}+1}, \cdots, v^{*(\sigma_i+1)}_{s+\tilde{\sigma}_{i}} ) &= T^{\sigma_i + 1}_{\bm{\omega}_{\sigma_i,i}}(v^{N}_{m_i}(t^+_i), v_{s+\tilde{\sigma}_{i-1}+1}, \cdots, v_{s+\tilde{\sigma}_{i}} ),
    \end{align}
    where we recall $T^{\sigma_i + 1}_{\bm{\omega}_{\sigma_i,i}}$ from Definition \ref{defn T omega}.
    \item \textbf{Final Step:} For $t_{k+1} = 0$, we obtain
    \begin{equation}
        Z^N_{s+\tilde{\sigma}_k}(0^+) = Z^N_{s+\tilde{\sigma}_k}(t_{k+1}^+) = (X^N_{s+\tilde{\sigma}_k}(t_k^-) - t_k V^N_{s+\tilde{\sigma}_k}(t_k^-), V^N_{s+\tilde{\sigma}_k}(t_k^-) ).
    \end{equation}

\end{itemize}

\begin{definition}
    Let $s\in \N$, $Z_s = (X_s,V_s) \in \R^{2ds}$, $(t_1,\cdots, t_k)\in \mathcal{T}_k(t)$, and $J = (j_1,\cdots, j_k)$, $\tilde{M} = (m_1, \cdots, m_k)$, such that $(J,\tilde{M}) \in \mathcal{U}_{s,k}$. For every $i = 1,\cdots, k$ and $\sigma \in S_k$, let $(\bm{\omega}_{\sigma_i,i}, \bm{v}_{\sigma_i,i})\in \mathbb{S}^{\sigma_i d -1}_1 \times B^{\sigma_i d}_R$. We call the sequence constructed above $\{ Z^N_{s+ \tilde{\sigma}_{i}}(t_i^+) \}_{i = 0, \cdots, k+1}$ the BBGKY hierarchy pseudo-trajectory of $Z_s$.
\end{definition}
We now state a proximity result of the BBGKY and Boltzmann hierarchy pseudo-trajectories.
\begin{lemma}\label{pseudo proximity}
    Let $s \in \N$, $Z_s = (X_s,V_s) \in \R^{2ds}$, $1\leq k \leq n$, $\sigma \in S_k$, $(J,M) \in \mathcal{U}_{s,k,\sigma}$, $t \in [0,T]$, and $(t_1, \cdots, t_k) \in \mathcal{T}_k(t)$. For each $i = 1,\cdots, k$, consider $(\bm{\omega}_{\sigma_i,i}, \bm{v}_{\sigma_i,i}) \in \mathbb{E}^{d \sigma_i - 1}_1 \times B^{d \sigma_i}_R$. For all $i = 1, \cdots, k$ and $p = 1, \cdots, s + \tilde{\sigma}_{i-1}$, we have
    \begin{equation}\label{x N x infty diff}
        |x^N_{p}(t_i^+) - x_p^\infty(t_i^+)| \leq \epsilon_{M+1} (i-1), \quad v^N_p(t_i^+) = v^\infty_p(t_i^+).
    \end{equation}
    Furthermore, if $s<n$, then for all $i = 1, \cdots, k$ we have,
    \begin{equation}\label{total proximity}
        |X^N_{x+\tilde{\sigma}_{i-1}}(t_i^+) - X^\infty_{x+\tilde{\sigma}_{i-1}}(t_i^+)| \leq \sqrt{M} n^{3/2} \epsilon_{M+1}.
    \end{equation}
\end{lemma}
\begin{proof}
The proof follows as a straight-forward generalization of the corresponding proof in \cite{AmpatzoglouThesis}.
\end{proof}

\subsection{Reformulation in terms of pseudo-trajectories}
We will now write the truncated elementary observables for the BBGKY and Boltzmann hierarchies in terms of pseudo-trajectories. 
\par 
Fix $s \in N$, assume $s < n$, $1\leq k \leq n$, $\sigma \in S_k$, $(J,\tilde{M}) \in \mathcal{U}_{s,k,\sigma}$, $t \in [0,T]$, and $(t_1,\cdots, t_k) \in \mathcal{T}_{k,\delta}(t)$. For $X_s \in \Delta_s^X(\epsilon_0)$, by noticing that in the Boltzmann hierarchy there is always free flow between collision times, we can write the Boltzmann hierarchy truncated elementary observable, originally given in \eqref{boltz trunc elem}, in terms of the pseudo-trajectory $\{ Z^\infty_{s+ \tilde{\sigma}_{i}}(t_i^+) \}_{i = 0, \cdots, k+1}$,
\begin{align}
    J^\infty_{s,k,R,\delta,\sigma}(t,J,\tilde{M})(X_s) &= \prod_{i=1}^k \frac{1}{\sigma_i!} \int_{\mathcal{M}_s^c(X_s)} \phi_s(V_s) \int_{\mathcal{T}_{k,\delta}(t)} \int_{\big(\mathcal{B}_{m_1}^{\sigma_1}(Z_s^\infty(t_1^+)) \big)^c} \cdots \int_{\big(\mathcal{B}_{m_k}^{\sigma_1}(Z_s^\infty(t_k^+)) \big)^c} \\
    & \times \prod_{i = 1}^k b_{\sigma_i + 1}^+(\bm{\omega}_{\sigma_i,i}, \bm{v}_{\sigma_i,i}; v_{m_i}^\infty(t_i^+)) f_0^{(s+ \tilde{\sigma}_k)}(Z^\infty_{s+\tilde{\sigma}_k}(0^+)) \prod_{i=1}^k (d\bm{\omega}_{\sigma_i,i}, d\bm{v}_{\sigma_i,i}) dt_k, \cdots, dt_1 dV_s,
\end{align}
where we denote,
\begin{equation}
    b_{\sigma_i + 1}^+(\bm{\omega}_{\sigma_i,i}, \bm{v}_{\sigma_i,i}; v_{m_i}^\infty) =b_{\sigma_i + 1}^+(\bm{\omega}_{\sigma_i,i}, v_{s+ \tilde{\sigma}_{i-1}+1}-v_{m_i}^\infty, \cdots, v_{s+ \tilde{\sigma}_{i}}-v_{m_i}^\infty ).
\end{equation}
By applying Lemma \ref{pseudo proximity}, we will rewrite the BBGKY hierarchy truncated elementary observables in a similar way. This is done by showing that after the removal of the pathological sets given in Lemma \ref{velocity subset lemma} and Proposition \ref{bad set triple}, the backwards $(\epsilon_2, \cdots, \epsilon_{M+1})$-flow $\Psi_s$ given in \eqref{interaction flow} coincides with the free flow $\Phi_s$ given in \eqref{free flow}.  
\par
Consider $(N, \epsilon_2, \cdots, \epsilon_{M+1})$ in the scaling \eqref{scaling} such that $\epsilon_{M } << \eta^2 \epsilon_{M+1}$ and $\sqrt{M} n^{3/2} \epsilon_{M+1} << \alpha$, where we recall the parameters from \eqref{choice of parameters}. Given $V_s \in \mathcal{M}_s^c(X_s)$, we have $Z_s \in G_s(\epsilon_{M+1}, \epsilon_0, \delta)$ implying that for all $\tau \geq 0$ we have $Z_s(\tau) \in \mathring{\mathcal{D}}_{s,M}$. Therefore,
\begin{equation}
    \Psi_s^{\tau-t_0} Z_s^N(t_0^-) = \Phi_s^{\tau-t_0} Z_s^N(t_0^-), \quad \forall \tau \in [t_1, t_0].
\end{equation}
Additionally, by definition of the pseudo-trajectory $\{ Z^\infty_{s+ \tilde{\sigma}_{i}}(t_i^+) \}_{i = 0, \cdots, k+1}$ given in Definition \ref{boltz pseudo}, we have
\begin{equation}
    Z_s \in G_s(\epsilon_{M+1}, \epsilon_0,\delta) \implies Z_s^\infty(t_1^+) \in G_s(\epsilon_0,0).
\end{equation}
A repeated application of Proposition \ref{bad set triple} for all $i = 1, \cdots, k$ gives us
\begin{equation}
    Z_{s+\tilde{\sigma}_i}^\infty(t_{i+1}^+) \in G_{s+ \tilde{\sigma}_i}(\epsilon_0,0), \quad \forall (\bm{\omega}_{\sigma_i,i}, \bm{v}_{\sigma_i,i}) \in \big(\mathcal{B}_{m_i}^{\sigma_i}(Z_{s+\tilde{\sigma}_{i-1}}(t_i)^+) \big)^c.
\end{equation}
By Lemma \ref{pseudo proximity} along with $\sqrt{M} n^{3/2} \epsilon_{M+1} << \alpha$, and $s<n$, yields the estimate
\begin{equation}
    |X^N_{s+\tilde{\sigma}_{i-1}}(t_i^+) - X^\infty_{s+\tilde{\sigma}_{i-1}}(t_i^+)| \leq \frac{\alpha}{2}, \quad \forall \tau \in [t_{i+1}, t_i].
\end{equation}
Then, by applying Proposition \ref{bad set triple}, for all $i = 1,\cdots, k$, we have
\begin{equation}
    \Psi_{s+\tilde{\sigma}_i}^{\tau-t_i} Z_{s+\tilde{\sigma}_i}^N(t_i^-) = \Phi_{s+\tilde{\sigma}_i}^{\tau-t_i} Z_{s+\tilde{\sigma}_i}^N(t_i^-), \quad \forall \tau \in [t_{i+1}, t_{i}].
\end{equation}
We also note that Lemma \ref{pseudo proximity} also implies that
\begin{equation}
    v_{m_i}^N (t_i^+) = v_{m_i}^\infty(t_i^+), \quad \forall i = 1,\cdots, k.
\end{equation}
Therefore, we can write the BBGKY hierarchy truncated elementary observables as,
\begin{align}
    J^N_{s,k,R,\delta,\sigma}(t,J,M)(X_s) &= \bm{A^{s,k,\sigma}_{N,\epsilon_2, \cdots,\epsilon_{M+1}}} \int_{\mathcal{M}_s^c(X_s)} \phi_s(V_s) \int_{\mathcal{T}_{k,\delta}(t)} \int_{\big(\mathcal{B}_{m_1}^{\sigma_1}(Z_s^\infty(t_1^+)) \big)^c} \cdots \int_{\big(\mathcal{B}_{m_k}^{\sigma_1}(Z_s^\infty(t_k^+)) \big)^c} \\
    & \times \prod_{i = 1}^k b_{\sigma_i + 1}^+(\bm{\omega}_{\sigma_i,i}, \bm{v}_{\sigma_i,i}; v_{m_i}^\infty(t_i^+)) f_{N,0}^{(s+ \tilde{\sigma}_k)}(Z^N_{s+\tilde{\sigma}_k}(0^+)) \prod_{i=1}^k (d\bm{\omega}_{\sigma_i,i}, \bm{v}_{\sigma_i,i}) dt_k, \cdots, dt_1 dV_s,
\end{align}
where recalling \eqref{BBGKY heir}, we denote
\begin{equation}
    \bm{A^{s,k,\sigma}_{N,\epsilon_2, \cdots,\epsilon_{M+1}}} = \prod_{i = 1}^k  A^{\sigma_{i}}_{N, \epsilon_{\sigma_i + 1}, s+ \tilde{\sigma}_{i-1}}.
\end{equation}

To aid us in approximating the BBGKY hierarchy truncated observables by the Boltzmann one, we first will define the following auxiliary functionals,

\begin{align}
    \widehat{J}^N_{s,k,R,\delta,\sigma}(t,J,M)(X_s) &=  \prod_{i=1}^k \frac{1}{\sigma_i!}\int_{\mathcal{M}_s^c(X_s)} \phi_s(V_s) \int_{\mathcal{T}_{k,\delta}(t)} \int_{\big(\mathcal{B}_{m_1}^{\sigma_1}(Z_s^\infty(t_1^+)) \big)^c} \cdots \int_{\big(\mathcal{B}_{m_k}^{\sigma_1}(Z_s^\infty(t_k^+)) \big)^c} \\
    & \times \prod_{i = 1}^k b_{\sigma_i + 1}^+(\bm{\omega}_{\sigma_i,i}, \bm{v}_{\sigma_i,i}; v_{m_i}^\infty(t_i^+)) f_{0}^{(s+ \tilde{\sigma}_k)}(Z^N_{s+\tilde{\sigma}_k}(0^+)) \prod_{i=1}^k (d\bm{\omega}_{\sigma_i,i}, d\bm{v}_{\sigma_i,i}) dt_k, \cdots, dt_1 dV_s.
\end{align}
Then, the proofs of the following Propositions follow in the same way as \cite{ternary, Gallagher}, e.g. see Propositions 11.3 and 11.4 in \cite{ternary}.
\begin{proposition}\label{JN est}
    Let $s,n \in \N$, $s< n$, $t\in [0,T]$, and parameters $\alpha, \epsilon_0,R,\eta,\delta$ as in \eqref{choice of parameters}. Then, for any $\zeta > 0,$ there is $N_1 = N_1(\zeta,n,\alpha,\eta,\epsilon_0) \in \N$, such that for all $(N, \epsilon_2, \cdots, \epsilon_{M+1})$ in the scaling \eqref{scaling} with $N > N_1$, we have
    \begin{equation}
        \sum_{k=1}^n \sum_{\sigma \in S_k} \sum_{(J,\tilde{M})\in \mathcal{U}_{s,k,\sigma}} || {J}^N_{s,k,R,\delta,\sigma}(t,J,M) -\widehat{J}^N_{s,k,R,\delta,\sigma}(t,J,\tilde{M})  ||_{L^\infty \big( \Delta_s^X(\epsilon_0) \big)} \leq C^n_{d,s,\mu_0, T} ||\phi_s||_{L^\infty_{V_s}} R^{d(s+(M+1)n)} \zeta^2 .
    \end{equation}
    For tensorized initial data, we can derive the improved estimate
    \begin{equation}
        \sum_{k=1}^n \sum_{\sigma \in S_k} \sum_{(J,\tilde{M})\in \mathcal{U}_{s,k,\sigma}} || {J}^N_{s,k,R,\delta,\sigma}(t,J,\tilde{M}) -\widehat{J}^N_{s,k,R,\delta,\sigma}(t,J,\tilde{M})  ||_{L^\infty \big( \Delta_s^X(\epsilon_0) \big)} \leq C^n_{d,s,\mu_0, T}||\phi_s||_{L^\infty_{V_s}} R^{d(s+(M+1)n)} \epsilon_{M+1}^{1/2}.
    \end{equation}
\end{proposition}

\begin{proposition}\label{J infty est}
    Let $s,n \in \N$, $s< n$, $t\in [0,T]$, and parameters $\alpha, \epsilon_0,R,\eta,\delta$ as in \eqref{choice of parameters}. Then, for any $\zeta > 0,$ there is $N_2 = N_2(\zeta,n,\alpha,\eta,\epsilon_0) \in \N$, such that for all $(N, \epsilon_2, \cdots, \epsilon_{M+1})$ in the scaling \eqref{scaling} with $N > N_2$, we have
    \begin{equation}
        \sum_{k=1}^n \sum_{\sigma \in S_k} \sum_{(J,\tilde{M})\in \mathcal{U}_{s,k,\sigma}} || {J}^\infty_{s,k,R,\delta,\sigma}(t,J,\tilde{M}) -\widehat{J}^N_{s,k,R,\delta,\sigma}(t,J,\tilde{M})  ||_{L^\infty \big( \Delta_s^X(\epsilon_0) \big)} \leq C^n_{d,s,\mu_0, T}||\phi_s||_{L^\infty_{V_s}} R^{d(s+(M+1)n)} \zeta^2 
    \end{equation}
    For H\"older continuous $C^{0,\gamma}$, $\gamma \in (0,1]$ tensorized initial data, we can derive the improved estimate
    \begin{equation}
        \sum_{k=1}^n \sum_{\sigma \in S_k} \sum_{(J,\tilde{M})\in \mathcal{U}_{s,k,\sigma}} || {J}^\infty_{s,k,R,\delta,\sigma}(t,J,\tilde{M}) -\widehat{J}^N_{s,k,R,\delta,\sigma}(t,J,\tilde{M})  ||_{L^\infty \big( \Delta_s^X(\epsilon_0) \big)} \leq C^n_{d,s,\mu_0, T}||\phi_s||_{L^\infty_{V_s}} R^{d(s+(M+1)n)} \epsilon_{M+1}^{\gamma}.
    \end{equation}
\end{proposition}

\subsection{Proof of Theorem \ref{convergence}} 
 We are now ready to prove the main result of this paper, Theorem \ref{convergence}. Fix $\theta >0$, $\phi_s \in C_c(\R^{ds})$, $t \in [0,T]$, and $n,s \in \N$ such that $ s < n$. We fix the parameters $\alpha$ $\epsilon_0$, $R$, $\eta$, and $\delta$ according to \eqref{choice of parameters}, and choose $\zeta >0$ sufficiently small. By the triangle inequality when applied to the estimates given in Propositions \ref{reduction}, \ref{I M diff}, \ref{boltz trunc obs est}, \ref{JN est}, \ref{J infty est}, along with Remark \ref{obs k zero} and part $(i)$ of Definition \ref{approx bbgky hier}, there exists $N^*(\zeta) \in \N$ such that for all $N > N^*$, we have
\begin{equation}
    ||I^N_s(t) - I^\infty_s(t)||_{L^\infty(\Delta_s^X(\epsilon_0))} \leq C\big(2^{-(n+1)} + e^{-\frac{\beta_0}{3}R^2} + \delta C^n \big) + C^n R^{(M+2)dn} \eta^{\frac{d-1}{2M d + 2}} + C^n R^{(M+2)dn} \zeta^2,
\end{equation}
where
\begin{equation}
    C:= C_{d,s,\beta_0,\mu_0,T,M} ||\phi_s||_{L^\infty_{V_s}} \max\{1, ||F_0||_{\infty, \beta_0,\mu_0}\} > 1.
\end{equation}
\par
We now will choose our parameters satisfying \eqref{choice of parameters} and depending only on $\zeta$. We fix $n$ and choose the parameters $ \delta,\eta,R,\epsilon_0,$ and $\alpha$ such that,

\begin{equation}
    \begin{split}
    &\max \{s, \log_{2}(C \zeta^{-1}) \} << n, \quad \delta << \zeta C^{-(n+1)}, \quad \max \{ 1, \sqrt{3} \beta_0^{-1/2} \ln^{1/2}(C \zeta^{-1})  \} << R << C^{-\frac{1}{(M+2)d}} \zeta^{-\frac{1}{(M+2)dn}}, \\
    &\eta << \zeta^{\frac{4Md + 4}{d-1}}, \quad \epsilon_0 << \min \{\theta, \eta \delta \}, \quad \alpha << \epsilon_0 \min\{1, R^{-1} \eta \}.
    \end{split}
\end{equation}
Therefore, there exists $N^*(\zeta) \in \N$ such that for all $(N,\epsilon_2, \cdots, \epsilon_{M+1})$ in the scaling \eqref{scaling} with $N>N^*$, we have 
\begin{equation}
    ||I^N_s(t) - I^\infty_s(t)||_{L^\infty(\Delta^X_s(\theta))} \leq ||I^N_s(t) - I^\infty_s(t)||_{L^\infty(\Delta^X_s(\epsilon_0))} < \zeta,
\end{equation}
proving Theorem \ref{convergence}.
\par
\textbf{Proof of Corollary \ref{cor prop of chaos}}
By Theorem \ref{theorem propagation of chaos} we have 
$\bm{F} = (f^{\otimes s})_{s\in \N}$, where $f$ is the mild solution of the Boltzmann equation. Similar as the preceding proof, using the tensorized initial data analogues in Propositions \ref{JN est} and \ref{J infty est}, for large enough $N$ we have,

\begin{equation*}
    ||I_{\phi_s}f_N^{(s)}(t) - I_{\phi_s}f^{\otimes s}(t)||_{L^\infty(\Delta_s^X(\epsilon_0))} \leq C\big(2^{-(n+1)} + e^{-\frac{\beta_0}{3}R^2} + \delta C^n \big) + C^n R^{(M+2)dn} \eta^{\frac{d-1}{2M d + 2}} + C^n R^{(M+2)dn} \epsilon_{M+1}^{\gamma_*}.
\end{equation*}
where $\gamma_* = \min \{\frac{1}{2}, \gamma \} \in (0,\frac{1}{2}]$ and $\gamma$ is the H\"older continuous regularity of $f_0$. Let $0 < r < \gamma_*$.
\par 
 We fix $n$ and choose the parameters $ \delta,\eta,R,\epsilon_0,$ and $\alpha$ such that,
\begin{equation*}
    \begin{split}
    &\max \{s, \log_{2}(C \epsilon_{M+1}^{-\gamma_*}) \} << n, \quad \delta << \epsilon_{M+1}^{\gamma_*} C^{-(n+1)}, \quad \max \{ 1, \sqrt{3} \beta_0^{-1/2} \ln^{1/2}(C \epsilon_{M+1}^{-\gamma_*})  \} << R << C^{-\frac{1}{(M+2)d}} \epsilon_{M+1}^{ \frac{r-\gamma_*}{(M+1)dn}\gamma_*} , \\
    &\eta << \epsilon_{M+1}^{\frac{4Md + 4}{d-1}\gamma_*}, \quad \epsilon_0 << \min \{\theta, \eta \delta \}, \quad \alpha << \epsilon_0 \min\{1, R^{-1} \eta \}.
    \end{split}
\end{equation*}
Therefore, for large enough $N$ we have,
\begin{equation}
    ||I_{\phi_s}f_N^{(s)}(t) - I_{\phi_s}f^{\otimes s}(t)||_{L^\infty(\Delta_s^X(\theta))} \leq ||I_{\phi_s}f_N^{(s)}(t) - I_{\phi_s}f^{\otimes s}(t)||_{L^\infty(\Delta_s^X(\epsilon_0))} < \epsilon_{M+1}^r,
\end{equation}
proving Corollary \ref{cor prop of chaos}.
\appendix

\section{Auxiliary Results}
\begin{lemma} \label{app det}
    Let $n \in \N,$ $\lambda \neq 0,$ and $w,u \in \R^n$. Then, we have,
    \begin{equation}
        \det(\lambda I_n + wu^T) = \lambda^n + ( 1+ \lambda^{-n} \langle w, u \rangle). 
    \end{equation}
\end{lemma}

\begin{lemma}\label{indicatrix lemma}
    Let $n\in \N$, $\Psi: \R^n \rightarrow \R$ be a $C^1$ function and $\gamma \in \R$. Assume there is a $\delta > 0$ with $\nabla \Psi(\bm{\omega}) \neq 0$ for $\bm{\omega} \in [\gamma - \delta <  \Psi< \gamma + \delta]$. Let $\Omega \subset \R^n$ be a domain and consider a $C^1$ map $F: \Omega \rightarrow \R^n$ of non-zero Jacobian in $\Omega$. Then for any measurable $g: \R^n \rightarrow [0,+\infty]$ or $g: \R^n \rightarrow [-\infty, + \infty]$ integrable
    \begin{equation}
        \int_{[\Psi = \gamma]} g(\nu) \mathcal{N}_F(\nu, [\Psi \circ F = \gamma]) d\sigma(\nu) = \int_{[\Psi \circ F = \gamma]} (g\circ F)(\omega) |\jac{F(\omega)}| \frac{|\nabla \Psi(F(\omega)|}{|\nabla(\Psi \circ F)(\omega)|} d\sigma(\omega),
    \end{equation}
    where given $\nu\in \R^n$ and $A\subset \Omega$, $\mathcal{N}_F(\nu,A):= \card{(\{ \omega \in A \, : \, F(\omega) = \nu\})}$ is the Banach indicatrix of $A$.
\end{lemma}

\end{document}